\documentclass[a4paper]{amsart}
\usepackage[utf8]{inputenc}
\usepackage[T1]{fontenc}
\usepackage{lmodern, enumerate}
\usepackage{amssymb,amsxtra}
\usepackage[all]{xy}
\usepackage{nicefrac,mathtools}
\usepackage{microtype}
\usepackage{amscd}
\usepackage{xcolor}
\usepackage[pdftitle={Imprimitivity theorems for weakly proper actions of locally compact groups},
 pdfauthor={Alcides Buss, Siegfried Echterhoff},
 pdfsubject={Mathematics}
]{hyperref}
\usepackage[lite]{amsrefs}
\newcommand*{\MRref}[2]{ \href{http://www.ams.org/mathscinet-getitem?mr=#1}{MR #1}}
\newcommand*{\arxiv}[1]{\href{http://www.arxiv.org/abs/#1}{arXiv: #1}}

\numberwithin{equation}{section}
\theoremstyle{plain}
\newtheorem{theorem}[equation]{Theorem}
\newtheorem{lemma}[equation]{Lemma}
\newtheorem{proposition}[equation]{Proposition}
\newtheorem{corollary}[equation]{Corollary}
\theoremstyle{definition}
\newtheorem{definition}[equation]{Definition}
\theoremstyle{remark}
\newtheorem{remark}[equation]{Remark}
\newtheorem{example}[equation]{Example}
\newtheorem{question}[equation]{Question}

\DeclareMathOperator{\Aut}{Aut}
\DeclareMathOperator{\cspn}{\overline{span}}

\DeclareMathOperator{\Ind}{Ind}
\DeclareMathOperator{\Inf}{Inf}
\DeclareMathOperator{\supp}{\mathrm{supp}}
\DeclareMathOperator{\sat}{\mathrm{sat}}
\DeclareMathOperator{\pt}{\mathrm{pt}}

\newcommand*{\nb}{\nobreakdash}
\newcommand*{\Star}{\(^*\)\nobreakdash-}
\newcommand*{\dd}{\,d}
\newcommand*{\bs}{\backslash}

\newcommand*{\C}{\mathbb C}
\newcommand*{\EG}{\underline{EG}}
\newcommand*{\Y}{\mathcal Y}
\newcommand*{\Z}{\mathbb Z}

\newcommand*{\Lb}{\mathcal L}
\newcommand*{\K}{\mathcal K}
\newcommand*{\M}{\mathcal{M}} 

\newcommand*{\Cst}{C^*} 
\newcommand*{\cont}{C}
\newcommand*{\contz}{\cont_0}
\newcommand*{\contc}{\cont_c}
\newcommand*{\contb}{\cont_b}

\newcommand*{\id}{\textup{id}}
\newcommand*{\Ad}{\textup{Ad}}

\newcommand*{\Hils}{\mathcal H}
\newcommand*{\U}{\mathcal U}
\newcommand*{\D}{\mathcal D}
\newcommand*{\E}{\mathcal E}
\newcommand*{\EE}{\mathbb E}
\newcommand*{\En}{E}
\newcommand*{\Enorm}{{\pn}}

\newcommand*{\defeq}{\mathrel{\vcentcolon=}}
\newcommand*{\congto}{\xrightarrow\sim}

\newcommand*{\braket}[2]{\langle#1\!\mid\!#2\rangle}

\newcommand*{\bbraket}[2]{\mathopen{\langle\!\langle}#1\!\mid\!#2\mathclose{\rangle\!\rangle}}

\newcommand*{\sbe}{\subseteq} 
\newcommand*{\F}{\mathcal F}
\newcommand*{\cstar}{\texorpdfstring{$C^*$\nobreakdash-\hspace{0pt}}{*-}}
\newcommand*{\into}{\hookrightarrow}
\newcommand*{\onto}{\twoheadrightarrow}
\newcommand*{\red}{\mathrm{r}}
\newcommand*{\redg}{{\mathrm{r}^G}}
\newcommand*{\redh}{{\mathrm{r}^H}}

\newcommand*{\un}{\mathrm{u}}
\newcommand*{\ung}{{\mathrm{u}^G}}
\newcommand*{\unh}{{\mathrm{u}^H}}
\newcommand*{\pn}{\mathrm{\mu}}
\newcommand*{\png}{{\mathrm{\mu}^G}}
\newcommand*{\pnh}{{\mathrm{\mu}^H}}
\newcommand*{\qn}{\mathrm{\nu}}

\newcommand*{\Qn}{\mathcal{Q}}
\newcommand*{\Fix}{\mathrm{Fix}}
\newcommand*{\CP}{\mathrm{CP}}
\newcommand*{\I}{\mathcal{I}} 

\newcommand{\cspan}{\overline{\operatorname{span}}}
\newcommand{\lt}{\mathrm{lt}}
\newcommand{\rt}{\mathrm{rt}}

\newcommand*{\X}{\mathcal X} 

\newcommand*{\CCCorr}{\mathfrak{Corr}}
\newcommand*{\Corr}[2]{\mathfrak{Corr}({#1},{#2})}
\newcommand*{\CCorr}[1]{\mathfrak{Corr}({#1})}
\newcommand*{\CorrG}{\mathfrak{Corr}(G)}
\newcommand*{\CorrH}{\mathfrak{Corr}(H)}
\newcommand{\ie}{\emph{i.e.}}
\newcommand{\eg}{\emph{e.g.}}

\begin{document}
\title[Imprimitivity theorems for weakly proper actions]{Imprimitivity theorems for weakly proper actions of locally compact groups}

\author{Alcides Buss}
\email{alcides.buss@ufsc.br}
\address{Departamento de Matem\'atica\\
 Universidade Federal de Santa Catarina\\
 88.040-900 Florian\'opolis-SC\\
 Brazil}

\author{Siegfried Echterhoff}
\email{echters@uni-muenster.de}
\address{Mathematisches Institut\\
University of M\"un\-ster\\
 Einsteinstr.\ 62\\
 48149 M\"unster\\
 Germany}

\begin{abstract}
In a recent paper the authors introduced universal and exotic generalized fixed-point algebras
for weakly proper group actions on \cstar{}algebras. Here we extend the notion of weakly proper actions
to actions on Hilbert modules. As a result we obtain several imprimitivity theorems
establishing important Morita equivalences between universal, reduced, or exotic
crossed products and appropriate universal, reduced, or exotic
fixed-point algebras, respectively. In particular, we obtain an exotic version of Green's imprimitivity theorem and
a very general version of the symmetric imprimitivity theorem by weakly proper actions of product groups $G\times H$.
In addition, we study functorial properties of generalized fixed-point algebras for equivariant categories of
\cstar{}algebras based on correspondences.
\end{abstract}


\thanks{Supported by Deutsche Forschungsgemeinschaft  (SFB 878, Groups, Geometry \& Actions) and by CNPq (Ciências sem Fronteira) -- Brazil.}
\thanks{The authors are grateful for various useful comments by the referee which helped to improve the results of the paper.}

\maketitle

\section{Introduction}
\label{sec:introduction}
Suppose that $\alpha:G\to \Aut(A)$ is an action of a locally compact group $G$ on the \cstar{}algebra $A$
and let $X$ be a proper $G$\nb-space. Then we say that $A$ is a weakly proper $X\rtimes G$-algebra
if there exists a $G$\nb-equivariant nondegenerate \Star{}homomorphism $\phi:C_0(X)\to \M(A)$.
This notion substantially generalizes the notion of (centrally) proper $C_0(X)$\nb-algebras, in which the
homomorphism $\phi:C_0(X)\to \M(A)$ is assumed to take values in the center $Z\M(A)$
of the multiplier algebra $\M(A)$. On the other hand, it follows from \cite[Proposition 5.9]{Rieffel:Integrable_proper} that
weakly proper actions are always proper in the  sense of Rieffel (see \cite{Rieffel:Proper, Rieffel:Integrable_proper}), who
 showed in \cite{Rieffel:Proper} that all proper actions in his sense
allow the construction of a generalized fixed-point algebra $A^G_\red$ which is Morita equivalent
to an ideal in the reduced crossed product $A\rtimes_{\alpha,\red}G$. But already in his paper
 \cite{Rieffel:Proper}, Rieffel discussed the question, whether it is possible to obtain similar constructions
 which involve the universal crossed product $A\rtimes_{\alpha ,\un}G$.
We show in \cite{Buss-Echterhoff:Exotic_GFPA} that such theory exists in the case of weakly proper actions. In that paper we also construct a universal version of the generalized
fixed-point algebras $A_{\un}^G$, which is
 Morita equivalent to an ideal in the universal crossed product
$A\rtimes_{\alpha ,\un}G$. It can be obtained as a completion of the
 {\em fixed-point algebra with compact supports}
 $$A_c^G=\contc(G\bs X)\cdot\{m\in \M(A)^G: f\cdot m, m\cdot f\in A_c\; \mbox{ for all } f\in \contc(X)\}\cdot \contc(G\bs X),$$
 with $f\cdot a:= \phi(f)a$ for $f\in C_b(X)$, $a\in \M(A)$, and $A_c:=\contc(X)\cdot A\cdot \contc(X)$.
 As an application, we obtained Landstad duality for maximal coactions and we answered
 a number of questions about duality properties of exotic crossed products which were raised in a recent paper by
Kaliszewski, Landstad and Quigg in \cite{Kaliszewski-Landstad-Quigg:Exotic}.

In this paper we further develop the theory by studying weakly proper actions on
Hilbert modules which are inspired by Meyer's theory of square-integrable Hilbert modules
as studied in \cite{Meyer:Generalized_Fixed}. To be more precise, if $(B,\beta)$ is a $G$\nb-algebra and
$(\E,\gamma)$ is a $G$\nb-equivariant Hilbert $B$-module, we say that $(\E,\gamma)$
is a weakly proper $(B, X\rtimes G)$-module if there exists a proper $G$\nb-space $X$ and
a $G$\nb-equivariant nondegenerate representation $\phi:C_0(X)\to \Lb(\E)$.
Of course, this implies that $(\K(\E), \Ad\gamma)$ is a weakly proper $X\rtimes G$-algebra
as above. Let $\F_c(\E):=\phi(\contc(X))\E\subseteq \E$. We show that $\F_c(\E)$ completes to a Hilbert $B\rtimes_{\beta,\un}G$-module $\F_{\un}(\E)$
with respect to the $\contc(G,B)$-valued inner product
$$\bbraket{\xi}{\eta}_{\contc(G,B)}(s)=\Delta(s)^{-1/2}\braket{\xi}{\gamma_s(\eta)}_B.$$
The corresponding universal generalized fixed-point algebra $\Fix_{\un}^G(\E)$ is defined as
$$\Fix_{\un}^G(\E):=\K(\F_{\un}(\E)).$$
It turns out that  $\Fix_{\un}^G(\E)$ is canonically isomorphic to the universal generalized fixed-point algebra $\K(\E)_{\un}^G$ (see Proposition~\ref{prop-exotic-compacts}).
It follows directly from the definition that $\F_{\un}(\E)$ implements a Morita equivalence
between $\Fix_{\un}^G(\E)$ and the ideal
$$\mathcal I_B:=\cspan \bbraket{\F_{\un}(\E)}{\F_{\un}(\E)}_{B\rtimes_{\beta,{\un}} G}
\subseteq B\rtimes_{\beta, {\un}} G.$$
We say that the action of $G$ on $\E$ is {\em universally saturated} if $\mathcal I_B=B\rtimes_{\beta,{\un}}G$.

We shall see that this construction has many interesting applications:
consider the Hilbert $B$-module $L^2(G)\otimes B\cong L^2(G,B)$ for the $G$-algebra $(B,\beta)$.
Let $\rho:G\to \U(L^2(G))$ denote the right regular representation of $G$ and let
$M:C_0(G)\to \Lb(L^2(G))$ denote the representation by multiplication operators.
Then  $L^2(G)\otimes B$ becomes a weakly proper Hilbert $(B, G\rtimes G)$-module
with $G$-action $\rho\otimes\beta$ and structure map
$M\otimes 1_B\colon C_0(G)\to \Lb(L^2(G)\otimes B)$.
If we restrict the action to a closed subgroup $H$ of $G$, $L^2(G,B)$ becomes a weakly proper
$(B, G\rtimes H)$-module, and it turns out that $\F_{\un}^H(L^2(G,B))$ is isomorphic to
Green's $C_0(G/H, B)\rtimes_{\tau\otimes \beta,{\un}}G- B\rtimes_{\beta,{\un}} H$ equivalence bimodule. Hence, we obtain Green's imprimitivity theorem as a special case of our constructions.
Indeed, if we start just with an action of $H$ on $B$, we obtain the version for induced algebras:
$$\Ind_H^G(B,\beta)\rtimes_{\Ind\beta,\un}G\sim_M B\rtimes_{\beta, {\un}} H.$$
The key result for this (Theorem~\ref{theo:TheFixedPointFunctorForCommutingActions}) is a general isomorphism
$(A\rtimes_{\alpha,{\un}} H)^G_{\un}\cong A^G_{\un}\rtimes_{\alpha, {\un}} H$
which works for any weak $X\rtimes (G\times H)$-algebra $(A,\alpha)$ such that $G$ acts properly
on $X$ (\ie, there exists a $G\times H$-equivariant structure map $\phi:\contc(X)\to \M(A)$, but we assume properness only for the action of $G$).
In this case, $A\rtimes_{\alpha,{\un}} H$ is a weakly proper $X\rtimes G$-algebra via the
composition of $\phi$ with the canonical embedding $i_A:A\to \M(A\rtimes_{\alpha, {\un}} G)$.

If both, $G$ and $H$ act properly on $X$, we obtain a very general version of the symmetric imprimitivity theorem:
there is a canonical partial $A_{\un}^G\rtimes_{\un} H-A_{\un}^H\rtimes_{\un} G$ equivalence bimodule which
implements a Morita equivalence if the actions of $G$ and $H$ on $A$ are both universally saturated,
which is always true if the actions of $G$ and $H$ on $X$ are free.
We show that the module can be realized as a completion of $A_c=\contc(X)\cdot A\cdot \contc(X)$
which can be made into a partial $\contc(H, A_c^G)-\contc(G, A_c^H)$ pre-equivalence bimodule
(a partial equivalence bimodule between two \cstar{}algebras $A$ and $B$ is an equivalence bimodule
between ideals $I_A$ and $I_B$ in $A$ and $B$, respectively).

Above we stated all results for the universal crossed-product norms. But all proofs are given also for
the reduced crossed products and, more generally, for certain compatible choices of exotic crossed-product norms for $G$, $H$ and $G\times H$ in the sense of Kaliszewski, Landstad and Quigg  (see \cite{Kaliszewski-Landstad-Quigg:Exotic}).
In particular, we shall see that Green's imprimitivity theorem holds for any given
exotic $G$-crossed product
corresponding to a $G$-invariant weak-* closed ideal $E\subseteq B(G)$ as constructed in \cite{Kaliszewski-Landstad-Quigg:Exotic}
 if we take the  $H$-crossed product with respect to the weak-* closed ideal $E_H\subseteq B(H)$ generated by $\{f|_H: f\in B(G)\}$.
In the reduced case we recover a number of results from the literature.
In particular, our symmetric imprimitivity theorem extends the one given
by an Huef, Raeburn and Williams in \cite{anHuef-Raeburn-Williams:Symmetric} for Rieffel proper actions.
However, we think that our approach is more direct and the description of the pre-equivalence
bimodule is much easier to handle. Of course, our result also covers the version of the symmetric
imprimitivity theorem for $C_0(X)$-algebras due to Kasparov (\cite{Kasparov:Novikov}) and
Raeburn (\cite{Raeburn:Induced_Symmetric}).

The paper is structured as follows: in a preliminary section we recall some background on crossed products
by actions on \cstar{}algebras and Hilbert modules with respect to universal, reduced and exotic crossed-product norms.
For the latter we restrict ourselves to norms which are attached to $G$\nb-invariant ideals of the Fourier-Stieltjes algebra $B(G)$, as studied in \cite{Kaliszewski-Landstad-Quigg:Exotic}. We also obtain some useful
functoriality properties
for such exotic crossed products which are needed in the sequel. In \S 3 we introduce
weakly proper Hilbert modules and the corresponding fixed-point algebras and we establish a number of
useful results about these objects. In \S 4 we study equivariant bimodules
$(\E,\gamma)$ between two weakly proper $X\rtimes G$-algebras $(A,\alpha)$ and $(B,\beta)$, and we show that
such modules always allow the construction of an $A^G_{\un}-B^G_{\un}$ fixed module $\E^G_{\un}$
(and similarly for other crossed-product norms), which will be a full equivalence bimodule if the actions of $G$ on $A$ and $B$ are both universally saturated.
A similar construction of fixed modules for weakly proper actions was done in the reduced case in \cite{Huef-Raeburn-Williams:FunctorialityGFPA} under the additional assumption that the action on $X$ is \emph{free}. Our construction does not use this assumption and allows functoriality of fixed-point algebras for \emph{saturated} weakly proper actions.
We also derive a useful decomposition result for such bimodules, which will play an important role
for the description of the direct module for the symmetric imprimitivity theorem, which, together with Green's
theorem, is derived in \S 5. Finally, in \S 6 we discuss some functorial properties of the construction
which assigns to a proper $X\rtimes G$-algebra its fixed-point algebra $A^G_\pn$ with respect to
exotic crossed-product norms $\|\cdot\|_\pn$ attached to $G$\nb-invariant subsets of $B(G)$.
This generalizes the main results  of the papers \cite{Huef-Kaliszewski-Raeburn-Williams:Naturality-Rieffel, Huef-Kaliszewski-Raeburn-Williams:Naturality-Symmetric} by an Huef, Kaliszewski, Raeburn  and Williams.

\section{Preliminaries}\label{sec:Preliminaries}

\subsection{Exotic crossed products} Let $(B,G, \beta)$ be a \cstar{}dynamical system. Then $\contc(G,B)$ becomes a \Star{}algebra with respect to the usual convolution and involution:
$$f*g(t)\defeq\int_G f(s)\beta_s(g(s^{-1}t))\dd{s}\quad\mbox{and}\quad f^*(t)\defeq \Delta(t)^{-1}\beta_t(f(t^{-1}))^*.$$
The full crossed product $B\rtimes_{\beta,\un} G$ is the completion of $\contc(G,B)$ with respect to the universal norm
$\|\cdot\|_\un$ obtained from integrated forms of covariant representations.
Let us write $(\iota_B,\iota_G)$ for the canonical covariant representation of $(B,G, \beta)$ into $\M(B\rtimes_{\beta}G)$, \ie, $\iota_B\colon B\to \M(B\rtimes_{\beta,\un} G)$ and $\iota_G\colon G\to\U\M(B\rtimes_{\beta,\un} G)$ are given by:
$$\big(\iota_B(b)\cdot f\big)(t)= bf(t),\quad f\cdot \iota_B(b)=f(t)\beta_t(b),$$
$$\big(\iota_G(s)\cdot f\big)(t)=\beta_s(f(s^{-1}t)),\quad \big(f\cdot\iota_G(s)\big(t)=\Delta(s)^{-1}f(ts^{-1}),$$
for $f\in \contc(G,B)$, $b\in B$ and $s\in G$. We refer to  \cite{Williams:Crossed} for
further details.

If $(B,G,\beta)$ is a C*-dynamical system, let $(i_B^\red, i_G^\red): (B,G)\to \M(B\otimes \K(L^2(G)))$ denote the covariant
pair given by $i_{G}^\red=1\otimes\lambda_G$ and $i_{B}^\red$ denotes the
composition
$$B\stackrel{\tilde\beta}{\longrightarrow} C_b(G,B)\subseteq \M(B\otimes C_0(G))\stackrel{\id_B\otimes M}{\longrightarrow}
\M(B\otimes \K(L^2(G)))$$
in which $\tilde\beta$ maps $b\in B$ to the function $t\mapsto \beta_{t^{-1}}(b)$ and
$M:C_0(G)\to \Lb(L^2(G))$ is the representation by multiplication operators.
Its integrated form $\Lambda_B:=i_B^\red\rtimes i_G^\red:B\rtimes_{\beta,\un}G\to \M(B\otimes \K(L^2(G)))$ is called the {\em regular representation} of $(B,G,\beta)$ and the {\em reduced crossed product} $B\rtimes_{\beta,\red}G$ is defined as the image
$\Lambda_B(B\rtimes_\beta G)\subseteq \M(B\otimes \K(L^2(G))$.
Of course, we may also
regard $B\rtimes_{\beta,r}G$ as the completion of $\contc(G,B)$ with respect to the {\em reduced norm}
$\|f\|_\red=\|\Lambda_B(f)\|$.

In this paper we shall also consider  {\em exotic} \cstar{}norms $\|\cdot\|_\pn$ on $\contc(G,B)$, \ie,
norms which satisfy the \cstar{}condition $\|f^** f\|_\pn=\|f\|_\pn^2$, so that the (Hausdorff) completion
$B\rtimes_{\beta,\pn}G$ of $\contc(G,B)$ with respect to this norm is a \cstar{}algebra. Such norms are most interesting if they
 satisfy
$\|f\|_r\leq \|f\|_\pn\leq \|f\|_u$ and if they are coming from a crossed-product functor $(B,G,\beta)\mapsto B\rtimes_{\beta,\pn}G$.
Such functors have been studied recently by several
authors (\eg, see \cite{Brown-Guentner:New_completions, Kaliszewski-Landstad-Quigg:Exotic,
Buss-Echterhoff:Exotic_GFPA,  Baum-Guentner-Willett}). Although some of our results could possibly be stated
for more general crossed-product norms, in this paper we shall work exclusively with
the crossed-product norms as introduced by Kaliszewski, Quigg, and Landstad in \cite{Kaliszewski-Landstad-Quigg:Exotic}.
As we shall see below, these norms have very good functorial properties with respect to equivariant morphisms
and correspondences.

To recall the construction, let $B(G)$ denote the Fourier-Stieltjes algebra of $G$, \ie, the set of all matrix coefficients
$s\mapsto \braket{\pi(s)\xi}{\eta}$ of all unitary representations $\pi$ of $G$.
Recall that $B(G)$ identifies with the Banach-space dual $\Cst(G)^*$ if we map the function $s\mapsto\braket{\pi(s)\xi}{\eta}$
 to the linear functional on $\Cst(G)$ given by $x\mapsto \braket{\pi(x)\xi}{\eta}$.
 Assume now that $\En\subseteq B(G)$ is a nonzero
  ideal in $B(G)$ which is $G$\nb-invariant with respect to left and right translation.
 It is shown in \cite[Lemma 3.1 and Lemma 3.14]{Kaliszewski-Landstad-Quigg:Exotic} that
 $I_\En:={^\perp \En}=\{x\in \Cst(G): f(x)=0\;\mbox{ for all } f\in \En\}$ is a closed ideal in $\Cst(G)$ which is contained in
 the kernel of the regular representation $\lambda_G:\Cst(G)\to C_\red^*(G)$, and hence the quotient
 $C_\En^*(G):=\Cst(G)/{I_\En}$ is a group \cstar{}algebra lying between $C_\red^*(G)$ and $\Cst(G)$.
Let $q_\En:\Cst(G)\to C_\En^*(G)$ denote the quotient map and let $u:G\to\U\M(C^*(G))$ denote
 the canonical homomorphism. If $\beta\colon G\to \Aut(B)$ is an action, let $\widehat\beta:=(i_B\otimes 1)\rtimes (i_G\otimes u)\colon B\rtimes_{\beta,\un} G\to \M(B\rtimes_{\beta,\un} G\otimes \Cst(G))$ denote the dual coaction.
Then Kaliszewski, Quigg and Landstad define the $E$\nb-crossed product $B\rtimes_{\beta,\Enorm_E}G$
as
$$B\rtimes_{\beta, \Enorm_E}G:=(B\rtimes_{\beta,\un}G)/J_{\beta,\En}\quad \text{with}
\quad J_{\beta,\En}:=\ker(\id\otimes \,q_\En)\circ \widehat{\beta}.$$
They also show that the dual coaction $\widehat{\beta}$ of $G$ on $B\rtimes_{\beta,\un} G$ factors to a
dual coaction $\widehat{\beta}_{\pn_E}$ on the $\En$-crossed product  (\cite[{Theorem~6.2}]{Kaliszewski-Landstad-Quigg:Exotic}).
Moreover, it follows from \cite[Corollary 3.13]{Kaliszewski-Landstad-Quigg:Exotic} that $C_\En^*(G)\cong \C\rtimes_{\Enorm_E}G$.
The $\En$-norm does not change if we pass from $\En$ to its weak-* closure in $B(G)$
by \cite[Lemma 3.5]{Kaliszewski-Landstad-Quigg:Exotic}, so we may always assume that $\En$ is weak\nb-* closed.
Moreover, for $\En=B(G)$ we obtain the universal norm $\|\cdot\|_\un$ and for $\En=B_\red(G)$, the weak-* closure of  Fourier algebra $A(G)$,  we obtain the reduced norm $\|\cdot\|_\red$.
In what follows we shall  write $\pn$ instead of $\Enorm_E$
if  confusion seems unlikely.

For later use it is important for us to see that the above construction of a crossed-product functor works for
any given quotient $C_\nu^*(G)$ of $C^*(G)$ with quotient map
$q:C^*(G)\to C_\nu^*(G)$. For this let $(B,G,\beta)$ be any given \cstar{}dynamical system and let
\begin{equation}\label{eq-new-crossed}
J_\nu:=\ker\left((\id_{B\rtimes_{\beta,\un}G}\otimes q)\circ \widehat{\beta}: B\rtimes_{\beta,\un}G\to \M\big(B\rtimes_{\beta,\un} G\otimes C^*_\nu(G)\big)\right).
\end{equation}
Then we can define $B\rtimes_{\beta,\nu}G:=(B\rtimes_{\beta,\un}G)/J_\nu$. But the following proposition shows that this does not give any new crossed products.

\begin{proposition}\label{prop-crossed}
Let $I\subseteq C^*(G)$ be the kernel of the quotient map $q:C^*(G)\to C_\nu^*(G)$, let $F:=I^\perp$ denote the annihilator of $I$ in
$B(G)$, and let $E:=\langle F\rangle$ denote the  weak-* closed ideal in $B(G)$ generated by $F$. Then the crossed product
$B\rtimes_{\beta,\nu}G$ constructed above coincides with the Kaliszewski-Landstad-Quigg crossed product $B\rtimes_{\beta, \pn}G$
constructed from $E$.  In particular, we have $\C\rtimes_\nu G\cong C_\nu^*(G)$ if and only if $F=I^\perp$ is an ideal in $B(G)$.
\end{proposition}
\begin{proof}
We first note that $F$ is $G$\nb-invariant by \cite[Lemma 3.1]{Kaliszewski-Landstad-Quigg:Exotic}.
Let $\delta_G: C^*(G)\to \M(C^*(G)\otimes C^*(G))$ denote the comultiplication on $C^*(G)$.
If $f\in F$ and $g\in B(G)$, then
the point-wise product $g\cdot f$
is given by the formula $(g\cdot f)(x)=(g\otimes f)\circ \delta_G(x)$ for $x\in C^*(G)$. This implies that
$x\in C^*(G)$ lies in the annihilator $^\perp E$ of $E=\overline{B(G)\cdot F}^{w^*}$ if and only if $x\in J:=\ker(\id_G\otimes q)\circ \delta_G$.

Now let $q_J:C^*(G)\to C^*(G)/J$ denote the quotient map. The proposition will follow as soon as we show that for any system
$(B,G,\beta)$ we get  $$\ker (\id_{B\rtimes_{\beta,\un} G}\otimes q_J)\circ \widehat\beta=\ker (\id_{B\rtimes_{\beta,\un} G}\otimes q)\circ \widehat\beta.$$
To see this we  compute
\begin{align*}
\ker(\id_{B\rtimes_{\beta,\un} G}\otimes q_J)\circ \widehat\beta&=\ker(\id_{B\rtimes_{\beta,\un} G}\otimes (\id_G\otimes q)\circ \delta_G)\circ \widehat\beta\\
&=\ker(\id_{B\rtimes_{\beta,\un} G}\otimes \id_G\otimes q)\circ (\id_{B\rtimes_{\beta,\un} G}\otimes \delta_G)\circ \widehat\beta\\
&= \ker(\id_{B\rtimes_{\beta,\un} G}\otimes \id_G\otimes q)\circ (\widehat\beta\otimes \id_G)\circ \widehat\beta\\
&= \ker(\widehat\beta\otimes q)\circ \widehat\beta\\
&=\ker(\id_{B\rtimes_{\beta,\un} G}\otimes q)\circ \widehat\beta.
\end{align*}
For the first equation we used $\ker q_J=\ker  (\id_G\otimes q)\circ \delta_G$ and for the last equation we used $\ker \widehat\beta=\{0\}$.
This finishes the proof of the first assertion. The last assertion in the statement follows from the first one and the fact observed above that for a $G$-invariant ideal $E\sbe B(G)$ we always have $\C\rtimes_{\mu_E}G\cong C_{E}^*(G)$.
\end{proof}

The following result shows that taking $\En$-crossed products has good functorial properties. For notation, if
$\beta:G\to \Aut(B)$ is an action we always write $(i_B^\pn, i_G^\pn)$ for the canonical
covariant homomorphism of $(B,G, \beta)$ into $\M(B\rtimes_{\beta,\pn}G)$ whose integrated form is the
quotient map $B\rtimes_{\beta,\un}G\onto B\rtimes_{\beta,\pn}G$.

\begin{lemma}\label{lem-sub} Let $(A,G,\alpha)$ and $(B,G,\beta)$ be \cstar{}dynamical systems and assume
that $\phi:A\to \M(B)$ is a (possibly degenerate) $G$\nb-equivariant map. Suppose further that $\En\subseteq B(G)$ is a $G$\nb-invariant nonzero ideal with corresponding crossed-product norm $\|\cdot\|_\pn$.
Then there is a canonical \Star{}homomorphism $\phi\rtimes_\Enorm G\colon A\rtimes_{\alpha,\Enorm}G\to
\M(B\rtimes_{\beta,\Enorm}G)$ given on $\contc(G,A)$ by $(\phi\rtimes_\Enorm G( f))(s)=\phi(f(s))$ (acting via convolution
on $\contc(G,B)\subseteq B\rtimes_{\beta,\Enorm}G$) and such that the following are true:
\begin{enumerate}
\item $\phi\rtimes_\Enorm G$ is the integrated form of the covariant homomorphism $(i_B^\Enorm\circ \phi, i_G^{\Enorm})$
of $(A,G,\alpha)$ whose integrated form factors through $A\rtimes_{\alpha,\Enorm}G$.
\item $\phi\rtimes_\Enorm G$ is nondegenerate if $\phi$ is nondegenerate.
\item $\phi\rtimes_\Enorm G$ is faithful if the corresponding map $\phi\rtimes_\un G:A\rtimes_{\alpha,\un}G\to \M(B\rtimes_{\beta,\un}G)$
for the full crossed products is faithful.
\item $\phi\rtimes_\Enorm G$ is equivariant with respect to the dual coactions
$\widehat{\alpha}_\Enorm$ and $\widehat{\beta}_\Enorm$, \ie, $\widehat\beta_\Enorm\circ (\phi\rtimes_\Enorm G)=(\phi\rtimes_\Enorm G\otimes \id_G)\circ \widehat\alpha_\Enorm$.
\end{enumerate}
\end{lemma}
 \begin{proof} Existence of the map $\phi\rtimes_\Enorm G$  is given by \cite[Proposition 5.2]{Buss-Echterhoff:Exotic_GFPA}
 and it is then easy to check that conditions (1) and (2) hold.  In order to check (4) we first note that
$\widehat\beta\circ (\phi\rtimes_\un G)=(\phi\rtimes_\un G\otimes \id_G)\circ \widehat\alpha$, which follows from the fact that both are given
 by the covariant homomorphism $(i_B\circ \phi\otimes 1, i_G\otimes u)$ of $(A,G,\alpha)$.
 It is clear that this factors through the condition in (4).
 Finally, to check (3) assume that
 $A\rtimes_{\alpha,\un}G$ injects into $\M(B\rtimes_{\beta,\un}G)$ via $\phi\rtimes_\un G$. It then follows that
 \begin{align*}
 J_{\alpha,\En}&=\ker(\id_{A\rtimes_\un G}\otimes q_\En)\circ \widehat{\alpha}
 =\ker(\phi\rtimes_\un G\otimes q_\En)\circ \widehat\alpha\\
&= \ker (\id_{B\rtimes_{\un} G}\otimes q_\En)\circ \widehat\beta\circ (\phi\rtimes_\un G)=\ker(\phi\rtimes_\Enorm G),
 \end{align*}
where in the last line we regard $\phi\rtimes_\Enorm G$ as a homomorphism from $A\rtimes_{\alpha,\un}G$ to
$\M(B\rtimes_{\beta,\Enorm}G)$. Note that all compositions in the above equation exist even if
$\phi$ is degenerate. The only place where degeneracy of $\phi$
 might lead to a problem is the composition  $\ker(\phi\rtimes_\un G\otimes q_\En)\circ \widehat\alpha$. But this
 exists because $\widehat{\alpha}$ takes values in the subalgebra $\widetilde{\M}(A\rtimes_{\alpha,\un}G\otimes C^*(G))$
of ${\M}(A\rtimes_{\alpha,\un}G\otimes C^*(G))$ consisting of those multipliers $m$ which satisfy
$m(1\otimes z)\in A\rtimes_{\alpha,\un}G\otimes C^*(G)$ for all $z\in C^*(G)$.
The result follows.
\end{proof}

\subsection{Crossed products by Hilbert modules}\label{sec:CrossedProductsHilbertModules}\label{subsec-crossed-modules}
If $(B,G,\beta)$ is a \cstar{}dynamical system, then a {\em $G$\nb-equivariant} Hilbert $B$-module is
a Hilbert $B$-module $\E$ together with a strongly continuous action $\gamma:G\to \Aut(\E)$
such that $\braket{\gamma_s(\xi)}{\gamma_s(\eta)}_B=\beta_s\left(\braket{\xi}{\eta}_B\right)$
and $\gamma_s(\xi\cdot b)=\gamma_s(\xi)\cdot\beta_s(b)$
for all $\xi, \eta\in \E$ and $b\in B$. We then say that $(\E,\gamma)$ is a {\em Hilbert $(B,G)$-module.}

If  $(\E,\gamma)$ is a Hilbert $(B, G)$-module, then $\contc(G,\E)$ can be turned into a pre-Hilbert $\contc(G,B)$-module with respect to the following operations:
\begin{equation}\label{eq:InnerProductOfExG}
\braket{x}{y}(t)\defeq \int_G\beta_{s^{-1}}(\braket{x(s)}{y(st)}_B)\dd{s}\quad \mbox{ for all } x,y\in \contc(G,\E),
\end{equation}
\begin{equation}\label{eq:RightActionOnExG}
(x * \varphi)(t)\defeq \int_G x(s)\beta_s(f(s^{-1}t))\dd{s}\quad \mbox{ for all } x\in\contc(G,\E), \varphi\in\contc(G,B).
\end{equation}
If, in addition, $\E$ carries a $G$\nb-equivariant left action of a $G$\nb-algebra $(A,\alpha)$, \ie, if there is a $G$\nb-equivariant \Star{}homomorphism $\phi\colon A\to \Lb(\E)$, written as $a\cdot \xi\defeq \phi(a)\xi$, then $\contc(G,A)$ also acts on the left of $\contc(G,\E)$ by:
\begin{equation}\label{eq-left-action}
(f* x)(t) \defeq \int_G f(s)\cdot \gamma_s(x(s^{-1}t))\dd{s},\quad \mbox{ for all } f\in \contc(G,A), x\in \contc(G,\E).
\end{equation}
In particular, this can be applied to $A=\K(\E)$ with $G$\nb-action $\alpha=\Ad\gamma$. In this case, we
also have a compatible left $\contc(G,\K(\E))$-valued inner product on $\contc(G,\E)$ given by
\begin{equation}\label{eq-CcGK}
_{\contc(G,\K(\E))}\braket{x}{y}(t)
=\int_G\Delta_G(t^{-1}s){ _{\K(\E)}\braket{x(s)}{\gamma_t(y(t^{-1}s))}}\dd{s},
\end{equation}
for all $x,y\in \contc(G,\E)$. It is shown in  \cite{Combes:Crossed_Morita} and \cite{Kasparov:Novikov} that $\contc(G,\E)$ completes to give
a Hilbert $B\rtimes_{\beta,\un} G$-module $\E\rtimes_{\gamma,\un}G$ such that the left action of $\contc(G,A)$
extends to a \Star{}ho\-mo\-mor\-phism $\phi\rtimes_\un G\colon A\rtimes_{\alpha,\un}G\to \Lb(\E\rtimes_{\gamma,\un}G)$.
Similarly, we may regard the above defined inner product with values in the reduced crossed product
$B\rtimes_{\beta,r}G$ and obtain a completion $\E\rtimes_{\gamma,\red}G$ with left action
 $\phi\rtimes_\red G:A\rtimes_{\alpha,\red}G\to \Lb(\E\rtimes_{\beta,\red}G)$. We want to extend this construction
 to arbitrary crossed-product norms. As we shall see below, this works
 especially well if we consider norms related to nonzero $G$\nb-invariant ideals $\En\subseteq B(G)$.

If $E$ is an ideal in $B(G)$ and $\|\cdot\|_\pn$ is the corresponding crossed-product norm on
$\contc(G,B)$  then the pre-Hilbert $\contc(G,B)$-module $\contc(G,\E)$ completes
to a Hilbert $B\rtimes_{\beta,\pn}G$-module $\E\rtimes_{\gamma,\pn}G$
and the identity map on $\contc(G,\E)$ extends to a surjective linear map
$$\tilde\Qn_\pn\colon \E\rtimes_{\gamma,\un}G\onto \E\rtimes_{\gamma,\pn}G.$$
This map is a morphism of Hilbert modules (see
\cite{Echterhoff-Raeburn:Multipliers,Echterhoff-Kaliszewski-Quigg-Raeburn:Categorical})
compatible with the canonical surjection $\Qn_\pn\colon B\rtimes_{\beta,\un}G\onto B\rtimes_{\beta,\pn}G$
meaning that $\tilde\Qn(x*f)=\tilde\Qn(x)\Qn(f)$ and
$\Qn(\braket{x}{y}_{B\rtimes_{\beta,\un}G})=\braket{\tilde\Qn(x)}{\tilde\Qn(y)}_{B\rtimes_{\beta,\pn}G}$
for all $x,y\in \E\rtimes_{\gamma,\un}G$ and $f\in B\rtimes_{\beta,\un}G$. In particular,
$\tilde\Qn$ induces a surjective
 \Star{}homomorphism
 $$\psi_\pn\colon
 \K(\E)\rtimes_{\Ad\gamma,\un} G\cong \K(\E\rtimes_{\gamma,\un} G)\onto \K(\E\rtimes_{\gamma,\pn} G)$$
 satisfying $\psi_\pn\big(_{\K(\E)\rtimes_{\gamma,\un} G}\braket{x}{y}\big)={_{\K(\E\rtimes_{\gamma,\pn}G)}\braket{\tilde\Qn(x)}{\tilde\Qn(y)}}$ for all $x,y\in \E\rtimes_{\gamma,\un} G$.
 We then get

 \begin{lemma}\label{lem-E-imp}
 The  \Star{}homomorphism  $\psi_\pn\colon
 \K(\E)\rtimes_{\Ad\gamma,\un} G\onto \K(\E\rtimes_{\gamma,\pn} G)$
 factors through an isomorphism $\K(\E)\rtimes_{\Ad\gamma,\pn} G\congto \K(\E\rtimes_{\gamma,\pn} G)$.
 Moreover, if $L(\E)=\left(\begin{smallmatrix}\K(\E) & \E\\ \E^* & B\end{smallmatrix}\right)$ is the linking algebra of $\E$
endowed with the canonical $G$\nb-action $\vartheta=\left(\begin{smallmatrix}\alpha & \gamma\\ \gamma^* & \beta\end{smallmatrix}\right)$ \textup(see \cite[Chapter 2.5]{Echterhoff-Kaliszewski-Quigg-Raeburn:Categorical}\textup), then the canonical identification
$$\contc(G,L(\E))\cong \left(\begin{array}{cc} \contc(G,\K(\E)) & \contc(G,\E)\\ \contc(G,\E)^* &\contc(G,B)\end{array}\right)$$
 induces an isomorphism $L(\E)\rtimes_{\vartheta, \pn}G\cong L(\E\rtimes_{\gamma, \pn}G)$.
 \end{lemma}
\begin{proof}
Since the canonical inclusion of $\K(\E)\rtimes_{\Ad\gamma,\un}G$ into $L(\E)\rtimes_{\vartheta,\un} G$ induced by the $G$\nb-equivariant inclusion
$\K(\E)\into L(\E)$ is injective, the canonical map of $\Enorm$\nb-crossed products
$\K(\E)\rtimes_{\Ad\gamma,\Enorm}G\into L(\E)\rtimes_{\vartheta, \Enorm} G$ is injective by Lemma~\ref{lem-sub}.
Thus the result will follow if we can show the  isomorphism
$L(\E)\rtimes_{\vartheta, \Enorm}G\cong L(\E\rtimes_{\gamma, \Enorm}G)$.
Note that this isomorphism holds for the universal crossed products. Applying Lemma~\ref{lem-sub}
to the inclusion $B\into L(\E)$, we see that $\K(\E)\rtimes_{\Ad\gamma, \Enorm}G$ and $B\rtimes_{\beta,\Enorm}G$
are opposite corners in $L(\E)\rtimes_{\vartheta,\Enorm}G$.
Hence, the canonical identification $
\left(\begin{array}{cc} \contc(G,\K(\E)) & \contc(G,\E)\\ \contc(G,\E)^* &\contc(G,B)\end{array}\right)
\cong \contc(G, L(\E))$
is isometric in the upper right and lower left corner with respect to the crossed-product norms $\|\cdot\|_{\Enorm}$
everywhere. But this implies that this identification extends to the desired isomorphism.
\end{proof}

A combination of Lemma~\ref{lem-sub} with Lemma~\ref{lem-E-imp} above shows
 that the $\En$-crossed product construction is functorial with respect to $G$\nb-equivariant correspondences:

\begin{corollary}\label{cor-E-correspond}
 Suppose that $(\E,\gamma)$ is a $G$\nb-equivariant correspondence from $({A},G,{\alpha})$ to $({B},G,{\beta})$,
 \ie, $(\E,\gamma)$ is a Hilbert $(B,G)$-module together with a $G$\nb-equivariant \Star{}homomorphism
 $\phi:A\to \Lb(\E)$.
 Then $\E\rtimes_{\gamma, \Enorm}G$ becomes a correspondence from ${A}\rtimes_{{\alpha},\Enorm}G$ to
 ${B}\rtimes_{{\beta},\Enorm}G$ with left action of $A\rtimes_{\alpha,\Enorm}G$ given by
 $$\phi\rtimes_\Enorm G\colon A\rtimes_{\alpha,\Enorm}G\to \M(\K(\E)\rtimes_{\Ad\gamma, \Enorm}G)\cong \Lb(\E\rtimes_{\gamma,\Enorm}G).$$
 Moreover, this $\En$-crossed product functor respects composition of correspondences in the sense that if $(\E,\gamma)$ is a $G$\nb-equivariant correspondence from $({A},G,\alpha)$ to $({B},G,{\beta})$ and $(\F,\tau)$ is a correspondence from $({B},G,{\beta})$ to $({C},G,{\sigma})$, then
 $$(\E\rtimes_{\gamma,\Enorm} G)\otimes_{B\rtimes_{\beta,\Enorm}G}(\F\rtimes_{\tau,\Enorm} G) \cong (\E\otimes_B\F)\rtimes_{\gamma\otimes\tau,\Enorm}G$$
 as correspondences from ${A}\rtimes_{{\alpha},\Enorm} G$ to ${C}\rtimes_{{\sigma},\Enorm}G$.
\end{corollary}
\begin{proof} The first result follows directly as a combination of
Lemma~\ref{lem-sub} with \linebreak Lemma~\ref{lem-E-imp}.
For the composition, we follow the arguments given in the proof of \cite[Theorem 3.7]{Echterhoff-Kaliszewski-Quigg-Raeburn:Categorical}
to see that the rule
$$\Upsilon(x\otimes y)(s)=\int_G x(t)\otimes \tau_t(y(t^{-1}s))\dd t$$
determines a map $\Upsilon: \contc(G,\E)\odot \contc(G,\F)\to \contc(G, \E\otimes_B\F)$ which preserves
the $\contc(G,C)$-valued inner product, the right $\contc(G,C)$-action, the left $\contc(G,A)$-action and has dense image in the inductive limit topology. Passing to completions with respect to the norms induced by the $E$\nb-norm $\|\cdot\|_\Enorm$
on $C_c(G,C)$ on both sides  gives the desired result.
\end{proof}

\section{Weakly proper Hilbert modules and generalized fixed-point algebras}

In this section we extend the theory of universal or exotic generalized fixed-point algebras
for weakly proper
$X\rtimes G$-algebras as introduced in \cite{Buss-Echterhoff:Exotic_GFPA} to cover weakly proper Hilbert modules as follows:

\begin{definition}
Let $(B,G, \beta)$ be a \cstar{}dynamical system and let $X$ be a proper $G$\nb-space.
A {\em weakly proper $(B, X\rtimes G)$-module} is a Hilbert $(B,G)$-module $(\E,\gamma)$ together
 with a $G$\nb-equivariant nondegenerate \Star{}homomorphism $\phi\colon \contz(X)\to\Lb_B(\E)$.
 In other words,  $(\E,\gamma)$ is a $G$\nb-equivariant correspondence from $C_0(X)$ to $B$.
\end{definition}

Recall from \cite{Buss-Echterhoff:Exotic_GFPA} that a $G$\nb-algebra $(A,\alpha)$ is called a {\em weakly proper} $X\rtimes G$-algebra for the
 proper $G$\nb-space $X$
 if there exists a $G$\nb-equivariant nondegenerate \Star{}homomorphism $\phi:C_0(X)\to \M(A)$. If we view $(A,\alpha)$ as a Hilbert $(A,G)$-module, this just means that $(A,\alpha)$ is a weakly proper
 $(A, X\rtimes G)$-module. Conversely, if $(\E,\gamma)$ is a weakly proper $(B,X\rtimes G)$-module,
 then $(\K(\E), \Ad\gamma)$ is a weakly proper $X\rtimes G$-algebra.
 We should notice that weakly proper $(B, X\rtimes G)$-modules are always continuously square-integrable in the sense of
 Meyer \cite{Meyer:Generalized_Fixed} (see also Remark~\ref{rem:weakly_proper=>square-integrable} below) and that part of our treatment is inspired by his results.

As usual, we use module notation for the left action of $\contz(X)$ on $\E$, \ie, we write $f\cdot\xi\defeq \phi(f)\xi$ for $f\in \contz(X)$ and $\xi\in \E$ whenever $(\E,\gamma)$ is a $(B, X\rtimes G)$-module as above. Moreover, in what follows
we denote by $\tau\colon G\to \Aut(C_0(X))$ the action of $G$ on $C_0(X)$ given by $\tau_s(f)(x)=f(s^{-1}\cdot x)$.

It follows from classical results by Rieffel and Green (see \cite{Green:algebras_transformation_groups, Rieffel:Applications_Morita}) that if
we view $\contc(G, C_0(X))$ as a dense subalgebra of $C_0(X)\rtimes_{\tau,\un} G$
then $\F_c(X):=\contc(X)$ can be made into
a pre-Hilbert $\contc(G, C_0(X))$-module by defining the
inner product and right action of $\contc(G, C_0(X))$ on $\F_c(X)$ by
\begin{align*}
\bbraket{\xi}{\eta}_{\contc(G, C_0(X))}(g,x)&=\Delta(g)^{-1/2}\overline{\xi(x)}\eta(g^{-1} x)\quad\text{and}\\
(\xi* \varphi)(x)&=\int_G\Delta(g)^{-1/2}\xi(g^{-1} x)\varphi(g^{-1}, g^{-1} x)\,dg
\end{align*}
for all $\xi,\eta\in \F_c(X)$ and $\varphi\in \contc(G, C_0(X))$. Hence $\F_c(X)$ completes to give a
Hilbert $C_0(X)\rtimes_{\tau,\un} G$-module $\F(X)$, and the left action of $C_0(G\backslash X)$ on $\F_c(X)$
given by $(f\cdot \xi)(x)=f(Gx)\xi(x)$ induces an isomorphism $C_0(G\backslash X)\cong \K(\F(X))$.
The inner product on $\F(X)$ is full if and only if $G$ acts freely on $X$ (see \cite{Marelli_Raeburn:Proper_actions_not_saturated}), in which case
$\F(X)$ becomes a $C_0(G\backslash X)-C_0(X)\rtimes_{\tau,\un} G$ equivalence bimodule. In general,
we obtain an equivalence bimodule between $C_0(G\backslash X)$ and the ideal
$$I_X:=\cspn\{\bbraket{\xi}{\eta}_{C_0(X)\rtimes_\un G}: \xi,\eta\in \F(X)\}\subseteq C_0(X)\rtimes_{\tau,\un} G.$$
We extend this construction to arbitrary $(B, X\rtimes G)$ Hilbert modules $(\E,\gamma)$ as follows:
let $\F_c(\E):=\contc(X)\cdot \E$. We define a $\contc(G,B)$-valued inner product on $\F_c(\E)$ and a
right action of $\contc(G,B)$ on $\F_c(\E)$ by

\begin{equation}\label{eq:InnerProduct}
\bbraket{\xi}{\eta}(t)\defeq \Delta(t)^{-1/2}\braket{\xi}{\gamma_t(\eta)}_B,
\quad\mbox{and}
\end{equation}
\begin{equation}\label{eq:ModuleConvolution}
\xi*\varphi\defeq \int_G\Delta(t)^{-1/2}\gamma_t(\xi\cdot \varphi(t^{-1}))\dd{t},
\end{equation}
for $\xi,\eta\in \F_c(\E)$ and $\varphi\in \contc(G,B)$. The following lemma will turn out to be
extremely useful. We need some notation:

\begin{definition}\label{ind-limit}
 We say that a net $(\xi_i)$ in $\F_c(\E)$ converges to $\xi\in \F_c(\E)$ in the {\em inductive
limit topology on $\F_c(\E)$}, if there exists $f\in \contc(X)$ such that $f\cdot \xi_i=\xi_i$ for all $i$
and $\xi_i\to \xi$ in $\E$ with respect to $\|\cdot\|_B=\sqrt{\|\braket{\cdot}{\cdot}_B\|}$.
\end{definition}

The inductive limit topology on $\contc(G,B)$ is the usual one given by uniform convergence
of nets with supports in a fixed compact subset of $G$.
\\
\\
{\bf Comment:} Note that although we use the word ``topology'' in the above definition,
we actually will not show (or use) that there is a topology on $\F_c(\E)$ such that convergence of nets  in this topology
is precisely the one defined above. Whenever we use terms like ``continuous with respect to the
inductive topology'' or ``inductive limit dense'', we just mean continuity in terms of inductive limit
convergent nets in the sense of the above definition or  the possibility of approximating elements
by such nets, respectively. We shall do so also for other spaces with similar definitions
of inductive limit convergent nets, like $C_c(G,B)$ or
for the fixed-point algebra with compact supports $A_c^G$ for a weakly proper $X\rtimes G$-algebra $A$.
The notion of inductive limit convergent nets in $A_c^G$ has been introduced in \cite[Definition 2.11]{Buss-Echterhoff:Exotic_GFPA}.
Hence we follow a similar policy as discussed in \cite[Remark 1.86]{Williams:Crossed}.
In this sense we can formulate:

\begin{lemma}\label{lem-inductive-limit-top}
The $\contc(G,B)$-valued inner product on $\F_c(\E)$ is jointly continuous with respect to the inductive
limit topology on $\F_c(\E)$ and the inductive limit topology on $\contc(G,B)$.
\end{lemma}
\begin{proof} This follows as in the proof of \cite[Lemma 2.10]{Buss-Echterhoff:Exotic_GFPA} in the case
of weakly proper $X\rtimes G$-algebras.
\end{proof}

\begin{proposition}\label{prop:PreHilbertModule} Let $\|\cdot\|_\pn$ be an $E$\nb-crossed-product norm on $\contc(G,B)$.
Then the operations~\eqref{eq:InnerProduct} and~\eqref{eq:ModuleConvolution} turn
$\F_c(\E)$ into a pre-Hilbert $\contc(G,B)$-module with respect to $\|\cdot\|_\pn$ and hence $\F_c(\E)$ completes to a Hilbert $B\rtimes_{\beta,\pn} G$-module $\F_\pn(\E)$.
Moreover, there is an isomorphism of Hilbert $B\rtimes_{\gamma,\pn}G$-modules
$$\Psi_\pn\colon \F(X)\otimes_{C_0(X)\rtimes G}(\E\rtimes_{\gamma,\pn}G)\to \F_\pn(\E)$$
given on elementary tensors $\xi \otimes x\in \F_c(X)\odot \contc(G,\E)$ by
\begin{equation}\label{eq-decom} \Psi_\pn(\xi\otimes x)= \xi*x:=\int_G \Delta(t)^{-1/2} \gamma_t(\xi\cdot x(t^{-1}))\dd{t} \in \F_c(\E).
\end{equation}
\end{proposition}

\begin{proof} The proof is very similar to the proof of \cite[Proposition 2.2]{Buss-Echterhoff:Exotic_GFPA} and \cite[Proposition 2.9]{Buss-Echterhoff:Exotic_GFPA}, but we follow a slightly different direction here by proving both
assertions of the proposition simultaneously.

First of all it is straightforward to check that all algebraic properties of a pre-Hilbert module, as
 spelled out in \cite[Lemma 2.16]{Raeburn-Williams:Morita_equivalence}, hold. The only property which deserves a closer look
 is positivity of the inner product: for all $\xi\in \F_c(\E)$ we need to check that $\bbraket{\xi}{\xi}_{\contc(G,B)}$
 is positive as an element of $B\rtimes_{\beta,\pn}G$. Since the product
 on $\F(X)\otimes_{C_0(X)\rtimes G}(B\rtimes_{\beta,\pn}G)$ is positive,  this will follow from
 Lemma~\ref{lem-inductive-limit-top} once we have shown the following
 facts:
 \begin{enumerate}
 \item[(a)] The map $\Psi_\pn$ preserves the $\contc(G,B)$-valued inner product on the dense subspace
 $\F_c(\E)\odot \contc(G,\E)$, and
 \item[(b)] the image $\Psi_\pn\big(\F_c(\E)\odot \contc(G,\E)\big)$ is dense in $\F_c(\E)$ with respect to the
 inductive limit topology.
 \end{enumerate}
 Now, statement (a) follows from a straightforward computation very similar to the one given in
 the proof of \cite[Proposition 2.9]{Buss-Echterhoff:Exotic_GFPA} and (b) follows exactly as in the proof of a similar
 statement given in \cite[Lemma 2.12]{Buss-Echterhoff:Exotic_GFPA}.
 Statements (a), (b) also imply that $\Psi_\pn$ extends from $\F_c(X)\odot \contc(G,B)$ to the
 desired isomorphism of Hilbert modules.
 \end{proof}

For later use we state the following lemma:

\begin{lemma}\label{lem-dense}
Suppose that $\mathcal D$ is norm-dense in $\E$. Then $\contc(X)\cdot \mathcal D$ is inductive limit
dense in $\F_c(\E)$, and therefore it is dense in $\F_\pn(\E)$ for every crossed-product norm
$\|\cdot\|_\pn$ on $\contc(G,B)$.
\end{lemma}
\begin{proof} The first assertion is clear and the second assertion follows from the first
together with the fact that the inductive limit topology on $\contc(G,B)$ is stronger than
the norm-topology for any crossed-product norm (since it is stronger than the $L^1$-topology).
\end{proof}

\begin{definition}
The $\pn$-generalized fixed-point algebra associated to a weakly proper $(B,X\rtimes G)$-module $(\E,\gamma)$ is, by definition, the \cstar{}algebra $\Fix_\pn^G(\E)\defeq \K(\F_\pn(\E))$ of compact operators on $\F_\pn(\E)$.
In the particular case where $\|\cdot\|_\pn=\|\cdot\|_\un$ is the universal norm (or the reduced norm $\|\cdot\|_\red$) we call
$\Fix_\un^G(\E)$ (resp $\Fix_\red^G(\E)$) the {\em universal (resp. reduced) generalized fixed-point algebra associated to $\E$.}
\end{definition}

\begin{corollary}
If $(\E,\gamma)$ is a weakly proper $(B, X\rtimes G)$-module, then $\Fix_\pn^G(\E)$ is Morita equivalent to the ideal
$$\I_\pn\defeq \cspn\bbraket{\F_\pn(\E)}{\F_\pn(\E)}_{B\rtimes_\pn G}=\cspn\{\bbraket{\xi}{\eta}_{B\rtimes_\pn G}:\xi,\eta\in\F_\pn(\E)\}$$
 of $B\rtimes_{\beta,\pn} G$.
\end{corollary}

\begin{remark}\label{rem:weakly_proper=>square-integrable}
Let $(\E,\gamma)$ be a weakly proper $(B, X\rtimes G)$-module. Then the subspace $\F_c(\E)$ of $\E$ consists of square-integrable elements and is relatively continuous by \cite[Theorem~6.5]{Meyer:Generalized_Fixed}.
Then our reduced fixed-point algebra $\Fix_\red^G(\E)$ coincides with the one constructed by
Meyer in \cite{Meyer:Generalized_Fixed}. It can also be obtained from Rieffel's theory in \cite{Rieffel:Proper}
by observing that $\Fix_\red^G(\E)$ coincides with the reduced generalized fixed-point algebra $\K(\E)^{G,\Ad\gamma}_\red$ with respect to the dense subalgebra $\K(\E)_c\defeq \contc(X)\cdot \K(\E)\cdot \contc(X)$, as we shall prove in Proposition~\ref{prop-exotic-compacts} below.
\end{remark}

Our constructions of $\F_\pn(\E)$ and $\Fix_\pn^G(\E)$ depend, \emph{a priori}, on the given structure map $\phi\colon\contz(X)\to\Lb(\E)$. However, the following result says that this dependence is not too strong:

\begin{proposition}\label{prop-independence}
Let $(\E,\gamma)$ be a $(B,G)$-module and let $X$ and $Y$ be proper $G$\nb-spaces.
Assume that $\phi:C_0(X)\to \Lb(\E)$ and $\psi:C_0(Y)\to \Lb(\E)$ are two $G$\nb-equivariant nondegenerate \Star{}homomorphisms such that
$$\phi(\contc(X))\psi(\contc(Y))=\psi(\contc(Y))\phi(\contc(X))$$
and let $\F_\pn(\E)^X$  \textup{(}resp. $\F(\E)^Y_\pn$\textup{)} denote the
module for the $(B,X\rtimes G)$-structure \textup{(}resp. $(B,Y\rtimes G)$-structure\textup{)} on $\E$.
Then $\F_\pn(\E)^X=\F_\pn(\E)^Y$ as Hilbert $B\rtimes_{\beta,\pn}G$-modules.
In particular, both structures give the same generalized fixed-point algebras $\Fix_\pn^G(\E)$.
\end{proposition}
\begin{proof} Since $\contc(Y)\cdot \E$ is dense in $\E$, it follows that
$\contc(X)\cdot \contc(Y)\cdot \E$ is inductive limit dense in $\F_c(\E)^X$
and hence also (norm) dense in $\F_\pn(\E)^X$ by Lemma~\ref{lem-dense}. The assumption $\phi(\contc(X))\psi(\contc(Y))=\psi(\contc(Y))\phi(\contc(X))$
implies that
$\contc(X)\cdot \contc(Y)\cdot \E=\contc(Y)\cdot \contc(X)\cdot \E$ is also dense in $\F_\pn(\E)^Y$.
Since the $\contc(G,B)$-valued inner products of both modules coincide on
$\contc(X)\cdot \contc(Y)\cdot \E$, the result follows.
\end{proof}

\begin{remark}\label{rem-unique} {(a)} It follows from the above result, that if
$\phi:C_0(X)\to \Lb(\E)$ takes values in the center of $\Lb(\E)$,
then every other $G$\nb-equivariant structure map $\psi:C_0(Y)\to\Lb(\E)$
will lead to the same modules and  fixed-point algebras
as $\phi$. This is the situation of proper actions as defined by Kasparov in \cite[\S 3]{Kasparov:Novikov}.
\\
{(b)} Suppose that $(\E_1,\gamma_1)$ is a
weakly proper $(B, X\rtimes G)$-module and $(\E_2, \gamma_2)$ is a weakly proper
$(C, Y\rtimes G)$-module. Let $\E_1\otimes_\nu \E_2$ denote the
the completion of the algebraic tensor product
$\E_1\odot\E_2$ with respect to the canonical $B\otimes_\nu C$-valued inner product for
a given \cstar{}tensor norm $\|\cdot\|_\nu$ on $B\odot C$.
Then $(\E_1\otimes_\nu \E_2, \gamma_1\otimes \gamma_2)$
carries three different canonical structures as weakly proper $(B\otimes_\nu C, G)$-modules:
one given by $\phi\otimes \psi\colon C_0(X\times Y)\to \Lb(\E_1\otimes_\nu \E_2)$,
the second given by $\phi\otimes 1\colon C_0(X)\to \Lb(\E_1\otimes_\nu \E_2)$ and third given by
$1\otimes \psi:C_0(Y)\to \Lb(\E_1\otimes_\nu \E_2)$. The previous proposition implies that
all three structures give the same modules and fixed-point algebras.
\\
{(c)} Suppose that $\E$ is a weakly proper
$(B,X\rtimes G)$-module with structural homomorphism $\phi\colon \contz(X)\to\Lb(\E)$ and suppose $\theta$ is a continuous $G$\nb-equivariant
map from $X$ into another proper $G$\nb-space $Y$. Then $\theta$ induces a $G$\nb-equivariant nondegenerate \Star{}homomorphism $\theta^*\colon \contz(Y)\to\contb(X)$, $f\mapsto f\circ\theta$, which therefore induces a $G$\nb-equivariant \Star{}homomorphism $\psi\defeq \phi\circ\theta^*\colon \contz(Y)\to \Lb(\E)$. It is clear that the images of
$\phi$ and $\psi$ commute (pointwise) in $\Lb(\E)$ and therefore both structures give the same modules $\F_\pn(\E)$ and fixed-point algebras $\Fix_\pn(\E)$ by Proposition~\ref{prop-independence}.
\\
{(d)} Recall from \cite{Baum-Connes-Higson:BC, Kasparov-Skandalis:Bolic}
 that for any locally compact group $G$ there
exists a universal proper $G$\nb-space $\EG$ with the property that for any other
proper $G$\nb-space $X$ we get a continuous $G$\nb-equivariant map $\varphi:X\to \EG$
which is unique up to $G$\nb-homotopy (and hence the space $\EG$ is unique up to $G$\nb-homotopy equivalence).
If $(\E,\gamma)$ is a weakly proper $(B,X\rtimes G)$-module with structure
map $\phi_X:C_0(X)\to \Lb(\E)$, and if $\varphi\colon X\to\EG$ is as above, then we obtain a $G$\nb-equivariant structure map
$\phi_{\EG}\colon C_0(\EG)\to \Lb(\E); \phi_{\EG}(f)=\phi_X(f\circ \varphi)$.
As a special case of (c), both structure maps give the same modules and fixed-point algebras.
In particular, we could always assume that $X=\EG$.
\\
{(e)} It is not clear to us, whether the commutation assumption
$\phi(\contc(X))\psi(\contc(Y))=\psi(\contc(Y))\phi(\contc(X))$ in the proposition is
necessary, \ie, whether there might exist two different structures of a given $(B,G)$-module
$(\E,\gamma)$ as a weakly proper module
which give different generalized fixed-point algebras.
\end{remark}

\begin{definition}\label{def-saturated}
A weakly proper $(B, X\rtimes G)$-module $(\E,\gamma)$ is called {\em universally saturated} if
$\I_u=\cspn\bbraket{\F_u(\E)}{\F_u(\E)}_{B\rtimes_{\beta,\un} G}=B\rtimes_{\beta,\un} G$. Moreover, if $\|\cdot\|_\pn$
is any crossed-product norm on $\contc(G,B)$, we say that $(\E,\gamma)$ is {\em $\pn$-saturated},
if the $B\rtimes_{\beta,\pn} G$-valued inner product on $\F_\pn(\E)$ is full.
\end{definition}

\begin{remark} If $(A,\alpha)$ is a proper $G$\nb-algebra in the sense of Rieffel, then he defined
the action to be {\em saturated}, if the module $\F_\red(A)$ is a full Hilbert $A\rtimes_{\alpha,\red}G$-module.
In our terminology, this means that $(A,\alpha)$
is $\red$-saturated. The following result shows that
this follows if the action is universally saturated, but it is not clear to us, whether the converse
is also true.
\end{remark}

\begin{lemma}\label{lem-univ-sat}
Let $(\E,\gamma)$ be a weakly proper $(B, X\rtimes G)$-module.
Then the following are true:
\begin{enumerate}
\item If the action of $G$ on $X$ is free and $\E$ is full as a Hilbert $B$-module, then $(\E,\gamma)$ is universally saturated.
\item If $(\E,\gamma)$ is universally saturated, then it is also $\pn$-saturated for every
crossed-product norm $\|\cdot\|_\pn$ on $\contc(G,B)$. More generally, if $(\E,\gamma)$ is $\qn$-saturated, then it is also $\pn$-saturated whenever $\|\cdot\|_\pn$ and $\|\cdot\|_\qn$ are crossed-product norms on $\contc(G,B)$ satisfying $\|\cdot\|_\pn\leq \|\cdot\|_\qn$.
\item If $(\E,\gamma)$ is $\pn$-saturated, then $\F_\pn(\E)$ is a $\Fix_\pn^G(\E)-B\rtimes_{\beta,\pn}G$ equivalence bimodule.
\end{enumerate}
\end{lemma}
\begin{proof} For the first assertion recall  that for free and proper actions on $X$ the
Hilbert $C_0(X)\rtimes_{\tau,\un} G$-module $\F(X)$ is full. Moreover, if  $\E$ is $B$-full, then $\E\rtimes_{\gamma,\un} G$ is $B\rtimes_{\beta,\un} G$-full.  Then Proposition~\ref{prop:PreHilbertModule} implies that
$\F_u(\E)\cong \F(X)\otimes_{C_0(X)\rtimes_\un G}(\E\rtimes_{\gamma,\un}G)$ is full as a Hilbert $B\rtimes_{\beta,\un} G$-module.
The second assertion follows from the fact that $\F_\pn(\E)$ is the quotient of $\F_\qn(\E)$
corresponding to the quotient map $Q_{\qn,\pn}\colon B\rtimes_{\beta,\qn} G\to B\rtimes_{\beta,\pn}G$.
The last assertion is clear.
\end{proof}

The following result generalizes the isomorphism of bimodules in Proposition~\ref{prop:PreHilbertModule}:

\begin{proposition}\label{prop:TensorDecompositionFixedPointModules}
Suppose that $(\E_1,\gamma_1)$ is a weakly proper Hilbert $(B,X\rtimes G)$-module and that
$(\E_2,\gamma_2)$ is a Hilbert $(C,G, \sigma)$-module.
Assume further that $\varphi\colon B\to \Lb(\E_2)$ is a $G$\nb-equivariant \Star{}homomorphism.
Then the action of $C_0(X)$ on the first factor of $\E_1\otimes_B\E_2$
gives $(\E_1\otimes_B\E_2,\gamma)$ with $\gamma=\gamma_1\otimes\gamma_2$
 the structure of a weakly proper
 Hilbert $(C,X\rtimes G)$-module.
 Let $\|\cdot\|_\pn$ be any crossed-product norm on $\contc(G,C)$. Then the map which sends
 an elementary tensor $\xi\otimes x\in \F_c(\E_1)\odot \contc(G, \E_2)$ to the element
 $$\xi*x{\defeq}\int_G\Delta(t)^{-1/2}\gamma_t(\xi\otimes x(t^{-1}))\dd{t}\in \F_c(\E_1\otimes_B\E_2)$$
 extends to an isomorphism of Hilbert $C\rtimes_\pn G$-modules:
\begin{equation}\label{eq-factor1}
U\colon \F_\qn(\E_1)\otimes_{B\rtimes_{\beta,\qn} G}(\E_2\rtimes_\pn G)\congto \F_\pn(\E_1\otimes_B\E_2)
\end{equation}
for any crossed-product norm $\|\cdot\|_\qn$ on $\contc(G,B)$ such that the left action of $\contc(G,B)$ on
$\E_2\rtimes_{\gamma_2,\pn}G$ extends to $B\rtimes_{\beta,\qn}G$ \textup(in particular, for $\|\cdot\|_\qn=\|\cdot\|_\pn$ by
Lemma~\ref{lem-E-imp}\textup).
\end{proposition}
\begin{proof}
Let us check that $U$
respects inner products of elementary tensors, \ie,
$$\bbraket{\xi\otimes x}{\eta\otimes y}_{\contc(G,C)}=\bbraket{\xi*x}{\eta*y}_{\contc(G,C)}$$
for all $\xi,\eta\in \F(\E_1)_c\;\text{and}\; x,y\in \contc(G,\E_2)$. In fact,
\begin{align*}
&\bbraket{\xi\otimes x}{\eta\otimes y}_{\contc(G,C)}(t)=
\braket{x}{\bbraket{\xi}{\eta}_{\contc(G,B)}* y}_{\contc(G,C)}(t)\\
&=\int_G \sigma_{s^{-1}}\big(\braket{x(s)}{(\bbraket{\xi}{\eta}_{\contc(G,B)}*y)(st)}_C\big)\,ds\, dr\\
&=\int_G\int_G \sigma_{s^{-1}}\big(\braket{x(s)}{\bbraket{\xi}{\eta}_{\contc(G,B)}(r)\gamma_{2,r}(y(r^{-1}st))}_C\big)\,ds \,dr\\
&=\int_G\int_G \Delta(r)^{-1/2}\sigma_{s^{-1}}\big(\braket{x(s)}{\braket{\xi}{\gamma_{1,r}(\eta)}_B\gamma_{2,r}(y(r^{-1}st))}_C\big)\,ds \,dr\\
&=\int_G\int_G \Delta(r)^{-1/2}\big(\braket{\gamma_{2,s^{-1}}(x(s))}{\braket{\gamma_{1, s^{-1}}(\xi)}{\gamma_{1,s^{-1}r}(\eta)}_B\gamma_{2,s^{-1}r}(y(r^{-1}st))}_C\big)\,ds \,dr\\
&\stackrel{r\mapsto str}{=}
\int_G\int_G \Delta(str)^{-1/2}\big(\braket{\gamma_{2,s^{-1}}(x(s))}{\braket{\gamma_{1, s^{-1}}(\xi)}{\gamma_{1,tr}(\eta)}_B\gamma_{2,tr}(y(r^{-1}))}_C\big)\,ds \,dr\\
&\stackrel{s\mapsto s^{-1}}{=}
\int_G\int_G \Delta(str)^{-1/2}\big(\braket{\gamma_{2,s}(x(s^{-1}))}{\braket{\gamma_{1, s}(\xi)}{\gamma_{1,tr}(\eta)}_B\gamma_{2,tr}(y(r^{-1}))}_C\big)\,ds \,dr\\
&= \int_G\int_G \Delta(str)^{-1/2}\big(\braket{\gamma_{1, s}(\xi)\otimes \gamma_{2,s}(x(s^{-1}))}{\gamma_{1,tr}(\eta)\otimes\gamma_{2,tr}(y(r^{-1}))}_C\big)\,ds \,dr\\
&= \int_G\int_G \Delta(str)^{-1/2}\big(\braket{\gamma_{ s}(\xi\otimes (x(s^{-1}))}{\gamma_{tr}(\eta\otimes y(r^{-1}))}_C\big)\,ds \,dr\\
&=\Delta(t)^{-1/2}\braket{\xi* x}{\gamma_t(\eta*y)}_C=\bbraket{\xi*x}{\eta*y}_{\contc(G,C)}(t).
\end{align*}
Hence $U$ extends to an isometry $\F_\pn(\E_1)\otimes_{B\rtimes_\pn G}\E_2\rtimes_\qn G\congto \F_\qn(\E)$ and the result will follow if we can show that it has dense image.
But an argument similar to the one used in the proof of \cite[Lemma 2.12]{Buss-Echterhoff:Exotic_GFPA} shows that the image is dense in $\F(\E_1\otimes_B\E_2)$ in the inductive
 limit topology, which gives the desired result.
\end{proof}

\begin{remark} The isomorphism $\F_\qn(\E_1)\otimes_{B\rtimes_\qn G}(\E_2\rtimes_\pn G)\cong \F_\pn(\E_1\otimes_B\E_2)$ can also be obtained via the following chain of isomorphisms:
\begin{align*}
\F_\qn(\E_1)\otimes_{B\rtimes_\qn G}(\E_2\rtimes_\pn G)
&\cong \F(X)\otimes_{\contz(X)\rtimes G}(\E_1\rtimes_\qn G)\otimes_{B\rtimes_\qn G}(\E_2\rtimes_\pn G)\\
&\cong \F(X)\otimes_{\contz(X)\rtimes G}\big((\E_1\otimes_B \E_2)\rtimes_\pn G\big)\\
&\cong \F_\pn(\E_1\otimes_B\E_2).
\end{align*}
The only problematic part is the isomorphism
$(\E_1\rtimes_\qn G)\otimes_{B\rtimes_\qn G}(\E_2\rtimes_\pn G)\cong (\E_1\otimes_B \E_2)\rtimes_\pn G$.
But this follows with exactly the same arguments as used in the proof of the corresponding result in
Corollary~\ref{cor-E-correspond}. The case of reduced norms was already known. It  can be obtained from \cite[Theorem~7.1]{Meyer:Generalized_Fixed}.
\end{remark}

Recall from \cite{Buss-Echterhoff:Exotic_GFPA} that if $(A,\alpha)$ is a weakly proper $X\rtimes G$-algebra and if $\|\cdot\|_\pn$
is a crossed-product norm on $\contc(G,A)$ corresponding to a $G$\nb-invariant ideal $E\subseteq B(G)$, then the $\pn$-generalized fixed-point algebra
$A_\pn^G$ can be obtained as a completion of the {\em fixed-point algebra with compact supports} $A_c^G$
defined as follows:
$$A_c^G:=\contc(G\backslash X)\cdot\{ m\in \M(A)^G: m\cdot f, f\cdot m\in A_c \;\forall f\in C_c(X)\}\cdot \contc(G\backslash X)\subseteq \M(A).$$
Here $\M(A)^G$ denotes the classical fixed-point algebra in $\M(A)$ and
$$A_c=\contc(X)\cdot A\cdot \contc(X)\subseteq A.$$
For each $a\in A_c$ the {\em strict unconditional integral} $\int_G^{st} \alpha_t(a)\,dt$  is the unique element in $\M(A)$ which satisfies  the equation
$$\left(\int_G^{st} \alpha_t(a)\,dt\right)c=\int_G \alpha_t(a)c\,dt\quad \forall c\in A_c$$
(see  \cite{Exel:Unconditional}). Note that the integrand on the right hand side is a continuous $A$-valued function with compact support, and therefore {the (Bochner) integral} exists, at least for $c\in A_c$. But the integral also converges for general $c\in A$ (see \cite{Rieffel:Integrable_proper}*{Propositions~4.4, 4.6 and Theorem~5.7} and also \cite{Meyer:Generalized_Fixed}*{Proposition~6.8}).
\medskip
\\
{\bf Notation:} In what follows, if $(A,\alpha)$ is a weakly proper $X\rtimes G$-algebra,
we write
$$\EE(a):=\int_G^{st}\alpha_t(a)\,dt,\quad\text{for $a\in A_c$}.$$
We then have  $A_c^G=\EE(A_c)$ by \cite[Lemma 2.5]{Buss-Echterhoff:Exotic_GFPA}.  Using this, it
 is shown in \cite[\S 2]{Buss-Echterhoff:Exotic_GFPA} that $\F_c(A)$ becomes a partial $A_c^G-\contc(G,A)$ pre-equivalence bimodule with left action of $A_c^G$ on $\F_c(A)$ given by multiplication inside $\M(A)$ and with
$A_c^G$-valued inner product on $\F_c(A)$  given by $_{A_c^G}\bbraket{\xi}{\eta}=
\EE(\xi\eta^*)=\int_G^{st} \alpha_t(\xi\eta^*)\,dt$. The completion $\mathcal F_\pn(A)$ of $\mathcal F_c(A)$ with respect to
$\|\cdot\|_{\F_\pn}:=\sqrt{\|\bbraket{\cdot}{\cdot}\|_\pn}$ then becomes an
$A_\pn^G-A\rtimes_{\alpha,\pn}G$ imprimitivity bimodule.

We now obtain a similar picture for
$\Fix_\pn^G(\E)$. Recall that if $(\E,\gamma)$ is a Hilbert $(B,X\rtimes G)$-module, then
$(\K(\E), \Ad\gamma)$ becomes a weakly proper
$X\rtimes G$-algebra by the same structure map $\phi:C_0(X)\to \M(\K(\E))\cong \Lb(\E)$ as for $\E$.

In the proof of the following proposition we  need the following fact: If $\E_2$ is a Hilbert $B$-module and
$\E_1$ is a Hilbert $\K(\E_2)$-module, then $k\mapsto k\otimes 1$ is an isomorphism between
 $\K(\E_1)$ and  $\K(\E_1\otimes_{\K(\E_2)}\E_2)$.  If the $\K(\E_2)$-valued inner product on $\E_1$ is full, this
 is a consequence of \cite[Proposition~4.7]{Lance:Hilbert_modules}. But the general case follows from this, since
 it is a straightforward consequence of the formula for the inner product on $\E_1\otimes_{\K(\E_2)}\E_2$ that
 $\E_1\otimes_{\K(\E_2)} \E_2\cong \E_1\otimes_{I}(I\cdot\E_2)$ for
 $I:=\overline{\braket{\E_1}{\E_1}}_{\K(\E_2)}$.

\begin{proposition}\label{prop-exotic-compacts}
Assume that $(\E,\gamma)$ is a Hilbert $(B,X\rtimes G)$-module for a proper $G$\nb-space $X$.
The action of $\K(\E)$ on $\E$ induces a left action
of $\K(\E)_c^G\subseteq \Lb(\E)$ on $\F_c(\E)$ and there is a left $\K(\E)_c^G$-valued inner product
on $\F_c(\E)$ given by
\begin{equation}\label{eq-KcG}
{_{\K(\E)_c^G}\bbraket{\xi}{\eta}}=\EE(_{\K(\E)}\braket{\xi}{\eta})\quad\left(=\int_G^{st}\alpha_t({_{\K(\E)}\braket{\xi}{\eta}}) \dd{t}\right),
\end{equation}
such that
${_{\K(\E)_c^G}\bbraket{\xi}{\eta}}\cdot\zeta=\xi\cdot\bbraket{\eta}{\zeta}_{\contc(G,B)}$
for all $\xi,\eta,\zeta\in \F_c(\E)$.
As a consequence, if $\|\cdot \|_\pn$ is a crossed-product norm on $\contc(G,B)$ corresponding to the $G$\nb-invariant
ideal $E\subseteq B(G)$,
then the left action of $\K(\E)_c^G$ on $\F_c(\E)$ extends to
an action of $\K(\E)_c^G$ on the completion $\F_\pn(\E)$ which extends to an isomorphism
$\K(\E)_\mu^G\cong  \Fix_\pn^G(\E)$.
\end{proposition}
\begin{proof} Note first that if $\xi,\eta\in \F_c(\E)=\contc(X)\cdot \E$, then $_{\K(\E)}\braket{\xi}{\eta}\in \K(\E)_c$,
since we can find $f\in \contc(X)$ such that
 $\xi=f\cdot\xi$ and $\eta=f\cdot\eta$ and then
 $$_{\K(\E)}\braket{\xi}{\eta}=_{\K(\E)}\braket{f\cdot \xi}{f\cdot \eta}=f\cdot{_{\K(\E)}\braket{\xi}{\eta}}\cdot\bar{f}.$$
 Hence, the definition ${_{\K(\E)_c^G}\bbraket{\xi}{\eta}}:=\EE(_{\K(\E)}\braket{\xi}{\eta})$ makes sense.
 An easy computation gives the compatibility condition
 ${_{\K(\E)_c^G}\bbraket{\xi}{\eta}}\cdot\zeta=\xi\cdot\bbraket{\eta}{\zeta}_{\contc(G,B)}$
for all $\xi,\eta,\zeta\in \F_c(\E)$. But this implies that the linear span
${_{\K(\E)_c^G}\bbraket{\F_c(\E)}{\F_c(\E)}}$ of all such
inner products ${_{\K(\E)_c^G}\bbraket{\xi}{\eta}}$
 forms a dense subspace of
 $\Fix_\pn^G(\E)=\K(\F_\pn(\E))$.

Since $\K(\E)_c^G=\EE(\K(\E)_c)$ by \cite[Lemma 2.5]{Buss-Echterhoff:Exotic_GFPA},
it follows that ${_{\K(\E)_c^G}\bbraket{\F_c(\E)}{\F_c(\E)}}$ lies in $\K(\E)_c^G$.
 Hence the result will follow, if we can show that there is an isomorphism
 $\K(\E)_{\pn}^G\cong \K(\F_\pn(\E))=\Fix_\pn^G(\E)$ which extends the
 left action of $\K(\E)_c^G$ on $\F_c(\E)$.
Recall from Lemma  \ref{lem-E-imp} that
$\K(\E)\rtimes_{\Ad\gamma,\pn}G\cong \K(\E\rtimes_{\gamma,\pn}G)$.
As a special case of the decomposition (\ref{eq-factor1}) applied to the Hilbert $(\K(\E), X\rtimes G)$-module
$(\K(\E),\Ad\gamma)$ and the given module $(\E,\gamma)$ we get
\begin{equation}\label{eq-decompose}
\F_\pn(\E)\cong\F_\pn(\K(\E)\otimes_{\K(\E)}\E)\cong \F_{\pn}(\K(\E))\otimes_{\K(\E)\rtimes_{\Ad\gamma,\pn}G}(\E\rtimes_{\gamma,\pn}G).
\end{equation}
Hence $\K(\E)_\pn^G=\K(\F_{\pn}(\K(\E)))\cong \K\big(\F_{\pn}(\K(\E))\otimes_{\K(\E)\rtimes_{\Ad\gamma,\pn}G}(\E\rtimes_{\gamma,\pn}G)\big)\cong \K(\F_\pn(\E))$ via $k\mapsto k\otimes 1$ by the discussion preceding this proposition. Using the explicit formula for the isomorphism (\ref{eq-decompose})
as given in Proposition~\ref{prop:TensorDecompositionFixedPointModules}, it is straightforward to check that this isomorphism extends
the canonical left action of $\K(\E)_c^G$ on $\F_c(\E)$.
\end{proof}

\begin{remark}
Let $(\E,\gamma)$ be a weakly proper Hilbert $(B,X\rtimes G)$-module.
Then $\K(\E)=\E\otimes_B\E^*$ as Hilbert $(\K(\E), X\rtimes G)$-modules and we can
apply (\ref{eq-factor1}) to obtain an isomorphism
$$\F_{\pn}(\K(\E))\cong \F_\pn(\E)\otimes_{B\rtimes_\pn G}(\E^*\rtimes_{\pn} G).$$
\end{remark}

In the following proposition we view $G$ as a proper $G$\nb-space with respect to the
{\em right translation action} of $G$ on $G$. Then, if $\rho$ denotes the right regular representation of $G$,  $(L^2(G), \rho)$ becomes a weakly proper $(\C, G\rtimes G)$-module with respect to the representation $M:C_0(G)\to \Lb(L^2(G))$ via
multiplication operators.

\begin{proposition}\label{prop-L2}
Let $(\E,\gamma)$ be a Hilbert $(B,G)$-module and consider the $(B,G\rtimes G)$-module $L^2(G,\E)\cong L^2(G)\otimes \E$
with respect to the diagonal action $\rho\otimes\gamma$ of $G$. Then, for any $E$\nb-crossed-product norm $\|\cdot \|_\pn$ on $\contc(G,B)$, we have
$$\F_\pn\big(L^2(G,\E)\big)\cong \E\rtimes_{\gamma,\pn} G$$
as Hilbert $B\rtimes_{\beta,\pn}G$-modules. It follows
that $\Fix_\pn^G(L^2(G,\E))\cong \K(\E)\rtimes_{\Ad\gamma,\pn}G$.
\end{proposition}
\begin{proof}
Recall that $\E\rtimes_{\gamma, \pn} G$ is the completion of the pre-Hilbert $\contc(G,B)$-module $\contc(G,\E)$ with respect to the structure defined by the formulas~\eqref{eq:InnerProductOfExG} and~\eqref{eq:RightActionOnExG}.
On the other hand, since $\contc(G,\E)$ is inductive-limit dense in $\F_c(L^2(G,\E))=\contc(G)\cdot L^2(G,\E)$, we may view
$\F_\pn\big(L^2(G,\E)\big)$ as the completion of the pre-Hilbert module $\contc(G,\E)$ with respect to the $B\rtimes_\pn G$-valued inner product
\begin{multline*}
\bbraket{x}{y}_{B\rtimes_\pn G}(t)=\Delta(t)^{-1/2}\braket{x}{(\rho\otimes \gamma)_t(y)}_B\\
=\int_G\Delta(t)^{-1/2}\braket{x(s)}{\gamma_t(y(st))}\Delta(t)^{1/2}\dd{s}=\int_G\braket{x(s)}{\gamma_t(y(st))}\dd{s},
\end{multline*}
and the right $\contc(G,B)$-action:
$$x*f(t)=\left(\int_G \Delta(s)^{-1/2}(\rho\otimes\gamma)_s(x\cdot f(s^{-1}))\dd{s}\right)(t)=\int_G\gamma_s(x(ts))\beta_s(f(s^{-1}))\dd{s}.$$
Defining $U\colon \contc(G,\E)\to\contc(G,\E)$ by $U(x)|_t\defeq \gamma_t(x(t))$, it is straightforward to check that $U$ is a linear bijection satisfying $\braket{U(x)}{U(y)}_{B\rtimes_\pn G}=\bbraket{x}{y}_{B\rtimes_\pn G}$ and $U(x*f)=U(x)\cdot f$ for all $x,y\in \contc(G,\E)$ and $f\in \contc(G,B)$. It follows that $U$ extends to a unitary isomorphism $\F_\pn\big(L^2(G,\E)\big) \congto \E\rtimes_\pn G$ of Hilbert $B\rtimes_\pn G$-modules as desired.
The last assertion is then a consequence of Proposition~\ref{prop-exotic-compacts}.
\end{proof}

\begin{corollary}\label{cor-fix-AtensorK}
Let $(A,\alpha)$ be a $G$\nb-algebra and let $\|\cdot\|_\pn$ be an $E$\nb-crossed-product norm on $\contc(G,A)$.
Let $\K=\K(L^2(G))$ be equipped with the $G$\nb-action $\Ad\rho$ and let $A\otimes \K$ be equipped with
the diagonal action $\alpha\otimes\Ad\rho$. Then
$$(A\otimes \K)^{G}_{\pn}\cong A\rtimes_{\alpha,\pn}G.$$
\end{corollary}
\begin{proof}
We have $(A\otimes\K)^G_{\pn}\cong \K(\F_\pn(L^2(G,A)))\cong\K(A\rtimes_\pn G)\cong A\rtimes_\pn G$.
\end{proof}

\begin{remark} {(a)} We should mention that the reduced version of Proposition~\ref{prop-L2}, and hence implicitly of Corollary~\ref{cor-fix-AtensorK}, is already given in \cite[\S7]{Meyer:Generalized_Fixed}. The reduced version of  Corollary~\ref{cor-fix-AtensorK} is  also given in
\cite[Corollary 4.4]{Kaliszewski-Muhly-Quigg-Williams:Fell_bundles_and_imprimitivity_theoremsII}.
As above, these results are always with respect to the canonical structure of $A\otimes \K(L^2(G))$
as a weakly proper $G\rtimes G$-algebra via the representation $1\otimes M\colon \contz(G)\to \M(A\otimes \K(L^2(G)))$ by multiplication operators.
However, if $A$ is a weakly proper $X\rtimes G$-algebra with respect to $\phi\colon \contz(X)\to \M(A)$, then $A\otimes\K(L^2(G))$ may also be viewed a weakly proper $X\rtimes G$-algebra with respect to $\phi\otimes 1\colon \contz(X)\to \M(A\otimes\K(L^2(G)))$ or  as a weakly proper $(X\times G)\rtimes G$-algebra with respect to $\phi\otimes M\colon \contz(X\rtimes G)\to \M(A\otimes \K(L^2(G)))$.
It follows from Proposition~\ref{prop-independence} together with Remark~\ref{rem-unique}(c) that all these structures give the same modules and fixed-point algebras. In particular, for all these structures we get $\F_\pn(L^2(G,A))\cong A\rtimes_{\alpha,\pn} G$  as Hilbert $A\rtimes_{\alpha,\pn}G$-module and hence $\Fix_\pn^G(L^2(G,A))=\K(\F_\pn(L^2(G,A)))
\cong A\rtimes_{\alpha,\pn} G$. This is a bit surprising since the counterexamples in \cite{Meyer:Generalized_Fixed} show that the $G$\nb-algebras of the form $A\otimes\K(L^2(G))$ allow the construction of generalized fixed-point algebras {\em for square-integrable actions}
which are not isomorphic nor Morita equivalent to $A\rtimes_\pn G$ even if $G=\Z$.
\\
{(b)} In \cite[Theorem 2.14]{Echterhoff-Emerson:Structure_proper} it is shown that
the full crossed product $A\rtimes_{\alpha,\un} G$ is isomorphic to the {\em reduced fixed-point algebra}
 $(A\otimes \K(L^2(G))^G_\red$ if $A\otimes \K(L^2(G))$ carries the structure of an $X\rtimes G$-algebra
 for some proper $G$\nb-space $X$ with structure map
 $\phi:C_0(X)\to Z\M(A)$, \ie, for centrally proper actions as studied by Kasparov in \cite{Kasparov:Novikov} and others (\eg, see \cite{Raeburn:Induced_Symmetric}).
 In particular, this shows that for centrally proper $X\rtimes G$-algebras $A$
 with structure map $\phi:C_0(X)\to Z\M(A)$ we get
 $$A\rtimes_{\alpha,\un}G\cong \big(A\otimes\K(L^2(G))\big)^G_\red\cong A\rtimes_{\alpha,\red}G$$
 so that in this case the full and reduced crossed products coincide. One can check that the corresponding isomorphism $A\rtimes_{\alpha,\un} G\congto A\rtimes_{\alpha,\red}G$ is given by the regular representation $\Lambda_A$. This  well-known fact  also follows from \cite[Proposition 4.11]{Tu:Baum-Connes}, since the transformation groupoid $X\rtimes G$ for a proper $G$\nb-space $X$ is amenable.
 \end{remark}

The following theorem shows that for centrally proper $X\rtimes G$-actions the fixed-point algebras $A_\pn^G$
can be described as
\begin{equation}\label{eq-centrally-proper}
A^G:=C_0(G\backslash X)\cdot\{ m\in \M(A)^G: \contz(X)\cdot m\sbe A\}\subseteq \M(A).
\end{equation}
This  coincides with the algebra  introduced by Kasparov in \cite[p.~164]{Kasparov:Novikov}.
 \begin{theorem}\label{thm-centrally-proper}
 Suppose that $(\E,\gamma)$ is a weakly proper $(B, X\rtimes G)$-module such that the structure map
 $\phi:C_0(X)\to \Lb(\E)$ takes values in the center of $\Lb(\E)$. Then
 $\F_\un(\E)=\Fix_\pn(\E)=\F_\red(\E)$ and, hence, $\Fix_\un^G(\E)=\Fix_\pn^G(\E)=\Fix_\red^G(\E)$
 for every $E$\nb-crossed-product norm $\|\cdot\|_\pn$.
 If $\E$ is universally saturated and full as a Hilbert $B$-module, then we also have $B\rtimes_{\beta,\un}G=B\rtimes_{\beta,\pn}G=B\rtimes_{\beta,\red}G$.

In particular, if $A$ is a proper $X\rtimes G$-algebra with structure map $\phi\colon C_0(X)\to Z\M(A)$,
then all generalized fixed-point algebras $A_\pn^G$
coincide with the algebra $A^G$  as defined in  (\ref{eq-centrally-proper}) above.

\end{theorem}
 \begin{proof} Since $\F_\red(\E)$ is a quotient of $\F_\pn(\E)$, which in turn is a quotient of $\F_\un(\E)$, and since they agree if and only if their algebras of compact operators agree, which follows from the Rieffel correspondence between submodules of $\F_\un(\E)$ and
 ideals in $\K(\F_\un(\E))=\Fix_\un^G(\E)$, it suffices for the first assertion to show that
 $\Fix_\red^G(\E)=\Fix_\un^G(\E)$. But, by Proposition~\ref{prop-exotic-compacts}, we have
 $\Fix_\red^G(\E)=\K(\E)_\red^G$ and $\Fix_\un^G(\E)=\K(\E)_\un^G$, so we may assume without loss of generality
 that we are in the situation of a proper $X\rtimes G$-algebra $A$ with structure map $\phi:C_0(X)\to Z\M(A)$.
 But then we know from the above remark that $A\rtimes_{\alpha,\un} G=A\rtimes_{\alpha,\red}G$, which
 implies that $\F_\un(A)=\F_\red(A)$ and hence $A^G_\un=\K(\F_\un(A))=\K(\F_\red(A))=A_\red^G$.
 If the action of $G$ on $\E$ is universally saturated and $\E$ is $B$-full, the equation $\F_\un(\E)=\F_\red(\E)$ also implies that
 $B\rtimes_{\beta,\un} G=B\rtimes_{\beta,\red}G$ by the Rieffel correspondence between submodules of
 $\F_\un^G(\E)$ and ideals in $B\rtimes_{\beta,\un} G$.

 For the second assertion, we first observe that $A^G$ lies in the closure $A_\red^G$ of $A_c^G$
 inside $\M(A)$. So we have to show that $A_r^G\subseteq A^G$. By continuity of multiplication in $\M(A)$, it follows that $\tilde{A}^G=\{ m\in \M(A)^G: \contz(X)\cdot m\sbe A\}$ is closed in $\M(A)$ and it obviously contains $A_c^G$. Then $A_r^G$ is a closed subalgebra of $\tilde{A}^G$.  Since $A_r^G$ is a nondegenerate $C_0(G\backslash X)$-module by \cite{Buss-Echterhoff:Exotic_GFPA}*{Proposition~6.2}, we get
\begin{equation*}
A_\red^G=C_0(G\backslash X)\cdot A_\red^G\subseteq C_0(G\backslash X)\cdot\tilde{A}^G=A^G.\qedhere
\end{equation*}
\end{proof}

\begin{remark}\label{left-action}
For later use, it is important to describe the isomorphism
 $$A\rtimes_{\alpha,\pn}G\cong (A\otimes \K(L^2(G)))^G_\pn$$
 of Proposition~\ref{prop-L2} for the case $\E=A$ as Hilbert $(A, G)$-module
on the level of the dense subalgebra $\contc(G,A)$.
For this, let $\tilde\alpha:A\to C_b(G,A)\subseteq \M(A\otimes C_0(G))$ be given by
$\tilde\alpha(a)(s)=\alpha_{s^{-1}}(a)$ and let $\pi:A\to \M(A\otimes \K(L^2(G)))$ be the composition
$\pi:=(\id_A\otimes M)\circ \tilde\alpha$.
Combining \cite[Lemmas 7.4 and 7.6]{Williams:Crossed},
we have a canonical isomorphism
$C_0(G,A)\rtimes_{\tau\otimes\alpha,\un}G\cong A\otimes \K(L^2(G))$  given by the integrated form
of the covariant homomorphism $\big((1\otimes M)\otimes \pi, 1\otimes\lambda_G\big)$.
The isomorphism
$$\Phi:=((1\otimes M)\otimes \pi)\rtimes (1\otimes \lambda)\colon
C_0(G,A)\rtimes_{\tau\otimes\alpha}G\congto A\otimes \K(L^2(G))$$
is $\sigma-\alpha\otimes \Ad\rho$-equivariant
for the action $\sigma:G\to \Aut(C_0(G,A)\rtimes_{\tau\otimes\alpha}G)$ given on
$F\in \contc(G, C_0(G, A))$ by $\sigma_t(F)(s,r)=F(s, rt)$.
 Thus, if we view $\contc(G,A)$ as a dense submodule of $L^2(G,A)$, the left module action of $\contc(G, \contc(G,A))$ and
 the left $\contc(G,\contc(G,A))$-valued
 inner product are given by
\begin{equation}\label{eq-L2GA}
 \begin{split}
 F\cdot x(t)&=\int_G \alpha_{t^{-1}}(F(s,t))x(s^{-1}t)\dd{s}\quad\text{and}\\
{_{\contc(G, \contc(G,A))}\braket{x}{y}}(s,r)&=\Delta_G(s^{-1}r)\alpha_r\big(x(r)y(s^{-1}r)^*\big)
 \end{split}
\end{equation}
for $F\in \contc(G, \contc(G,A))$ and $x,y \in \contc(G,A)$.
The first formula follows from evaluating $\Phi(F)x$ at $t\in G$.  The second formula comes from
 the requirement ${_{\contc(G, \contc(G,A))}\braket{x}{y}}\cdot z=x\cdot\braket{y}{z}_A$ for $x,y,z\in
 \contc(G,A)$ together with the formulas  $\braket{x}{y}_A=\int_Gx(s)^*y(s)\dd{s}$
 and $(x\cdot a)(t)=x(t)a$ for $x,y\in  L^2(G,A)$ and $a\in A$.

Now, if we consider $A$ as the subalgebra of constant functions in $\M(C_0(G,A))$, we obtain
a natural embedding of $\contc(G,A)$ into $\big(C_0(G,A)\rtimes_{\tau\otimes \alpha}G)^G_c$.
Restricting $\Phi$ to this subspace gives the integrated form of the covariant pair
$(\pi, 1\otimes \lambda)$ of $(A,G,\alpha)$, which is just the regular representation $\Lambda_A$
of $(A,G,\alpha)$ on $L^2(G,A)$ as discussed in \S 2.
Hence, the left action of $\contc(G,A)$ and the left $\contc(G,A)$-valued inner product on $\contc(G,A)\subseteq \F_c(L^2(G,A))$ are given by
\begin{equation}\label{eq-Fix-L2GA}
\begin{split}
(f\cdot x)(t)&=\int_G \alpha_{t^{-1}}(f(s))x(s^{-1}t)\dd{s}\quad\text{and}\\
{_{\contc(G,A)}\bbraket{x}{y}}(t)&=\left(\int_G\sigma_r({_{\contc(G, \contc(G,A))}\braket{x}{y}})\dd{r}\right)(t)\\
&=\int_G \Delta_G(t^{-1}r)\alpha_r\big(x(r)y(r^{-1}t)^*\big)\dd{r}.
\end{split}
\end{equation}
One can check that the transformation
$U\colon \F_c(L^2(G,A))\to \contc(G,A); (Ux)(t)=\alpha_t(x(t))$,
as given in the proof of Proposition~\ref{prop-L2}, transforms these actions and inner products into the usual
convolution formulas for $f*x$ and $x *y^*$ on $\contc(G,A)\subseteq A\rtimes_{\pn}G$.
\end{remark}

We close this section with a result which shows that starting with a
$(B, X\rtimes G)$-module $(\E, \gamma)$ we can recover $(\E,\gamma)$ as a $(B,G)$-module
 from any of the modules $\F_\pn(\E)$ if $\|\cdot\|_\pn$ is the
$E$\nb-crossed-product norm on
$\contc(G,B)$ for some $G$\nb-invariant ideal $E\subseteq B(G)$. The result is an analogue of
Meyer's Theorem~5.3 in \cite{Meyer:Generalized_Fixed}. But we should note that we did not succeed to recover
also the structure map $C_0(X)\to\Lb(\E)$ from the modules $\F_\pn(\E)$, which is not surprising in the light of
Proposition~\ref{prop-independence} and Remark \ref{rem-unique}.

\begin{proposition}\label{prop:RecoveringModuleViaL2(G,B)}
Let $(\E,\gamma)$ be as above. For $g\in C_c(G,B)$, write $\tilde{g}(t):=\beta_t(g(t^{-1}))$. Then there is an isomorphism of Hilbert $(B,G)$-modules
$$\Phi\colon \F_\pn(\E)\otimes_{B\rtimes_\pn G}L^2(G,B)\congto \E$$
given on elementary vectors $\xi\otimes g\in \F_c(\E)\odot\contc(G,B)$ by
$$\Phi(\xi\otimes g){\defeq}\xi*\tilde g{=\int_G\Delta(t)^{-1/2}\gamma_{t}(\xi\cdot \tilde g(t^{-1}))\dd{t}=}\int_G\Delta(t)^{-1/2}\gamma_{t^{-1}}(\xi)g(t)\dd{t}.$$
\end{proposition}
\begin{proof}
Given $\xi,\eta\in \F_c(\E)$ and $f,g\in \contc(G,B)$, we have
\begin{align*}
\braket{\xi\otimes f}{\eta\otimes g}_B & =\braket{f}{\Lambda_B(\bbraket{\xi}{\eta}_{B\rtimes G})g}_B=\int_G f(t)^{-1}\big(\Lambda_B(\bbraket{\xi}{\eta}_{B\rtimes G})g\big)(t)\dd{t}\\
    &=\int_G\int_G f(t)^*\beta_{t^{-1}}(\bbraket{\xi}{\eta}_{B\rtimes G}(s))g(s^{-1}t)\dd{s}\dd{t}\\
    &=\int_G\int_G \Delta(s)^{-1/2}f(t)^*\beta_{t^{-1}}(\braket{\xi}{\gamma_s(\eta)}_B)g(s^{-1}t)\dd{s}\dd{t}\\
    &=\int_G\int_G \Delta(ts)^{-1/2}f(t)^*\braket{\gamma_{t^{-1}}(\xi)}{\gamma_{s^{-1}}(\eta)}_Bg(s)\dd{s}\dd{t}\\
    &=\int_G\int_G \braket{\Delta(t)^{-1/2}\gamma_{t^{-1}}(\xi)f(t)}{\Delta(s)^{-1/2}\gamma_{s^{-1}}(\eta)g(s)}_B\dd{s}\dd{t}\\
    &=\braket{\xi*\tilde f}{\eta*\tilde g}_B.
\end{align*}
It follows that $\xi\otimes f\mapsto \xi*\tilde f$ extends to an isometry $\Phi\colon \F_\pn(\E)\otimes_{B\rtimes_\pn G}L^2(G,B)\congto \E$. To see that it is surjective, let $\xi\in \F_c(\E)$ and $\epsilon>0$. By the continuity of $t\mapsto \Delta(t)^{-1/2}\gamma_{t^{-1}}(\xi)$, there is a compact unit neighborhood $V\sbe G$
with
$$\|\Delta(t)^{-1/2}\gamma_{t^{-1}}(\xi)-\xi\|<\epsilon/2\quad\mbox{ for all } t\in V.$$
 Take $b\in B$ with $\|b\|\leq 1$ and $\|\xi-\xi\cdot b\|<\epsilon/2$ and a non-negative function $g\in \contc(G)$ with $\supp(g)\sbe V$ and $\int_Gg(t)\dd{t}=1$. Then, for $f\defeq g\otimes b\in \contc(G,B)$, we get
\begin{multline*}
\|\xi*f-\xi\|\leq \|\xi*f-\xi\cdot b\|+\|\xi\cdot b-\xi\|\leq \\
\int_V\left(\|\Delta(t)^{-1/2}\gamma_{t^{-1}}(\xi)-\xi\|\|b\|g(t)\right)\dd{t}+\frac{\epsilon}{2}\leq \epsilon.
\end{multline*}
Writing $\gamma'$ for the $G$-action $\beta\otimes\rho$  on $L^2(G,B)\cong B\otimes L^2(G)$, it
is easy to see that $\gamma_s(\xi*\tilde f)=\xi*\widetilde{\gamma'_s(f)}$ for all $\xi\in\F_c(\E)$, $f\in\contc(G,B)$, $s\in G$.
This implies that $\Phi$ is $G$\nb-equivariant. This finishes the proof.
\end{proof}

\section{Weakly proper equivalence bimodules}

In this section we want to study partial equivalence bimodules between two weakly proper $X\rtimes G$-algebras.
We shall see that such bimodules allow a canonical construction of a \emph{fixed bimodule}, which is a partial
equivalence bimodule between the generalized fixed-point algebras.
In the next lemma, we let $(L(\E), \theta)=\left(\left(\begin{smallmatrix} A&\E\\ \E^*& B\end{smallmatrix}\right), \left(\begin{smallmatrix} \alpha&\gamma\\ \gamma^*& \beta\end{smallmatrix}\right)\right)$ denote the linking algebra of a $G$\nb-equivariant partial $(A,\alpha)-(B,\beta)$ equivalence bimodule $(\E,\gamma)$ and $\|\cdot\|_\pn$ denotes the $E$\nb-crossed-product norm for a fixed $G$\nb-equivariant ideal $E\subseteq B(G)$. By Lemma~\ref{lem-E-imp}, there is a canonical isomorphism $L(\E)\rtimes_{\theta,\pn}G\cong L(\E\rtimes_{\gamma,\pn}G)$.

\begin{lemma}\label{lem-XG-bimodule}
Suppose that $(A,\alpha)$ and $(B,\beta)$ are weakly proper $X\rtimes G$-algebras
with structure maps $\phi_A:C_0(X)\to \M(A)$ and $\phi_B:C_0(X)\to \M(B)$ and let
$(\E, \gamma)$ be a partial $(A,\alpha)-(B,\beta)$ equivalence bimodule.
Then the linking algebra $(L(\E),\theta)$
becomes a weak $X\rtimes G$-algebra with structure map
$$\phi_L:C_0(X)\to \M(L(\E)); \quad \phi_L(f)=\left(\begin{smallmatrix} \phi_A(f)& 0\\ 0& \phi_B(f)\end{smallmatrix}\right).$$
Moreover, if $p=\left(\begin{smallmatrix} 1& 0\\ 0& 0\end{smallmatrix}\right)$ and $q=\left(\begin{smallmatrix} 0& 0\\ 0& 1\end{smallmatrix}\right)\in \M(L(\E))$, the corner $\E^G_\pn:=p\big(L(\E)^G_\pn\big) q$ becomes a
partial $A_\pn^G- B_\pn^G$ equivalence bimodule.
\end{lemma}

\begin{remark}\label{rem-proper-module}
{(a)} Of course, the lemma will apply for any $(B, X\rtimes G)$-module $(\E,\gamma)$ if
$(B,\beta)$ itself is a weakly proper $X\rtimes B$-algebra. In this case,
$(\E,\gamma)$ becomes an equivariant partial $\K(\E)-B$ equivalence bimodule
and $(\K(\E),\Ad\gamma)$ is a weakly proper $X\rtimes G$-algebra.
\\
{(b)} The above result can easily be extended to the case where $(A,\alpha)$ is a weakly proper
$X\rtimes G$-algebra and $(B,\beta)$ is a weakly proper $Y\rtimes G$-algebra for two
proper $G$\nb-spaces $X$ and $Y$, since by Remark~\ref{rem-unique} both spaces $X$ and $Y$ can be replaced
by a universal proper $G$\nb-space $\EG$.
\\
{(c)} By construction, the linking algebra $L(\E_\pn^G)$ of the partial $A^G_\pn-B^G_\pn$ equivalence bimodule $\E^G_\pn$ is
 canonically isomorphic to $L(\E)^G_\pn$.
\end{remark}

The proof of the lemma will follow at once, if we can show that $p$
(and hence also $q=1-p$) may be regarded as projections in $\M(L(\E)^G_\pn)$. But this follows from the following more general result:

\begin{lemma}
Let $(A,\alpha)$ be a weakly proper $X\rtimes G$-algebra and let $p\in \M(A)$ be a $G$\nb-invariant projection such that $p\cdot f=f\cdot p$ for all $f\in\contz(X)$.
Then, for every $E$\nb-crossed-product norm $\|\cdot\|_\pn$ on $\contc(G,A)$, the projection $p$ induces a projection $\tilde p\in \M(A^G_\pn)$ which acts as $p$ on $A^G_c$, \ie, $\tilde pa=pa$ for all $a\in A^G_c$.
\end{lemma}
\begin{proof}
Given $\xi=f\cdot a\in \F_c(A)$, for some $f\in \contc(X)$ and $a\in A$, we observe that $p\xi=f\cdot (pa)\in \F_c(A)$. Moreover, if $\phi\rtimes_\pn G:C_0(X)\rtimes G\to \M(A\rtimes_{\alpha,\pn}G)$ denotes the homomorphism
determined by the $G$-equivariant structure map $\phi:C_0(X)\to \M(A)$, we get
\begin{align*}
\bbraket{p\xi}{p\xi}_{A\rtimes_\mu G}&=\bbraket{(p\cdot f)a}{(p\cdot f)a}_{A\rtimes_\mu G}\\
&=\iota_A^\mu(a)^*\iota_A^\mu(p)(\phi\rtimes_\pn G(\bbraket{f}{f}_{C_0(X)\rtimes G})\iota_A^\mu(p)\iota_A^\mu(a).
\end{align*}
Since $p$ commutes with every element in $\phi(\contz(X))$ and is $G$\nb-invariant, it follows that $P\defeq \iota_A^\mu(p)\in \M(A\rtimes_\pn G)$ is a projection which commutes with every element in $\phi\rtimes_\pn G(\contz(X)\rtimes G)$. In particular, $P$ commutes with the positive operator $T\defeq\phi\rtimes_\pn G(\bbraket{f}{f}_{{C_0(X)\rtimes G}})$. It follows that $PTP\leq T$ and hence
$$\bbraket{p\xi}{p\xi}_{A\rtimes_\pn G}\leq \iota_A^\mu(a)^*\left(\phi\rtimes_\pn G(\bbraket{f}{f}_{{C_0(X)\rtimes G}})\right)\iota_A^\mu(a)=\bbraket{\xi}{\xi}_{A\rtimes_\mu G}.$$
Therefore, the map $\F_c(A)\to\F_c(A), \xi\mapsto p\xi$ extends to a bounded operator $\tilde p\colon \F_\pn(A)\to \F_\pn(A)$ which is adjointable with $\tilde p^*=\tilde p=\tilde p^2$, \ie, a projection $\tilde p\in \Lb(\F_\pn(A))=\M(A^G_\pn)$ as desired.
\end{proof}

\begin{remark}\label{rem:AlternativeDescriptionFixedModule}
The module $\E_\pn^G$ of Lemma~\ref{lem-XG-bimodule} can also be obtained as follows:
we first observe that the structure maps $\phi_A,\phi_B:C_0(X)\to \M(A), \M(B)$ define left and right actions of $C_0(X)$
on $\E$ by $f\cdot \xi:= \phi_A(f)\xi$ and $\xi\cdot f:=\xi\phi_B(f)$. With this notation we let
$\E_c:=\contc(X)\cdot \E\cdot \contc(X)$ and define
$$\E_c^G:=\contc(G\bs X)\cdot\{ \xi\in \M(\E)^G: f\cdot \xi,\,\xi\cdot f\in \E_c\;\mbox{ for all } f\in \contc(X)\}\cdot \contc(G\bs X).$$
Here, for convenience, we define $\M(\E)=pM(L(\E))q$ with $p,q$ as above. But it follows from \cite[Proposition 1.51]{Echterhoff-Kaliszewski-Quigg-Raeburn:Categorical}
that this coincides with the usual notion of the multiplier bimodule $\M(\E)\cong \Lb_B(B,\E)$ if
the left inner product is full (\ie, if $A\cong \K(\E)$ -- see \cite[\S 1.2]{Echterhoff-Kaliszewski-Quigg-Raeburn:Categorical}).
The module actions and inner products on $\E$ extend to the multiplier module and
then induce an $A_c^G-B_c^G$ pre-Hilbert bimodule structure
on $\E_c^G$ which then completes to a partial $A_\pn^G-B_\pn^G$ equivalence bimodule $\E_\pn^G$.
Using the obvious identification $\E_c^G\cong p L(\E)_c^G q$ it is not difficult to check that both constructions coincide.
\end{remark}

In what follows, we are going to relate the fixed bimodule construction $\E\mapsto \E^G_\pn$ with our previous construction $\E\mapsto \F^G_\pn(\E)$.

\begin{proposition}\label{prop-product}
 Suppose that $(A,\alpha)$ and $(B,\beta)$ are weakly proper $X\rtimes G$-algebras
 and assume that $(\E,\gamma)$ and $(\D,\theta)$ are two weakly proper $(C, X\rtimes G)$-modules
 for the $G$\nb-algebra $(C,\sigma)$ such that $(A,\alpha)\cong (\K(\E),\Ad\gamma)$ and
 $(B,\beta)=(\K(\D),\Ad \theta)$. Then there is a canonical isomorphism of partial $A^G_\pn-B^G_\pn$ equivalence bimodules
 $$\F^G_\pn(\E)\otimes_{C\rtimes_{\pn} G} \F^G_\pn(\D)^*\cong (\E\otimes_C \D^*)_\pn^G$$
 which is given on elementary tensors $\xi\otimes \eta\in \F_c(\E)\odot \F_c(\D)^*$ by
 $$\EE(\xi\otimes \eta)\defeq\int_G^{st} \gamma_s(\xi)\otimes\theta_s(\eta)\dd{s}\in (\E\otimes_B\D^*)_c^G.$$
\end{proposition}
\begin{proof}
We should first give an interpretation for the strict unconditional integral $\int_G^{st} \gamma_s(\xi)\otimes\theta_s(\eta)\dd{s}$:
If $\xi\in \F_c(\E)=\contc(X)\cdot \E$ and $\eta\in \F_c(\D)^*=\D^*\cdot C_c(X)$, then
$\xi\otimes \eta\in \contc(X)\big(\E\otimes_C\D^*)\cdot \contc(X)\subseteq \contc(X)\cdot L(\E\otimes_C\D^*)\cdot \contc(X)=
L(\E\otimes \D^*)_c$, and the integral can be interpreted as the restriction of the map
$\EE_L\colon L(\E\otimes_C\D^*)_c\to L(\E\otimes_C\D^*)_c^G$ to the upper right corner $(\E\otimes_C\D^*)_c$.
By \cite[Lemma 2.5]{Buss-Echterhoff:Exotic_GFPA}, the map $\EE_L$ is surjective, and since $\E\odot\D^*$ is norm dense in
$\E\otimes_C\D^*$, it follows that $\F_c(\E)\odot \F_c(\D^*)=\contc(X)\cdot \big(\E\odot\D^*\big)\cdot \contc(X)$
is inductive limit dense in $(\E\otimes_C\D^*)_c$. Thus, the map $\EE$ of the lemma is an isomorphism if it
preserves the inner products and module actions.

We do the computation for the right inner product and the right
module action. The left formulas follow by symmetry from similar computations.
So let $\xi_1,\xi_2\in \F_c(\E)$ and let $\eta_1,\eta_2\in \F_c(\D)^*$. Then, for elements $b,c\in B_c$, we get
\begin{align*}
b \bbraket{\EE(\xi_1\otimes\eta_1)}{&\,\EE(\xi_2\otimes \eta_2)}_{B_c^G} c\\
&=
b\left(\int_G^{st}\int_G^{st}\braket{\gamma_s(\xi_1)\otimes \theta_s(\eta_1)}{\gamma_t(\xi_2)\otimes \theta_t(\eta_2)}_B\dd{s}\dd{t}\right)c\\
&=\int_G\int_G b\braket{\gamma_s(\xi_1)\otimes \theta_s(\eta_1)}{\gamma_t(\xi_2)\otimes \theta_t(\eta_2)}_B c\dd{s}\dd{t}\\
&=\int_G\int_G
b\braket{\theta_s(\eta_1)}{\braket{\gamma_s(\xi_1)}{\gamma_t(\xi_2)}_C \theta_t(\eta_2)}_Bc\dd{s}\dd{t}\\
&=\int_G\int_G
b\beta_s\big(\braket{\eta_1}{\braket{\xi_1}{\gamma_{s^{-1}t}(\xi_2)}_C \theta_{s^{-1}t}(\eta_2)}_B\big)c\dd{s}\dd{t}\\
 &\stackrel{t\mapsto st}{=}\int_G\int_G
b \beta_s\big(\braket{\eta_1}{\braket{\xi_1}{\gamma_{t}(\xi_2)}_C \theta_{t}(\eta_2)}_B\big) c\dd{s}\dd{t}\\
 &=\int_G\int_G
b\beta_s\big(\braket{\eta_1}{\Delta(t)^{1/2}\bbraket{\xi_1}{\xi_2}_{C\rtimes_\pn G}(t) \theta_{t}(\eta_2)}_B\big)c\dd{s}\dd{t}\\
 &=\int_G
b\beta_s\big(\braket{\eta_1}{\bbraket{\xi_1}{\xi_2}_{C\rtimes_\pn G}\cdot \eta_2}_B\big)c\dd{s}\\
&=b\bbraket{\eta_1}{\bbraket{\xi_1}{\xi_2}_{C\rtimes_\pn G}\cdot\eta_2}_{B_c^G}c
=b\bbraket{\xi_1\otimes\eta_1}{\xi_2\otimes\eta_2}_{B_c^G}c.
\end{align*}
Note that all integrals starting from the third line are over functions with compact supports,
 hence they exist and the manipulations
are justified.
Assume now that $b\in B_c^G$, $\xi\in \F_c(\E)$ and $\eta\in \F_c(\D)^*$. Since $b\in \M(B)$ is $G$\nb-invariant,
we have $\EE(\xi\otimes \eta)\cdot b=\EE(\xi\otimes \eta \cdot b)$, which implies that $\EE$ is also compatible
with the right module actions.
\end{proof}

\begin{corollary}\label{cor-full}
Let $(A,\alpha)$, $(B,\beta)$ and $(\E,\gamma)$ be as in Lemma~\ref{lem-XG-bimodule}.
If the $A$-valued inner product on $\E$ is full, \ie, if $A\cong \K(\E)$, then
$$\E_\pn^G\cong \F_\pn^G(\E)\otimes_{B\rtimes_\pn G}\F_\pn^G(B)^*\cong \F_\pn^G(A)\otimes_{A\rtimes_\pn G}(\E\rtimes_\pn G)\otimes_{B\rtimes_\pn G}\F_\pn^G(B)^*$$
as partial $A^G_\pn-B^G_\pn$ equivalence bimodules.

In particular, if $\E$ is an $A-B$ equivalence bimodule and the actions of $G$ on $A$ and $B$ are $\pn$-saturated, then $\E_\pn^G$ becomes a \textup(full\textup) $A_\pn^G-B_\pn^G$ equivalence bimodule.
\end{corollary}
\begin{proof}
We have a canonical isomorphism $\E\cong \E\otimes_B B^*$, so
the isomorphism $\E_\pn^G\cong \F_\pn^G(\E)\otimes_{B\rtimes_\pn G}\F_\pn^G(B)^*$ follows directly from Proposition~\ref{prop-product}.
Proposition~\ref{prop:TensorDecompositionFixedPointModules} applied for $A\otimes_A\E\cong \E$ yields a canonical isomorphism
$\F^G_\pn(\E)\cong \F^G_\pn(A)\otimes_{A\rtimes_\pn G}\E\rtimes_\pn G$ which in combination with the first part implies the second isomorphism
$\E_\pn^G\cong \F_\pn^G(A)\otimes_{A\rtimes_\pn G}(\E\rtimes_\pn G)\otimes_{B\rtimes_\pn G}\F_\pn^G(B)^*$.
\end{proof}

\begin{example}\label{ex-not-equivalent}
We should note that the final conclusion of Corollary~\ref{cor-full} does not hold in general
without the assumption that the actions on both algebras are saturated.
As an example, let $G$ be a compact group and consider $L^2(G)$ as a $\K(L^2(G))-\C$ equivalence bimodule. Let $G$ act on $L^2(G)$ via the right regular
representation $\rho:G\to \U(L^2(G))$. Then $L^2(G)$ becomes an equivariant
$\K(L^2(G))-\C$ equivalence bimodule with respect to the weakly proper
$\{\pt\}\rtimes G$-actions $\Ad\rho$ and $\id_\C$ on the algebras.
By Corollary~\ref{cor-fix-AtensorK}, we get $\K(L^2(G))^G\cong \Cst(G)$ and
we certainly have $\C^G=\C$. It is clear that both algebras are Morita equivalent
if and only if $G$ is the trivial group.
So, in general, the module $\E_\pn^G$ will only be a partial equivalence bimodule,
even if $\E$ is assumed to be a full $A-B$ equivalence bimodule.
\end{example}

\section{The imprimitivity theorems}

In this section we want to apply our techniques to derive several imprimitivity theorems for weakly proper actions.
Throughout this section we fix two locally compact groups $G$ and $H$. These groups will act on spaces, algebras and modules in a commutative way, \ie,
we will consider $G\times H$-actions. If $G\times H$ acts on a space $X$, we will write $\tau^G$ and $\tau^H$ for the corresponding $G$- and $H$\nb-actions.
We also use similar notation for actions on algebras or modules. Also, to avoid confusion and to be clear which action is being used, we make use of superscript notations, like $\F^G$ for our functor $\E\mapsto \F^G(\E)$ whenever $\E$ is a $(B,X\rtimes G)$-module.

Recall that for any $G\times H$-algebra $(B,\beta)$, the natural map $\contc(G\times H,B)\to \contc(G, \contc(H,B))$
induces isomorphisms
$$B\rtimes_{\beta,\un}(G\times H)\cong (B\rtimes_{\beta^G,\ung}G)\rtimes_{\beta^H,\unh}H$$
for the full crossed products. Here,  by abuse of notation, we also write $\beta^H$ for the action of $H$ on $B\rtimes_{\beta^G,\ung}G$
which is given for $f\in \contc(G,B)$ by $\beta^H_h(f)(t):=\beta^H_h(f(t))$ (and similarly
for the action of $G$ on $B\rtimes_{\beta^H,\unh}H$).
The isomorphism follows from the observation that $C_c(G\times H,B)$
is dense in both algebras and that $(\pi, v)\mapsto (\pi\rtimes v|_G, v|_H)$ gives a bijection of
covariant representations. This correspondence of representations sends the regular representation $\Ind_{\{e\}}^{G\times H}\pi$
induced from a faithful
representation $\pi:B\to \Lb({\Hils_\pi})$
of $(B, G\times H,\beta)$ on $L^2(G\times H,{\Hils_\pi})$ to the regular representation $\Ind_{\{e\}}^H(\Ind_{\{e\}}^G\pi)$
of $(B\rtimes_{\beta^G,\red^G}G, H, \beta^H)$
on $L^2(H, L^2(G,{\Hils_\pi}))$, and therefore we obtain a similar isomorphism
$B\rtimes_{\beta,\red}(G\times H)\cong (B\rtimes_{\beta^G,\redg}G)\rtimes_{\beta^H,\redh}H$
for the reduced crossed products.

In particular, changing the roles of $G$ and $H$, we obtain canonical isomorphisms
$$(B\rtimes_{\beta^H,\unh}H)\rtimes_{\beta^G,\ung}G\cong B\rtimes_{\beta,\un}(G\times H)\cong (B\rtimes_{\beta^G,\ung}G)\rtimes_{\beta^H,\unh}H$$
and a similar isomorphism exist for the reduced crossed products.

Suppose now that we have two crossed-product functors for $G$ and $H$ in the sense of Kaliszewski-Landstad-Quigg
(see \S2)  with corresponding norms $\pn^G$ and $\pn^H$.
By functoriality (see Lemma~\ref{lem-sub}), it follows that $\beta^G$ factors through an action
on $B\rtimes_{\beta^H,\mu^H}H$ (and vice versa), and we obtain two canonical quotient maps
\begin{align*}
q_{H,G}: &B\rtimes_{\beta,\un}(G\times H)\to (B\times_{\beta^H,\pnh}H)\rtimes_{\beta^G,\png}G\\
q_{G,H}: &B\rtimes_{\beta,\un} (G\times H)\to (B\times_{\beta^G,\png}G)\rtimes_{\beta^H,\pnh}H.
\end{align*}
Here we increased the abuse of notation by denoting the action of $G$ on $B\rtimes_{\beta^H,\mu^H}H$ still by $\beta^G$ (and similarly
for $\beta^H$).

\begin{definition}\label{def-compatible-action}
Let $(\png, \pn, \pnh)$ be a triple of  crossed-product norms for $G, G\times H$, and $H$ corresponding
to ideals $E_G\subseteq B(G), E\subseteq B(G\times H)$, and $E_H\subseteq B(H)$, respectively.
We then say that $(\png, \pn, \pnh)$ is {\em compatible}, if for every $G\times H$-algebra $B$, the above quotient maps
$q_{H,G}$ and $q_{G,H}$ factor through isomorphisms
$$(B\times_{\beta^G,\png}G)\rtimes_{\beta^H,\pnh}H \stackrel{q_{G,H}}{\longleftarrow}
B\rtimes_{\beta,\mu}(G\times H)\stackrel{q_{H,G}}{\longrightarrow}  (B\times_{\beta^H,\pnh}H)\rtimes_{\beta^G,\png}G,$$
respectively.
\end{definition}

By the above discussion, we know that the triple $(\ung, \un, \unh)$  of universal norms and the
triple $(\redg,\red,\redh)$ of reduced norms are compatible triples of crossed-product norms.
But in general it seems to be quite difficult to decide whether a given triple of crossed-product norms
is compatible. But if one of the groups, say $H$,  is amenable, we shall see that every crossed-product norm
$\png$ for $G$ fits into a triple $(\png,\pn, \unh)$ for the unique crossed-product norm $\unh$ for $H$:

\begin{proposition}\label{prop-compatible}
Suppose that $G$ and $H$ are locally compact groups such that $H$ is amenable. Let $E\subseteq B(G)$ be any
weak-* closed $G$\nb-invariant ideal of $B(G)$ and let $F\subseteq B(G\times H)$ be the ideal generated by
$E\otimes 1_H\subseteq B(G\times H)$. Let $\pn$ be the crossed-product norm for $G\times H$ corresponding to $F$
and $\png$ the crossed-product norm for $G$ corresponding to $E$. Then $(\png,\pn,\unh)$ is a compatible triple
of crossed-product norms for $G, G\times H, H$.
\end{proposition}
\begin{proof}
Let us start with some notational comments: Let $(i_B, i_G, i_H)$ denote the canonical maps
of $(B,G,H)$ into $\M(B\rtimes_{\beta,\un}(G\times H))$. Then we use the same symbols for the corresponding
maps of $(B,G,H)$ into the iterated crossed products $(B\rtimes_{\beta_G,\ung}G)\rtimes_{\beta^H,\unh}H$
and $(B\rtimes_{\beta_H,\unh}H)\rtimes_{\beta^G,\ung}G$ which correspond to $(i_B, i_G, i_H)$ via the
canonical isomorphisms between  $B\rtimes_{\beta,\un}(G\times H)$ and the respective iterated crossed products.
To make things more precise: if $(j_{B\rtimes G}, j_H)$ denote the canonical maps of $(B\rtimes_{\beta^G,\ung}G, H)$
into $\M((B\rtimes_{\beta_G,\ung}G)\rtimes_{\beta^H,\unh}H)$ and $(k_B, k_G)$ denote the
canonical maps of $(B,G)$ into $\M(B\rtimes_{\beta^G,\ung}G)$, then $i_B=j_{B\rtimes G}\circ k_B$, $i_G=j_{B\rtimes G}\circ k_G$,
and $i_H=j_H$ if we identify $B\rtimes_{\beta,\un}(G\times H)$ with the iterated crossed product
$(B\rtimes_{\beta_G,\ung}G)\rtimes_{\beta^H,\unh}H$, and similarly for  the other iterated crossed product.

Consider now the canonical quotient map $q: C^*(G\times H)\to C_E^*(G)$ given by the composition
of the canonical quotient map $C^*(G\times H)\to C^*(G)$  followed by the quotient map
$q_{E}:C^*(G)\to C_E^*(G)$.
Realizing elements of $E$ as
coefficient functions of representations of $C_E^*(G)$, it is straightforward to check that the kernel
$I:=\ker q$ equals the annihilator $^\perp{(E\otimes 1_H)}\subseteq C^*(G\times H)$. Thus, it follows from
Proposition~\ref{prop-crossed}
that for any action $\beta: G\times H\to \Aut(B)$ the crossed-product $B\rtimes_{\beta,\pn}(G\times H)$
is given by the quotient of $B\rtimes_{\beta,\un}(G\times H)$ by the kernel $J_{\beta,F}$ of the homomorphism
$$(\id_{B\rtimes (G\times H)}\otimes q)\circ \widehat{\beta}: B\rtimes_{\beta,\un}(G\times H)\to
\M( B\rtimes_{\beta,\un}(G\times H)\otimes C_E^*(G)).$$
If ${v_E}:G\to\U\M(C_E^*(G))$ denotes the unitary representation which integrates to the
quotient map $q_{E}:C^*(G)\to C_E^*(G)$, then $(\id_{B\rtimes (G\times H)}\otimes q)\circ \widehat{\beta}$ is the integrated
form of the covariant homomorphism $(i_B\otimes 1)\rtimes \big((i_G\otimes {v_E})\times (i_H\otimes 1)\big)$.

We now consider the iterated crossed products. We start with $(B\rtimes_{\beta^G,\png}G)\rtimes_{\beta^H,\unh}H$.
The crossed product $B\rtimes_{\beta^G,\png}G$ can be identified with the image
of the homomorphism
$$\Psi:=(\id_{B\rtimes G}\otimes q_E)\circ\widehat{\beta^G}: B\rtimes_{\beta^G,\ung}G\to \M(B\rtimes_{\beta^G,\ung}G\otimes C_E^*(G)).$$
If we let $H$ act on $B\rtimes_{\beta^G,\ung}G\otimes C_E^*(G)$ via $\beta^H\otimes \id$, then $\Psi$ becomes $H$\nb-equivariant
and since the $\png$-crossed product is functorial by Lemma~\ref{lem-sub}, we see that $J_{\beta^G,E}\subseteq B\rtimes_{\beta^G, \ung}G$ is $\beta^H$-invariant. Since $H$ is amenable, and hence the universal crossed product by $H$ coincides with the reduced one, it follows that the homomorphism $$\Psi\rtimes H: (B\rtimes_{\beta^G,\ung}G)\rtimes_{\beta^H,\unh}H\to \M\big((B\rtimes_{\beta^G,\ung}G\otimes C_E^*(G))\rtimes_{\beta^H\otimes\id,\unh}H\big)$$
factors through an isomorphism of $(B\rtimes_{\beta^G,\png}G)\rtimes_{\beta^H,\unh}H$ with the image of $\Psi\rtimes H$.
This follows from the general observation that if  $C$ and $D$ are $H$-algebras and $\Phi:C\to \M(D)$ is a faithful $H$-equivariant representation, then the composition of $\Phi\rtimes H$ with a regular representation of $D\rtimes_\unh H$ is a regular representation of $C\rtimes_\unh H$.
By the amenability of $H$, we also have a canonical isomorphism
$$(B\rtimes_{\beta^G,\ung}G\otimes C_E^*(G))\rtimes_{\beta^H\otimes\id,\unh}H\cong
\big((B\rtimes_{\beta^G,\ung}G)\rtimes_{\beta^H\otimes\id,\unh}H\big)\otimes C_E^*(G)$$
(\eg, see \cite[Lemma 7.16]{Williams:Crossed}). Note that if $(l_{B\rtimes G\otimes C_E^*(G)}, l_H)$ are the canonical maps
of $(B\rtimes_{\beta^G,\ung}G\otimes C_E^*(G), H)$ into $\M\big((B\rtimes_{\beta^G,\ung}G\otimes C_E^*(G))\rtimes_{\beta^H\otimes\id,\unh}H\big)$
and if $(k_B, k_G)$ denote  the canonical maps of  $(B,G)$ into $\M(B\rtimes_{\beta^G,\ung}G)$, then, via the above isomorphism, we get the
identifications
$$i_B\otimes 1= l_{B\rtimes G\otimes C_E^*(G)}\circ (k_B\otimes 1), \quad i_G\otimes 1=  l_{B\rtimes G\otimes C_E^*(G)}\circ (k_G\otimes 1),
\quad\text{and}\quad
i_H\otimes 1= l_H.$$
Then, with a little book keeping, we see that the homomorphism
$\Psi\rtimes H$, viewed as a map into $\M((B\rtimes_{\beta^G,\ung}G)\rtimes_{\beta^H\otimes\id,\unh}H\big)\otimes C_E^*(G))$
is given as the integrated form of the covariant homomorphism
$\big((i_B\otimes 1)\rtimes (i_G\otimes {v_E}), i_H\otimes 1\big)$. This coincides with
$(i_B\otimes 1)\rtimes \big((i_G\otimes {v_E})\times (i_H\otimes 1)\big)=(\id_{B\rtimes (G\times H)}\otimes q)\circ \widehat{\beta}$ after identifying $(B\rtimes_{\beta^G,\ung}G)\rtimes_{\beta^H,\unh}H$ with $B\rtimes_{\beta,\un}(G\times H)$.
This implies the isomorphism $B\rtimes_{\beta,\mu}(G\times H)\cong (B\rtimes_{\beta^G,\png}G)\rtimes_{\beta^H,\unh}H$.

For the other iterated crossed product, note that the crossed product
$(B\rtimes_{\beta^H,\unh}H)\rtimes_{\beta^G,\png}G$ is the quotient of
$(B\rtimes_{\beta^H,\un}H)\rtimes_{\beta^G,\ung}G$ by the kernel of the homomorphism
$(\id_{B\rtimes H\rtimes G}\otimes q)\circ \widehat{\beta^G}$ which is given as the integrated form of the homomorphism
$({i}_B\otimes 1)\rtimes (i_H\otimes 1), i_G\otimes v_{E})$  which again corresponds to
$(i_B\otimes 1)\rtimes \big((i_G\otimes {v_E})\times (i_H\otimes 1)\big)=(\id_{B\rtimes (G\times H)}\otimes q)\circ \widehat{\beta}$ after identifying $(B\rtimes_{\beta^H,\unh}H)\rtimes_{\beta^G,\ung}G$ with $B\rtimes_{\beta,\un}(G\times H)$. This completes the proof.
\end{proof}

Assume that $(B,\beta)$ is a $G\times H$-algebra
and $X$ is a $G\times H$-space such that the action of $G$ on $X$ is proper.
Assume further that $(\E,\gamma)$ is a \emph{weak} $(B, X\rtimes(G\times H))$-module, meaning that we have
a $G\times H$-equivariant nondegenerate structure map $\phi:C_0(X)\to \Lb(\E)$, but we require
properness only for the $G$\nb-action. Assume further that $(\png,\pn, \pnh)$ is a compatible triple of
norms for $G, G\times H$, and $H$ as explained above.

\begin{lemma}\label{lem-actionH}
In the above situation, the Hilbert $B\rtimes_{\beta^G,\png} G$-module $\F_\png(\E)$ carries an $H$\nb-action,  given
on $\F_c(\E)\subseteq \E$ by the restriction of $\gamma^H$ to this subspace  (and which by further abuse of notation we
also denote $\gamma^H$) such that $(\F_\png(\E),\gamma^H)$ becomes a $(B\rtimes_{\beta^G,\png}G, H)$-Hilbert module.
As a consequence, if we write $(A,\alpha):=(\K(\E),\Ad\gamma)$, we obtain the $H$\nb-action $\alpha^H:=\Ad\gamma^H$ on
the fixed-point algebra $A_\png^G:=\Fix_\png^G(\E)=\K(\F_\png(\E))$.
It is given for $m\in A_c^G\subseteq \M(A)$ via the extension of $\alpha$ to $\M(A)$.
\end{lemma}
\begin{proof}
One can use the isomorphism $\F_\png(\E)\cong \F(X)\otimes_{\contz(X)\rtimes G}(B\rtimes_{\beta^G,\png}G)$ from Proposition~\ref{prop:PreHilbertModule} and the diagonal action $\tau^H\otimes \beta^H$ on $\F(X)\otimes_{\contz(X)\rtimes G}(B\rtimes_{\beta^G,\png}G)$
to get an $H$\nb-action $\gamma^H$ on $\F_\png(\E)$ compatible with $\beta^H$. A simple computation using the explicit isomorphism in Proposition~\ref{prop:PreHilbertModule} shows that $\gamma^H$, on the dense subspace $\F_c(\E)\sbe \F_\png(\E)$, is the restriction of the original $H$\nb-action on $\E$ to $\F_c(\E)\subseteq \E$. Since
${_{\K(\E)_c^G}\bbraket{\F_c(\E)}{\F_c(\E)}}$ is inductive limit dense in $A_c^G$, this also implies the last statement of the lemma.
\end{proof}

On the other hand, writing $(A,\alpha):=(\K(\E),\Ad\gamma^H)$, if we first take crossed products by $H$, it follows from Lemma~\ref{lem-E-imp}
that we obtain the Hilbert $B\rtimes_{\beta^H,\pnh}H$-module
$\E\rtimes_{\gamma^H,\pnh}H$ with compact operators $A\rtimes_{\alpha^H,\pnh}H$.
As above, we obtain canonical actions $\alpha^G, \gamma^G$ and $\beta^G$  on
these crossed products. Moreover, the composition
$i_{A}^\pnh\circ \phi:C_0(X)\to \M(A\rtimes_{\alpha^H,\pnh}H)\cong\Lb(\E\rtimes_{\gamma^H, \pnh}H)$
gives $(\E\rtimes_{\gamma^H, \pnh}H, \gamma^G)$ the structure of a weakly proper
$(B\rtimes_{\beta^H,\pnh}H, X\rtimes G)$-module. Using these notations, we get

\begin{theorem}\label{theo:TheFixedPointFunctorForCommutingActions}
Let $(\E,\gamma)$ be a weak $(B,X\rtimes (G\times H))$-module as above such that $G$ acts properly on $X$. Then, identifying
$$(B\rtimes_{\beta^G,\png}G)\rtimes_{\beta^H,\pnh} H\cong (B\rtimes_{\beta^H,\pnh} H)\rtimes_{{\beta}^G,\png}G$$
and writing $(A,\alpha):=(\K(\E),\Ad\gamma^H)$ we get
$$\F^G_\png(\E)\rtimes_{\gamma^H,\pnh} H\cong \F^G_\png(\E\rtimes_{\gamma^H,\pnh} H)
\quad\text{and} \quad A_\png^G\rtimes_{\alpha^H,\pnh}H\cong
(A\rtimes_{\alpha^H,\pnh}H)_\png^G.$$
The second isomorphism is an extension of
the canonical inclusion of $\contc(H, A_c^G)\subseteq \contc(H, \M(A))$ into $(A\rtimes_{\alpha^H,\pnh}H)_c^G\subseteq
\M(A\rtimes_{\alpha^H,\pnh}H)$.
\end{theorem}

\begin{proof}
 Observe that $\contc(H,\F_c^G(\E))$, with respect to the topology on $\F_c^G(\E)$ inherited from $\E$, is inductive limit dense in
$\F_c^G(\E\rtimes_{\pnh}H)$  and in $\contc(H, \F_\png^G(\E))$
in the sense of Definition~\ref{ind-limit}. It therefore suffices to check that the
$\contc(G\times H, B)$-valued inner products on both spaces coincide on $\contc(H,\F_c^G(\E))$.
But it follows from the formulas (\ref{eq:InnerProductOfExG}) and (\ref{eq:InnerProduct}) that both inner
products are given for $\xi,\eta\in \contc(H,\F_c^G(\E))$ by the formula
$$\bbraket{\xi}{\eta}(s,h)=\int_H\Delta_G(s)^{-1/2}\beta_{l^{-1}}\big(\braket{\xi(l)}{\beta_s(\eta(lh))}_B\big)\dd{l}.$$
The resulting isomorphism of Hilbert modules $\F^G_\png(\E)\rtimes_\pnh H\cong \F^G_\png(\E\rtimes_\pnh H)$
induces an isomorphism between the compact operators
$\K(\F^G_\png(\E)\rtimes_\pnh H)=A_\png^G\rtimes_{\alpha^H,\pnh}H$
and
$\K(\F_\png^G(\E\rtimes_{\pnh}H))=(A\rtimes_{\alpha^H,\pnh}H)_\png^G$, and this isomorphism sends
the left $A_\png^G\rtimes_{\alpha^H,\pnh}H$-valued inner product of two elements $\xi,\eta \in \contc(H,\F_c^G(\E))$
to the $(A\rtimes_{\alpha^H,\pnh}H)_\png^G$-valued inner product of these elements.
But for $\xi,\eta\in \contc(H, \F_c^G(\E))$ the first inner product takes values in $\contc(H, A_\png^G)$ with the
formula
\begin{equation*}
_{\contc(H, A_\pn^G)}\bbraket{\xi}{\eta}(h)=\int_H \Delta_H(h^{-1}l) {_{A_c^G}\bbraket{\xi(l)}{\gamma_h(\eta(h^{-1}l))}}\dd{l}
\end{equation*}
\begin{equation}\label{eq:integral}
=\int_H\int_G^{st}\Delta_H(h^{-1}l)\alpha^G_s\big(_A\braket{\xi(l)}{\gamma_{h}(\eta(h^{-1}l))}\big)\dd{s}\dd{l},
\end{equation}
which follows from a combination of (\ref{eq-CcGK}) with (\ref{eq-KcG}),
while the $(A\rtimes_{\alpha^H,\pnh}H)_\png^G$-valued inner product takes values in
$(A\rtimes_{\alpha^H,\pnh}H)_c^G$ given by
$$_{(A\rtimes_{\alpha^H,\pnh}H)_c^G}\bbraket{\xi}{\eta}=\int_G^{st} \alpha^G_s\big(_{\contc(H,A)}\braket{\xi}{\eta}\big)\dd{s}.$$
If we evaluate the integrand at $h\in H$, we just obtain \eqref{eq:integral}, which
 implies the last assertion of the theorem.
\end{proof}

Of course, the isomorphism $A_\png^G\rtimes_{\alpha^H,\pnh}H\cong (A\rtimes_{\alpha^H,\pnh}H)_\png^G$
applies for any weak $X\rtimes(G\times H)$-algebra $(A,\alpha)$ in which the action of $G$ on $X$ is proper.
We only need to apply the theorem to the weak Hilbert $(A, X\rtimes(G\times H))$-module $(A,\alpha)$.
In particular we get $A_\ung^G\rtimes_{\alpha^H,\unh}H\cong (A\rtimes_{\alpha^H,\unh}H)_\ung^G$
and $A_{\redg}^G\rtimes_{\alpha^H,\redh}H\cong (A\rtimes_{\alpha^H,\redh}H)_{\redg}^G$ for the triple of universal norms
$(\ung,\un, \unh)$ and the triple of reduced norms $(\redg, \red, \redh)$, respectively.
Moreover, if $H$ is amenable, a combination of Theorem~\ref{theo:TheFixedPointFunctorForCommutingActions}
with Proposition~\ref{prop-compatible} implies:

\begin{corollary}\label{cor-amenable-fix}
Suppose that $G$ and $H$ are locally compact groups with $H$ amenable, and let $\|\cdot\|_\png$ be the $E$\nb-crossed-product norm for $G$
associated to a $G$\nb-invariant ideal $E\subseteq B(G)$. Let $X$ be a $G\times H$-space such that $G$ (resp. $H$) acts properly on $X$
and let $A$ be a weak $X\rtimes (G\times H)$-algebra with action $\alpha:G\times H\to \Aut(A)$. Then
$A_\png^G\rtimes_{\alpha^H}H\cong (A\rtimes_{\alpha^H}H)_\png^G$ \textup(resp. $A^H\rtimes_{\alpha^G,\png}G\cong (A\rtimes_{\alpha^G,\png})^H$\textup). \end{corollary}

Note that in the above corollary  we omitted to indicate the norm in the notation of the $H$\nb-crossed products and $H$\nb-fixed-point algebras, respectively, since for amenable groups all $H$\nb-crossed-product norms coincide.

\subsection{Green's imprimitivity theorem}
If $G$ is a locally compact group and $H$ is a closed subgroup of $G$,
we write $\lt$ for the left
translation action of $G$  on $\contz(G/H)$. Then $\K(L^2(G))^{H,\Ad\rho}_{\unh}\cong \contz(G/H)\rtimes_{\lt,\ung} G$. In fact, the canonical isomorphism $\contz(G)\rtimes_{\lt,\ung} G\cong \K(L^2(G))$ transforms the action of $H$ on $C_0(G)\rtimes_{\lt,\ung} G$ induced by the right translation action $\rt$ of $H$ on $\contz(G)$ to the $H$\nb-action $\Ad\rho$ on $\K(L^2(G))$. Hence, applying {Theorem~\ref{theo:TheFixedPointFunctorForCommutingActions}}, we get
$$\K(L^2(G))^{H,\Ad\rho}_\unh\cong (\contz(G)\rtimes_{\lt, \ung}G)^{H,\rt }_\unh\cong \contz(G)^{H,\rt}\rtimes_{\lt,\ung} G\cong \contz(G/H)\rtimes_{\lt,\ung} G.$$
A similar argument shows that $(A\otimes\K(L^2(G)))^{H,\alpha\otimes\Ad{\rho}}_\unh\cong (A\otimes \contz(G/H))\rtimes_{\alpha\otimes\rt, \ung}G$ whenever $(A,\alpha)$ is a $G$\nb-algebra.

More generally, suppose $(B,\beta)$ is an $H$\nb-algebra and let
$(\Ind_H^G(B,\beta),\Ind\beta)$ be the \emph{induced $G$\nb-algebra} $\contz(G,B)^H$ obtained as in Lem\-ma~\ref{lem-actionH}
from the weak $G\rtimes (H\times G)$-algebra $\contz(G,B)$ with structure map $M:C_0(G)\to Z\M(C_0(G,B))$, the
proper diagonal $H$\nb-action $\rt\otimes \beta$ and the proper $G$\nb-action $\lt\otimes\id_B$.
It follows then from Theorem~\ref{thm-centrally-proper} that the full and reduced fixed-point algebras
coincide and that $\Ind_H^G(B,\beta)=\contz(G,B)^H$ can be described as
\begin{multline*}
\Ind_H^G(B,\beta)=\{f\in \contb(G,B): \beta_t(f(st))=f(s)\mbox{ for all }s\in G, t\in H,\\
\mbox{ and }(sH\mapsto \|f(s)\|)\in \contz(G/H)\},
\end{multline*}
with induced $G$\nb-action given by left translation: $(\Ind\beta)_s(f)|_t=f(s^{-1}t)$.

\begin{proposition}
Let $(B,\beta)$ be any $H$-algebra and let $(\png,\pnh)$
be any pair of crossed-product norms which fit into a compatible triple of norms $(\png,\pn,\pnh)$ as
in Definition \ref{def-compatible-action}. Then
 \begin{equation}\label{eq-greencompacts}
 \Ind_H^G(B,\beta)\rtimes_{\Ind\beta,\png}G \cong \big(B\otimes \K(L^2(G))\big)^{H,\beta\otimes \Ad\rho}_\pnh.
 \end{equation}
\end{proposition}
\begin{proof}
Observe that we have a canonical $H$\nb-equivariant isomorphism
$$C_0(G,B)\rtimes_{\lt\otimes\id,\png}G\cong (B\otimes C_0(G))\rtimes_{\id\otimes\lt,\png}G$$
$$\cong B\otimes \big(C_0(G)\rtimes_{\lt,\png}G\big)\cong B\otimes \K(L^2(G)).$$
In fact, for the universal norm $\png=\ung$ this is well-known
(see \cite[Lemma 7.16]{Williams:Crossed}). But since the $G$-action on $\contz(G,B)$ is centrally proper, all crossed-product norms coincide for this action by Theorem~\ref{thm-centrally-proper}, from which the observation follows. The isomorphism~\eqref{eq-greencompacts} now follows directly from Theorem~\ref{theo:TheFixedPointFunctorForCommutingActions}:
\begin{align*}
 \Ind_H^G(B,\beta)\rtimes_{\Ind\beta,\png}G & \cong \big(C_0(G, B)\rtimes_{\lt\otimes \id,\png}G\big)^{H,\rt\otimes \beta}_\pnh\\
 &\cong \big(B\otimes \K(L^2(G))\big)^{H,\beta\otimes \Ad\rho}_\pnh.\qedhere
 \end{align*}
\end{proof}

As a first application we want to obtain a quite general version of Green's imprimitivity theorem for
classical and exotic crossed products. To prepare for the result, assume that $H$ is a closed subgroup of $G$
and let $\beta:H\to \Aut(B)$ be an action. Consider
$\contc(G,B)$ as a {pre-Hilbert} $\contc(G,\Ind_H^GB)-\contc(H,B)$ bimodule with
module actions and inner products
 given by the formulas
 \begin{equation}\label{eq-Green}
\begin{split}
\xi\cdot\varphi(t)&=\int_H\Delta_H(h)^{-1/2}\Delta_G(h)^{1/2}\beta_h(\xi(th) \cdot \varphi(h^{-1}))\dd{h}\\
f\cdot \xi(t)&=\int_G f(s,t)\xi(s^{-1}t)\dd{s}\\
\bbraket{\xi}{\eta}_{\contc(H,B)}(h)
&=\Delta_H(h)^{-1/2}\Delta_G(h)^{1/2}\int_G{\xi(s)^*}\beta_h(\eta(sh))\dd{s}\quad\text{and}\\
_{\contc(G, \Ind B)}\bbraket{\xi}{\eta}(t,r)&=
\int_H \Delta_G(t^{-1}rh) \beta_h\big(\xi(rh){\eta(t^{-1}rh)^*}\big)\dd{h},
\end{split}
\end{equation}
for $\xi,\eta \in \contc(G,B)$, $\varphi\in \contc(H,B)$ and $f\in \contc(G,\Ind_H^GB)$.

\begin{theorem}[Green's imprimitivity theorem for exotic norms. Version 1]\label{thm-Green}
Suppose that $(\png,\pn, \pnh)$ is  a compatible triple of norms on $G, G\times H$ and $H$ as in Definition \ref{def-compatible-action}.
Then the module actions and inner products on $C_c(G,B)$ as given in (\ref{eq-Green}) extend to give
a $\Ind_H^G(B,\beta)\rtimes_{\Ind\beta,\png}G - B\rtimes_{\beta,\pnh}H$ imprimitivity
bimodule.
\end{theorem}
\begin{proof}
Let $B\otimes L^2(G)$ be the weakly proper $(B, G\rtimes H)$-module with
$H$\nb-action $\beta\otimes \rho$ and structure map $1\otimes M: C_0(G)\to \Lb(B\otimes L^2(G))$.
Then it follows from Proposition~\ref{prop-exotic-compacts} that
$\F_\pnh^H(B\otimes L^2(G))$ is a $\K(B\otimes L^2(G))_{\pnh}^H-B\rtimes_{\beta,\pnh}H$  imprimitivity
bimodule, and combining this with the isomorphism $\K(B\otimes L^2(G))_{\pnh}^H\cong \Ind_H^G(B,\beta)\rtimes_{\Ind\beta,\png}G$
given by  (\ref{eq-greencompacts}) we see that $\F_\pnh^H(B\otimes L^2(G))$ becomes an
$\Ind_H^G(B,\beta)\rtimes_{\Ind\beta,\png}G-B\rtimes_{\beta,\pnh}H$ imprimitivity bimodule.

The subspace $\contc(G,B)\subseteq L^2(G,B)\cong B\otimes L^2(G)$ is dense in
$\F_c(L^2(G,B))=\contc(G)\cdot L^2(G,B)$ with respect to the inductive limit topology and hence it is norm dense in $\F_\pnh^H(L^2(G,B))$.
The result then follows if we can show that the module actions and inner products
 are given on $C_c(G,B)$ and the dense
subalgebras $\contc(H,B)$ and $\contc(G,\Ind B)$ of $B\rtimes_{\beta,\pnh} H$ and
$\Ind_H^G(B,\beta)\rtimes_{\Ind\beta,\png}G$ by (\ref{eq-Green}).
Indeed, the first
three formulas follow directly from the corresponding formulas for the $\contc(G, C_0(G,B))- B$ submodule
$\contc(G,B)$ of the $C_0(G,B)\rtimes_{\lt\otimes \id}G-B$ equivalence bimodule $L^2(G,B)$ together
with (\ref{eq:InnerProduct}) and (\ref{eq:ModuleConvolution}), and the last formula follows from the
requirement that $_{\contc(G, \Ind B)}\bbraket{\xi}{\eta}\cdot\zeta=\xi\cdot \bbraket{\eta}{\zeta}_{\contc(H,B)}$
for all $\xi,\eta,\zeta\in \contc(G,B)$.
\end{proof}

\begin{remark}\label{rem-green-formulas} We should note that the formulas given in (\ref{eq-Green}) correspond to the formulas
given in  \cite[Appendix B.1]{Echterhoff-Kaliszewski-Quigg-Raeburn:Categorical} via the transformation
 $\xi\mapsto \Delta_G^{-1/2}\xi$ on $C_c(G,B)$. This follows from some straightforward computations which we omit.
\end{remark}

Applying Theorem~\ref{thm-Green} to the triples $(\ung,\un,\unh)$ and $(\redg,\red,\redh)$, we
recover the well-known versions of Green's imprimitivity theorem for
universal and reduced crossed products.
{Moreover,} we shall see below that Green's theorem also applies
for arbitrary exotic crossed-product norms $\mu^G$ for $G$ corresponding to a weak-* closed $G$\nb-invariant ideal
$E\subseteq B(G)$
in the sense of \cite{Kaliszewski-Landstad-Quigg:Exotic} if we let
 $\mu^H$ be the crossed-product norm for $H$ corresponding to the
 weak-* closed ideal $E_H\subseteq B(H)$ generated by $\{f|_H: f\in E\}$, although we do not know whether
 the pair $(\png,\pnh)$ always fits into a compatible triple of norms as in Definition \ref{def-compatible-action}.
For the proof we make use of an adaptation of
\cite[Theorem 4.11]{Echterhoff-Kaliszewski-Quigg-Raeburn:Categorical} to universal crossed products.

Recall from \cite{Echterhoff-Kaliszewski-Quigg-Raeburn:Categorical}*{Lemma~3.19} that if $\delta: A\to \M(A\otimes C^*(H))$ is a coaction of a closed subgroup $H$ of a locally compact group $G$ and if $\iota_H:C^*(H)\to \M(C^*(G))$ denotes the canonical homomorphism, then $\delta$ can be inflated to a coaction of $G$ by defining
$$\Inf\delta:= (\id_B\otimes \iota_H) \circ \delta: B\to \M(B\otimes C^*(G)).$$
We then get

\begin{theorem}[Green's imprimitivity theorem for exotic norms. Version 2]\label{thm-equiv-dual coactions}
Suppose that $\beta:H\to\Aut(B)$ is an action of the closed subgroup $H$ of $G$ and let ${\X}_\un$ denote
Green's $\Ind_H^G(B,\beta)\rtimes_{\beta,\ung}G-B\rtimes_{\beta,\unh}H$ imprimitivity bimodule of Theorem~\ref{thm-Green}.
Then there exists a coaction
$\delta_{{\X}_\un}: {\X}_\un\to \M({\X}_\un\otimes C^*(G))$  which implements a
Morita equivalence between the dual coaction $\widehat{\Ind\beta}$ of $G$ on $\Ind_H^G(B,\beta)\rtimes_{\Ind\beta,\ung}G$
and the inflation $\Inf\widehat{\beta}$ of the dual coaction $\widehat\beta$ of $H$ on $B\rtimes_{\beta,\unh}H$.

Moreover, let $\mu^G$ and $\mu^H$ denote the crossed-product norms
corresponding to the weak\nb-* closed ideal $E\subseteq B(G)$
and the weak-* closed ideal $E_H\subseteq B(H)$ generated by $E|_H\defeq \{f|_H:f\in E\}$, respectively.
Then there is a unique quotient module ${\X}_\pn$ of
${\X}_\un$ such that $\delta_{{\X}_\un}$ factors through a coaction
$\delta_{X_\pn}:{\X}_\pn\to \M({\X}_\pn\otimes C^*(G))$ which
implements a Morita equivalence between the coactions
$(\widehat{\Ind\beta})_\png$ and $\Inf(\widehat{\beta}_\pnh)$ of $G$ on
$\Ind_H^G(B,\beta)\rtimes_{\Ind\beta,\png}G$ and $B\rtimes_{\beta,\pnh}H$, respectively. In particular,
Green's theorem holds for the pair $(\png,\pnh)$ of exotic crossed-product norms.
\end{theorem}
\begin{proof} By Remark \ref{rem-green-formulas} we may use the formulas for the actions and inner products as
given in \cite[Chapter 6]{Echterhoff-Kaliszewski-Quigg-Raeburn:Categorical}. We are then in the
setting of \cite[Theorem 4.11]{Echterhoff-Kaliszewski-Quigg-Raeburn:Categorical}, which gives our theorem
in the case of the reduced crossed-product norms  $(\redg,\redh)$. Indeed, all constructions and computations are
done in \cite{Echterhoff-Kaliszewski-Quigg-Raeburn:Categorical} on the level of functions with compact supports on
$G$ with values in various (multiplier) algebras.
Since these spaces embed  into the universal crossed product and module completions as well as for the reduced ones, we may go through the
proof of \cite[Theorem 4.11]{Echterhoff-Kaliszewski-Quigg-Raeburn:Categorical}  step by step to see that precisely the same
arguments give the desired coaction $\delta_{{\X}_\un}: {\X}_\un\to \M({\X}_\un\otimes C^*(G))$ of our theorem.

Assume now that $\png$ is a crossed-product norm corresponding to a weak-* closed ideal $E\subseteq B(G)$ in the sense of
\cite{Kaliszewski-Landstad-Quigg:Exotic} and let $C_E^*(G)$ denote the corresponding exotic group {\cstar{}}algebra.
Let
$$L({\X}_\un)=\left(\begin{matrix} \Ind_H^G(B,\beta)\rtimes_{\Ind\beta,\ung}G & {\X}_\un\\ {\X}_\un^*& B\rtimes_{\beta,\unh}H\end{matrix}\right)$$
denote the linking algebra for ${\X}_\un$. Then there is a unique coaction
$$\delta_{L({\X}_\un)}:=\left(\begin{smallmatrix} \widehat{\Ind\beta}& \delta_{{\X}_\un}\\ \delta_{{\X}_\un^*} &\Inf\widehat{\beta}\end{smallmatrix}\right):
L({\X}_\un)\to {\M(}L({\X}_\un)\otimes C^*(G){)}\cong {\M(}L({\X}_\un\otimes C^*(G)){)}$$
of $G$ on $L({\X}_\un)$ (see \cite[Chapter 2.5]{Echterhoff-Kaliszewski-Quigg-Raeburn:Categorical}).
The intersection of the kernel $I_{L({\X}_\un)}$ of the composition
$$(\id_{L({\X}_\un)}\otimes q_E)\circ \delta_{L({\X}_\un)}:L({\X}_\un)\to \M(L({\X}_\un)\otimes C_E^*(G))\cong
\M(L({\X}_\un\otimes C_E^*(G)))$$
with the corners $\Ind_H^G(B,\beta)\rtimes_{\Ind\beta,\ung}G$ and $B\rtimes_{\beta,\unh}H$, respectively, are given by
$J_E:=\ker( \id_{\Ind B\rtimes G}\otimes q_E)\circ \widehat{\Ind\beta}$ and the ideal $I_E:=\ker(\id_{B\rtimes H}\otimes q_E)\circ \Inf\widehat\beta$.
It follows  from  \cite[Lemma 1.20 and Lemma 1.52]{Echterhoff-Kaliszewski-Quigg-Raeburn:Categorical}
that the kernel ${\Y}:=\ker(\id_{{\X}_\un}\otimes q_E)\circ \delta_{{\X}_\un}\subseteq {\X}_\un$ matches
$J_E$ and $I_E$ under the Rieffel correspondence, and hence ${\X}_\un$ factors to give a
$\Ind_H^G(B,\beta)\rtimes_{\Ind\beta, \png}G- B\rtimes_{\beta,\pnh}H$ imprimitivity bimodule
 ${\X}_\pn:= {\X}_\un/{\Y}$ as soon as we can show that $B\rtimes_{\beta,\pnh}H= (B\rtimes_{\beta,\unh}H)/I_E$.

 To see this we first observe that the annihilator $^\perp (E|_H)$ of $E|_H\subset B(H)$ in $C^*(H)$ is precisely the
 kernel of the composition $q_E\circ \iota_H: C^*(H)\to \M(C_E^*(G))$. This follows  from the fact that $E$ consists of
 all coefficient functions of unitary representations of $G$ which annihilate $\ker q_E$. The intersection of  kernels in $C^*(H)$
 of the restrictions of  these representations to $H$ coincides with $\iota_H^{-1}(\ker q_E)= \ker(q_E\circ \iota_H)$.
The equation $B\rtimes_{\beta,\pnh}H= (B\rtimes_{\beta,\unh}H)/I_E$ follows then from
 Proposition~\ref{prop-crossed} applied to the quotient $C_\nu^*(H):=C^*(H)/\ker(q_E\circ \iota_H)$.

 We finally check that  $\delta_{{\X}_\un}$ factors through a coaction
$\delta_{\X_\pn}$ on the quotient module ${\X}_\pn$ which
implements a Morita equivalence between the coactions
$(\widehat{\Ind\beta})_\png$ and $\Inf(\widehat{\beta}_\pnh)$, which exist by \cite{Kaliszewski-Landstad-Quigg:Exotic}*{Theorem~6.2}.
For this
let $q_{L({\X})}: L({\X}_\un)\to L({\X}_\un)/I_{L({\X}_\un)}\cong
L({\X}_\un/{\Y})$ denote the quotient map. Then the intersection of the kernel of the composition
$$(q_{L({\X})}\otimes \id_G)\circ  \delta_{L({\X}_\un)}: L({\X}_\un)\to \M(L({\X}_\un/\Y)\otimes C^*(G))$$
with the upper left corner $\Ind_H^G(B,\beta)\rtimes_{\Ind\beta,\ung}G$ coincides with $J_{\Ind\beta,E}$, since we know from
\cite[Theorem 6.2]{Kaliszewski-Landstad-Quigg:Exotic} that $\widehat{\Ind\beta}$ factors through an injective coaction
on $\Ind_H^G(B,\beta)\rtimes_{\Ind\beta,\png}G=(\Ind_H^G(B,\beta)\rtimes_{\Ind\beta,\ung}G)/J_{\Ind\beta,E}$.
Since $(\id_{{\X}_\un}\otimes q_E)\circ \delta_{{\X}_\un}$ is an imprimitivity bimodule homomorphism,
the intersection of the kernel of $(q_{L({\X})}\otimes \id_G)\circ  \delta_{L({\X}_\un)}$ with the other corners
correspond to $J_{\Ind\beta,E}$ by the Rieffel correspondence (use again \cite[Lemma 1.20 and Lemma 1.52]{Echterhoff-Kaliszewski-Quigg-Raeburn:Categorical}), hence they must coincide with ${\Y}$ and $I_E$, respectively.
Thus we see that the kernel of $(q_{L({\X})}\otimes \id_G)\circ  \delta_{L({\X}_\un)}$
coincides with $I_{L({\X}_\un)}=\ker\big((\id_{L({\X}_\un)}\otimes q_E)\circ \delta_{L({\X}_\un)}\big)$.
It follows that $(q_{L({\X})}\otimes \id_G)\circ  \delta_{L({\X}_\un)}$ factors through a coaction on
$L({\X}_\un)/I_{L({\X}_\un)}\cong
L({\X}_\un/\Y)$ whose restriction to the upper right corner gives a coaction on ${\X}_\pn:= {\X}_\un/{\Y}$
which implements the desired Morita equivalence between $(\widehat{\Ind\beta})_\png$ and $\Inf(\widehat{\beta}_\pnh)$.
\end{proof}

The following corollary shows that our Version 1 of Green's theorem is actually a special case of Version 2. It also indicates
that one cannot expect very much freedom for independent choices of norms in compatible triples $(\png,\pn,\pnh)$.

\begin{corollary}\label{cor-subgroup-triple} Suppose that $H$ is a closed subgroup of $G$ and that $(\png,\pn,\pnh)$ is a compatible triple
of norms for $G, G\times H$, and $H$ as in Definition \ref{def-compatible-action}. If $\png$ is associated to the
weak-* closed $G$-invariant ideal $E\subseteq B(G)$, then $\pnh$ is associated to the weak-* closed ideal $E_H$ in $B(H)$
generated by $E|_H=\{f|_H:f\in E\}$.
\end{corollary}
\begin{proof} Let $\beta:H\to \Aut(B)$ be any action. Then it follows from Theorems \ref{thm-Green} and \ref{thm-equiv-dual coactions}
that  both crossed products $B\rtimes_{\beta,\pnh}H$ and
$B\rtimes_{\beta,\pn_{E_H}}H$  correspond to the crossed product $\Ind_H^G(B,\beta)\rtimes_{\Ind\beta,\png}G$
under the Rieffel correspondence for Green's $\Ind_H^G(B,\beta)\rtimes_{\Ind\beta,\ung}G - B\rtimes_{\beta,\unh}H$ imprimitivity
bimodule $\X_{\un}$. Hence they must be the same.
\end{proof}

\subsection{The symmetric imprimitivity theorem}
We are now going to derive a very general version of the symmetric imprimitivity theorem,
which has been first obtained for commuting free and proper actions of two groups $G$ and $H$ on a space $X$
by Green and Rieffel (\eg, see \cite{Rieffel:Applications_Morita}).
This result was extended by Raeburn \cite{Raeburn:Induced_Symmetric}, Kasparov \cite[\S 3]{Kasparov:Novikov}, and
Quigg-Spielberg \cite{Quigg-Spielberg} to commuting actions of $G$ and $H$ such that there exists a
$G\times H$-space $X$ on which $G$ and $H$ act freely and properly
together with a  structure map $\phi:C_0(X)\to \M(A)$ which takes its values in the {\em center} $Z\M(A)$ of $\M(A)$.
A version of the symmetric imprimitivity theorem for  saturated
Rieffel proper actions
is given by an Huef, Raeburn, and Williams  in \cite{anHuef-Raeburn-Williams:Symmetric}. Although the result in \cite{anHuef-Raeburn-Williams:Symmetric} was
formulated in principle for a more general class of actions (with a number of extra technical assumptions), the
 only class of examples where it was explicitly shown to apply is given by
weak $X\rtimes (G\times H)$-algebras $(A,\alpha)$ in our sense, such that $G$ and $H$ both act
freely and properly on $X$. Also, the results in \cite{anHuef-Raeburn-Williams:Symmetric} are restricted to
the case of reduced fixed-point algebras and crossed products.
Below we shall give a version of the symmetric imprimitivity theorem for such algebras
which also works for  universal crossed products and for exotic crossed products
for any compatible triple $(\png,\pn,\pnh)$ in the sense of Definition \ref{def-compatible-action}.
Since we assume the presence
of a $G\times H$-equivariant structure map $\phi:C_0(X)\to \M(A)$ from the beginning (without assuming the actions of $G$ and $H$ to be free),
our proof and the derived formulas turn out to be much easier and less technical than the ones
obtained in \cite{anHuef-Raeburn-Williams:Symmetric}. Allowing non-free actions and exotic norms, our results
also extend the results of Raeburn, Kasparov and Spielberg-Quigg on {\em centrally proper} commuting actions.

Recall from Propositions \ref{prop:PreHilbertModule} and \ref{prop-exotic-compacts} that whenever we have
a weakly proper $X\rtimes G$-algebra $(A,\alpha)$, the module
$\F_\png^G(A)$ becomes a {\em partial} $A_\png^G-A\rtimes_{\alpha,\png}G$
equivalence bimodule. The word partial indicates that the $A\rtimes_{\alpha,\png}G$-valued inner product might not be full.
It is always full if the action of $G$ on $A$ is $\png$-saturated as in Definition~\ref{def-saturated}, which is true if
 $G$ acts freely on $X$.
Now, if $(A,\alpha)$ is a weak $X\rtimes (G\times H)$-algebra such that $G$ acts properly on $X$, we
obtain the partial $A_\png^G\rtimes_{\alpha^H,\pnh}H-(A\rtimes_{\alpha^G,\png}G)\rtimes_{\alpha^H,\pnh} H$ equivalence bimodule
$\F_\png^G(A)\rtimes_\pnh H$
as above by passing to the $H$\nb-descent. This is an equivalence bimodule, if (and only if) $\F_\png^G(A)$
is one. If also the action of $H$ on $X$ is proper, we may reverse the role of $G$ and $H$ to
obtain a partial $A_\pnh^H\rtimes_{\alpha^G,\png}G-(A\rtimes_{\alpha^H,\pnh}H)\rtimes_{\alpha^G,\png} G$ equivalence bimodule
$\F_\pnh^H(A)\rtimes_\png G$, which is an equivalence bimodule
if the action of $H$ on $A$ is $\pnh$-saturated.
Using the canonical isomorphism
$$(A\rtimes_{\alpha^G,\png}G)\rtimes_{\alpha^H,\pnh} H\cong A\rtimes_{\beta,\pn}(G\times H)\cong
 (A\rtimes_{\alpha^H,\pnh}H)\rtimes_{\alpha^G,\png} G,$$
which follows from our compatibility assumption for the given norms $(\png,\pn, \pnh)$, we may
form the product
$$\F_\pn(A,G,H):=\left(\F_\png^G(A)\rtimes_\pnh H\right)\otimes_{ A\rtimes_{\beta,\pn}(G\times H)}\left(\F_\pnh^H(A)\rtimes_\png G\right)^*$$
which is a partial $A_\png^G\rtimes_{\alpha^H,\pnh}H-A_\pnh^H\rtimes_{\alpha^G,\png}G$ equivalence bimodule.
If both actions of $G$ and $H$ on $A$ are saturated, which holds in particular if both actions of $G$ and $H$
on $X$ are free, then this will become a full equivalence bimodule. Thus we have shown:

\begin{theorem}[Symmetric Imprimitivity Theorem. Version 1]\label{thm-symmetric}
Let $X$ be a $G\times H$-space with commuting proper $G$ and $H$\nb-actions and assume that $(\png,\pn,\pnh)$ is a compatible triple
of norms as in Definition \ref{def-compatible-action}
and let  $(A,\alpha)$ be a weak $X\rtimes(G\times H)$-algebra.
Then $\F_\pn(A,G,H)$ constructed above is a partial $A_\png^G\rtimes_{\alpha^H,\pnh}H-A_\pnh^H\rtimes_{\alpha^G,\png}G$ equivalence bimodule.
Moreover, if both actions of $G$ and $H$ on $A$ are $\png$- (resp. $\pnh$-) saturated, then $\F_\pn(A,G,H)$ becomes  a Morita equivalence
$$A^G_\png\rtimes_{\alpha^H,\pnh} H\sim_M A^H_\pnh\rtimes_{\alpha^G,\png} G.$$
\end{theorem}

As a consequence of the previous theorem and the isomorphisms $(A\rtimes_\png G)^H_\pnh\cong A^H_\pnh\rtimes_\png G$ and $A^G_\png\rtimes_\pnh H\cong (A\rtimes_\pnh H)^G_\png$ from Theorem~\ref{theo:TheFixedPointFunctorForCommutingActions}, we immediately get the following ``twisted symmetric imprimitivity theorem'' in which the roles of fixed-point algebras and crossed products are interchanged in comparison with the usual version of that theorem:

\begin{corollary}
Let $(A,\alpha)$ be a weak $X\rtimes (G\times H)$-algebra such that the actions of $G$ and $H$ on $X$ are proper
and the actions of $G$ and $H$ on $A$ are $\png$- (resp. $\pnh$-) saturated. Then
$$(A\rtimes_{\alpha^H,\pnh} H)^G_\png\sim_M (A\rtimes_{\alpha^G,\png} G)^H_\pnh.$$
\end{corollary}

\begin{remark}
 It is interesting to observe that, even if we start with a $X\rtimes (G\times H)$-algebra $A$ such that
the structure map $\phi:C_0(X)\to Z\M(A)$ takes values in the center of $\M(A)$ and such that
$G$ and $H$ act freely and properly on $X$, which is the situation considered by Raeburn
and Kasparov in \cite{Raeburn:Induced_Symmetric} and \cite{Kasparov:Novikov}, the above corollary does not make sense
outside the setting of weakly proper algebras since the canonical
structure maps from $C_0(X)$ into $\M(A\rtimes_{\alpha^G, \png}G)$ or $\M(A\rtimes_{\alpha^H, \pnh}H)$
will rarely take values in the centers of these multiplier algebras.
\end{remark}

If one of the groups, say $H$, is amenable, we may combine Theorem~\ref{thm-symmetric} with Proposition~\ref{prop-compatible}
to obtain

\begin{corollary}\label{cor-symmetric-amenable}
Let $(A,\alpha)$ be a weak $X\rtimes (G\times H)$-algebra such that $G$ and $H$ act properly on $X$.
Assume further that $H$ is amenable. Then for any fixed $E$\nb-crossed-product norm $\png$ for $G$ we
obtain a partial $A_\png^G\rtimes_{\alpha^H}H-A^H\rtimes_{\alpha^G,\png}G$ equivalence bimodule $\F_\pn(A,G,H)$.
Moreover, if the actions of $G$ and $H$ on $A$ are $\png$- (resp. $\unh$-) saturated, then we obtain Morita equivalences
$$A_\png^G\rtimes_{\alpha^H}H \sim_M A^H\rtimes_{\alpha^G,\png}G\quad\text{and}\quad
(A\rtimes_{\alpha^H} H)^G_\png\sim_M (A\rtimes_{\alpha^G,\png} G)^H.$$
\end{corollary}

In what follows next, we want to obtain a more direct description of the \linebreak
${A^G_\png\rtimes_{\alpha^H,\pnh} H} - A^H_\pnh\rtimes_{\alpha^G,\png} G$ bimodule
$\F_\pn(A,G, H)$.
For this we need the following lemma.

\begin{lemma}\label{proper-product actions}
Suppose that $(B,\beta)$ is a $G\times H$-algebra, $X$ is a proper $G\times H$-space and $(\E,\gamma)$
is a weak $(B, X\rtimes (G\times H))$-module. Moreover, let $(A,\alpha):=(\K(\E), \Ad\gamma)$ and let $(\png,\pn,\pnh)$
be any compatible triple of norms as in Definition \ref{def-compatible-action}.
Then $(\F_\png^G(\E), \gamma^H)$ is a  weakly
proper $(B\rtimes_{{\beta}^G,\png}G, G\bs X\rtimes H)$-module  with $H$\nb-action $\gamma^H$ induced from the given $H$\nb-action $\gamma|_H$
on $\F_c^G(\E)\subseteq \E$,
and we get isomorphisms
$$\F_\pnh^H(\F_\png^G(\E))\cong \F_\pn^{G\times H}(\E)\quad\text{and}\quad (A_\png^G)_\pnh^H\cong A_\pn^{G\times H}.$$
The first isomorphism is given by the identification
$\F_c^H(\F_c^G(\E))= \F_c^{H\times G}(\E)$ and the second isomorphism is given by the
canonical inclusion $A_c^{H\times G}\into (A_\png^G)_c^H$.
\end{lemma}
 \begin{proof}
 Observe that $\F^H_c(\F^G_c(\E))=\contc(G\bs X)\cdot \contc(X)\cdot \E=\contc(X)\cdot \E=\F_c^{H\times G}(\E)$, since $\contc(G\bs X)\cdot \contc(X)=\contc(X)$.
 Now, if $\xi,\eta\in \contc(X)\cdot \E$, we have
 \begin{align*}
 \bbraket{\xi}{\eta}_{\contc(H\times G, B)}(h,s)&=\Delta_H(h)^{-1/2}\Delta_G(s)^{-1/2} \braket{\xi}{\gamma_{(h,s)}(\eta)}_B\\
 &=\big(\Delta_H(h)^{-1/2}\bbraket{\xi}{\gamma_h(\eta)}_{\contc(G,B)})(s)\\
 &=\bbraket{\xi}{\eta}_{\contc(H, \contc(G,B))}(h,s)
 \end{align*}
 which shows that the imbedding $\F_c^{H\times G}(\E)\into \F_c^H(\F_c^G(\E))$ is isometric.
 A similar computation also shows that it preserves the right module structure.
 Since $\F_c^G(\E)$ is norm dense in $\F_\png^G(\E)$, it follows that $\F_c^H(\F_c^G(\E))$ is inductive limit
 dense in $\F_c^H(\F_\png^G(\E))$ and the isomorphism $\F_\pnh^H(\F_\png^G(\E))\cong \F_\pn^{G\times H}(\E)$ follows from
  the isomorphism $(B\rtimes_{\beta^H,\pnh}H)\rtimes_{\alpha^G,\png}G
 \cong B\rtimes_{\alpha,\pn}(G\times H)$ of the coefficient algebras.
 Finally, using Proposition~\ref{prop-exotic-compacts}, we get
 $A_\png^G=\K(\E)_\png^G=\K(\F_\png^G(\E))$  and  then
\begin{align*}
 (A_\png^G)_\pnh^H&=\K(\F_\png^G(\E))_\pnh^H=\K(\F_\pnh^H(\F_\png^G(\E)))\\
 &=\K(\F_\pn^{G\times H}(\E))
 =\K(\E)_\pn^{G\times H}\\
 &= A_\pn^{G\times H}.\qedhere
 \end{align*}
 \end{proof}

We now want to obtain a more concrete version of our symmetric
imprimitivity theorem with precise formulas for the actions and inner products  on the
equivalence bimodule
$$
\F_\pn(A,G,H)=\left(\F_\png^G(A)\rtimes_\pnh H\right)\otimes_{ A\rtimes_{\beta,\pn}(G\times H)}\left(\F_\pnh^H(A)\rtimes_\png G\right)^*.
$$
So in what follows let us assume that $(A,\alpha)$ is a weak $X\rtimes (G\times H)$\nb-algebra
such that $G$ and $H$ both  act properly on $X$, and let $(\png,\pn,\pnh)$ be a compatible triple of norms.

Recall from Theorem~\ref{theo:TheFixedPointFunctorForCommutingActions}, applied to the Hilbert $(A, X\rtimes(G\times H))$-module $(A,\alpha)$, that $\F_\png^G(A)\rtimes_{\alpha^H,\pnh} H\cong \F_\png^G(A\rtimes_{\alpha^H\pnh} H)$ with isomorphism given by the identity map on $\contc(H, \F_c^G(A))$, which is a submodule of both modules. On the other hand, if  we view $L^2(H)\otimes A$ as a weakly proper $(A, (H\times X)\rtimes H)$-module with $H$\nb-action $\rho_H\otimes \alpha$ and structure map $M_H\otimes \phi\colon C_0(H\times X)\to \Lb(L^2(H)\otimes A)$, it follows from Proposition~\ref{prop-independence} and Proposition~\ref{prop-L2} that $A\rtimes_{\alpha^H,\pnh} H\cong \F_\pnh^H(L^2(H)\otimes A)$. If we let $G$ act on $L^2(H)\otimes A$ via $\id_{L^2(H)}\otimes \alpha$, we
get commuting actions of $H$ and $G$ on $L^2(H)\otimes A$ such that this module together with the structure map
$M_H\otimes \phi$ becomes a weakly proper $(A, (H\times X)\rtimes (H\times G))$-module. Together with Lemma~\ref{proper-product actions} we get
\begin{align*}
 \F_\png^G(A)\rtimes_{\alpha^H,\pnh} H&\cong \F_\png^G(\E\rtimes_{\alpha^H,\pnh} H)\cong \F_\png^G(\F_\pnh^H( L^2(H)\otimes A))\\
 &\cong \F_\pn^{G\times H}(L^2(H)\otimes A).
 \end{align*}

 Similarly, changing the role of $G$ and $H$ in the above discussion, and viewing
 $L^2(G)\otimes A$ as a weakly proper $(A, (G\times X)\rtimes (H\times G))$-module
 with action $\rho_G\otimes \alpha$ and structure map $M_G\otimes\phi$, we get an isomorphism
 $$ \F_\pnh^H(A)\rtimes_{\alpha^G,\png} G\cong \F_\pn^{G\times H}(L^2(G)\otimes A).$$
 By Proposition~\ref{prop-product}, we then get
 \begin{align*}
 \F_\pn(A,G,H)&=\left(\F_\png^G(A)\rtimes_\pnh H\right)\otimes_{A\rtimes_\pn(G\times H)}\left(\F_\pnh^H(A)\rtimes_\png G\right)^*\\
 &\cong \F_\pn^{G\times H}(L^2(H)\otimes A)\otimes_{A\rtimes_\pn(G\times H)} \F_\pn^{G\times H}(L^2(G)\otimes A)^*\\
 &\stackrel{(*)}{\cong} \big(( L^2(H)\otimes A)\otimes_A(L^2(G)\otimes A)^*\big)^{G\times H}_\pn\\
 &\stackrel{(**)}{\cong} \big(L^2(H)\otimes L^2(G)^*\otimes A\big)^{G\times H}_\pn.
 \end{align*}
 The isomorphism $(**)$ in the last line is given by multiplication in $A$, and the isomorphism $(*)$ is given
 as in Proposition~\ref{prop-product}.

 Next, we give a detailed description of  the  module
 $\big(L^2(H)\otimes L^2(G)^*\otimes A\big)^{G\times H}_\pn$. For this we first
  recall the structure of the
 $\K(L^2(H))\otimes A- \K(L^2(G))\otimes A$-module $L^2(H)\otimes L^2(G)^*\otimes A$.
 If we identify $\K(L^2(H))\otimes A$ with $C_0(H,A)\rtimes_{\tau\otimes\alpha}H$
 and $\K(L^2(G))\otimes A$ with $C_0(G,A)\rtimes_{\tau\otimes\alpha}G$
 as in Remark~\ref{left-action}, and if we restrict our attention to the dense submodule
 $$\contc(G\times H, A_c) :=\contc(X)\cdot \contc(G\times H, A)\cdot \contc(X),$$
 we get a (partial) $\contc(H, \contc(H, A_c))- \contc(G, \contc(G, A_c))$ pre-equivalence bimodule structure
 with actions and inner products given by
 \begin{equation}\label{eq-module-formulas}
 \begin{split}
 (f\cdot \xi)(s,h)&=\int_H\alpha_{h^{-1}}(f(l,h))\xi(s, l^{-1}h)\dd{l}\\
 (\xi\cdot \psi)(s,h)&=\int_G \xi(ts,h)\alpha_{t^{-1}s^{-1}}(\psi(t, ts)\big)\dd{t}\\
 {_{\contc(H, \contc(H, A))}\braket{\xi}{\eta}}(h,l)&=
 \Delta_H(h^{-1}l)\int_G \alpha_l\big(\xi(s,l)\eta(s,h^{-1}l)^*\big)\dd{s}\\
 \braket{\xi}{\eta}_{\contc(G, \contc(G, A))}(s, r)&=
\Delta_G(s^{-1}r) \int_H \alpha_r\big(\xi(r, l)^* \eta(s^{-1}r,l)\big)\dd{l}
\end{split}
\end{equation}
for $f\in \contc(H, \contc(H, A_c)), \psi\in \contc(G, \contc(G, A_c))$ and
$\xi,\eta\in \contc(G\times H, A_c)$. The formulas follow from the
isomorphism
$$( L^2(H)\otimes A)\otimes_A(L^2(G)\otimes A)^*\cong L^2(H)\otimes L^2(G)^*\otimes A$$
given by multiplication in $A$ together with the
formulas for the $\contc(H, \contc(H,A))-A$ module $\contc(H,A)\subseteq L^2(H)\otimes A$
and the $A- \contc(G, \contc(G,A))$-module $\contc(G,A)\subseteq (L^2(G)\otimes A)^*$
(see Remark~\ref{left-action} for $\contc(H,A)$ and use the general procedure for passing from
a bimodule $\E$ to its adjoint $\E^*$ for the second).

We now study the projections into the fixed-point algebras and the fixed-point module
via the strict unconditional integrals over $G\times H$. Since $H$ acts on $\contc(H, \contc(H,A_c))$
 by right translation of the second variable
of $f(h,l)$ and since $G$ acts by $\alpha$, this
 will send
 a function $f\in \contc(H, \contc(H, A_c))$ on the left
to the function $\EE(f):H\times H\to A_c^G$ given by
\begin{equation}\label{eq-Ecoefficient}
\EE(f)(h,l)=\int_G^{st}\int_H \alpha_s(f(h,lk))\dd{k}\dd{s}.
\end{equation}
The translation $k\mapsto l^{-1}k$ in the integral shows that
this function is constant in $l$ and that $\EE(f)$ may be regarded
as a function $\EE(f)\in \contc(H, A_c^G)$ given by the formula
\begin{equation}\label{eq-Ecoefficient1}
\EE(f)(h)=\int_G^{st}\int_H\alpha_s(f(h,l))\dd{l}\dd{s}.
\end{equation}
Similarly, if $\psi\in \contc(G,\contc(G, A_c))$
we may regard $\EE(\psi)$ as a function in $\contc(G, A_c^H)$ given by
$$\EE(\psi)(t)=\int_H^{st}\int_G \alpha_{l}(\psi(t, s))\dd{s}\dd{l}.$$
 Moreover,
the map
$\EE\colon \contc(G\times H, A_c)\to (L^2(H)\otimes L^2(G)^*\otimes A)^{G\times H}_c$
sends $\xi\in \contc(G\times H, A_c)$ to a function $\EE(\xi):G\times H\to A_c$
 given by
\begin{equation}\label{eq-Emodule}
\EE(\xi)(t,h)=\int_G\int_H \Delta_H(l)^{1/2}\Delta_G(s)^{1/2}\alpha_{(s,l)}(\xi(ts,hl))\dd{l}\dd{s}
\end{equation}
which satisfies the equation $\EE(\xi)(t,h)=\Delta_G(t)^{-1/2}\Delta_H(h)^{-1/2}\alpha_{(t^{-1}, h^{-1})}(\EE(\xi)(e,e))$.
Thus every such function is uniquely determined by its value at the unit $(e,e)$ and we may regard
$\EE(\xi)$ as an element of $A_c$ given by
$$\EE(\xi)=\int_G\int_H \Delta_H(l)^{1/2}\Delta_G(s)^{1/2}\alpha_{(s,l)}(\xi(s,l))\dd{l}\dd{s}.$$
If we apply the formulas (\ref{eq-module-formulas}) to elements
$f\in \contc(H, A_c^G)$, $\psi\in \contc(G, A_c^H)$ and $\xi, \eta\in A_c$, regarding them
as functions on $H\times H$, $G\times G$ and $G\times H$ via the relations observed above,
we see that $A_c$ becomes a partial $\contc(H, A_c^G)-\contc(G, A_c^H)$ pre-equivalence bimodule
with inner products and module actions given by
 \begin{equation}\label{eq-module-formulas-fixed}
 \begin{split}
 f\cdot \xi&=\int_H\Delta_H(l)^{1/2} f(l)\alpha_l(\xi)\dd{l}\\
 \xi\cdot \psi&=\int_G\Delta_G(t)^{-1/2} \alpha_t(\xi \psi(t^{-1})\big)\dd{t}\\
 {_{\contc(H, A_c^G)}\bbraket{\xi}{\eta}}(h)&=
 \Delta_H(h)^{-1/2}\int_G^{st} \alpha_s\big(\xi\alpha_h(\eta^*)\big)\dd{s}\\
 \bbraket{\xi}{\eta}_{\contc(G, A_c^H)}(s)&=
\Delta_G(s)^{-1/2} \int_H^{st} \alpha_l\big(\xi^*\alpha_s(\eta)\big)\dd{l}
\end{split}
\end{equation}

At this point we should note that the correct continuity condition for a function $f\in \contc(H, A_c^G)$ (and similarly
for $\contc(G, A_c^H)$) is continuity with respect to the inductive limit topology on $A_c^G$
as defined in \cite[Definition 2.11]{Buss-Echterhoff:Exotic_GFPA}.
To be more precise, we define $C_c(H, A_c^G)$ as the set of all functions
$f: G\to A_c^G$ with compact supports such that the following
conditions are satisfied:
\begin{enumerate}
\item  there exists a function $\psi\in \contc(G\bs X)$ such that
$f(h)=\psi \cdot f(h)\cdot \psi$ for all $h\in H$ (which simply means that the elements $f(h)\in A_c^G$ have controlled
support in $G\bs X$), and
\item  for every $\varphi\in \contc(X)$ the functions $h\mapsto \varphi\cdot f(h), f(h)\cdot \varphi$
are norm-continuous with controlled compact supports in $X$.
\end{enumerate}
 As a sample computation, we show how we derive the last formula  in \eqref{eq-module-formulas-fixed}
from (\ref{eq-module-formulas}): using the derivation of (\ref{eq-Ecoefficient1}) from (\ref{eq-Ecoefficient})
we may regard
$ \bbraket{\xi}{\eta}_{\contc(G, A_c^H)}$ as a function $G\times G\to A_c^H$ which is constant in the second variable.
In the same way, by the discussion following
 (\ref{eq-Emodule}) we may regard $\xi,\eta\in A_c$ as functions on $G\times H$ such that
 $\xi(t,h)=\Delta_G(t)^{-1/2}\Delta_H(h)^{-1/2}\alpha_{(t^{-1}, h^{-1})}(\xi)$ and similarly for $\eta$.
 If we now apply the last formula in (\ref{eq-module-formulas}) to these functions, we obtain
 \begin{align*}
 \bbraket{\xi}{\eta}_{\contc(G, A_c^H)}(s)&=\bbraket{\xi}{\eta}_{\contc(G, A_c^H)}(s,e)\\
 &=\Delta_G(s^{-1}) \int_H^{st} \xi(e, l)^* \eta(s^{-1},l)\dd{l}\\
 &=\Delta_G(s^{-1}) \int_H^{st}\Delta_H(l^{-1})\Delta_G(s)^{1/2} \alpha_{l^{-1}}(\xi^* \alpha_s(\eta))\dd{l}\\
 &\stackrel{l\mapsto l^{-1}}{=}
 \Delta_G(s)^{-1/2} \int_H^{st}\alpha_{l}(\xi^* \alpha_s(\eta))\dd{l}.
 \end{align*}
 The other formulas follow by the same recipe. Note that the integrals make sense as strict unconditional integrals; and
 one can check that the appropriate pairings with elements in the original module
 $\contc(H\times G, A_c)\subseteq L^2(H)\otimes L^2(G)^*\otimes A$ and the corresponding subalgebras
 in the coefficient algebras of this module reveal that these formulas determine the unique extension
 of the module operations to the multiplier modules and multiplier algebras. As a result of this discussion we get

\begin{theorem}[Symmetric Imprimitivity Theorem. Version 2]\label{thm-symmetric-direct}
Suppose that $(A,\alpha)$ is a  weak $X\rtimes (G\times H)$-algebra such that
the actions of $G$ and $H$ on $X$ are proper and let $(\png,\pn,\pnh)$ be a compatible triple of norms as
in Definition \ref{def-compatible-action}. Then
the $\contc(H, A_c^G)-\contc(G, A_c^H)$ bimodule $A_c$ equipped with the
inner products and module actions as in (\ref{eq-module-formulas-fixed})
completes to give a partial $A_\png^G\rtimes_{\alpha^H,\pnh} H-A_\pnh^H\rtimes_{\alpha^G,\png}G$ equivalence bimodule $\F_\pn(A)_G^H$.
It will be a full equivalence bimodule if the actions of $H$ and $G$ on $A$ are  $\png$- (resp. $\pnh$-) saturated.
\end{theorem}
\begin{proof} The above discussion shows that the completion
$\F_\pn(A)_G^H$ of $A_c$ with respect to the actions and inner products
of (\ref{eq-module-formulas-fixed}) is isomorphic to
the module
$\left(\F_\pn^G(A)\rtimes_\pn H\right)\otimes_{A\rtimes_\pn(G\times H)}\left(\F_\pn^H(A)\rtimes_\pn G\right)^*$
of Theorem~\ref{thm-symmetric}.
\end{proof}

\begin{remark} We should remark that the version of  the Symmetric Imprimitivity Theorem proved by
Raeburn in \cite{Raeburn:Induced_Symmetric} is only for full crossed products. The reduced version was later
obtained by Quigg and Spielberg in \cite{Quigg-Spielberg} by some rather indirect methods.
Our proof does all cases simultaneously based on the isomorphisms
$$(A\rtimes_{\png}G)\rtimes_{\pnh}H\cong A\rtimes_{\pn}(G\times H)\cong (A\rtimes_{\pnh} H)\rtimes_{\png} G$$
for the iterated crossed products together with
the compatibility of several other modules with  $E$\nb-crossed products (\eg, see Lemma~\ref{lem-E-imp}).
However, we should remark that Kasparov  in \cite{Kasparov:Novikov} also obtained his version
of the full and reduced Symmetric Imprimitivity Theorem in one step.
\end{remark}

Observe that the discussion above also shows that $\F_\pn(A)_G^H$ is isomorphic to the fixed-point bimodule $\big(L^2(H)\otimes L^2(G)^*\otimes A\big)^{G\times H}_\pn$ under the identifications of $(A\otimes \K(L^2H))^{G\times H}_\pn\cong A_\png^G\rtimes_\pnh H$ and $(A\otimes \K(L^2(G)))^{G\times H}_\pn\cong A^H_\pnh\rtimes_\png G$.
Of course, if $H$ is the trivial group, this bimodule is also isomorphic to $\F_\png^G(A)$. Therefore, as a consequence we get the following:

\begin{corollary}
For any weakly proper $X\rtimes G$-algebra $(A,\alpha)$ and any $E$\nb-crossed-product norm $\pn$ for $G$, the partial $A^G_\pn-A\rtimes_{\alpha,\pn} G$ equivalence bimodule $\F_\pn^G(A)$ is isomorphic to the fixed-point bimodule $(A\otimes L^2(G)^*)^G_\pn$ under the  identification $(A\otimes \K(L^2(G)))^G_\pn\cong A\rtimes_{\alpha,\pn} G$ as given in  Remark~\ref{left-action}.
\end{corollary}

Taking duals (or taking $G$ to be the trivial group in the above discussion), it also follows that
the $\pn$-fixed-point bimodule of $A\otimes L^2(G)$ is isomorphic to the dual $\F_\pn^G(A)^*$ of $\F_\pn(A)$, \ie,
$$\big(A\otimes L^2(G)\big)^{G,\alpha\otimes\rho}_\pn\cong \F_\pn^G(A)^*.$$
Recall that, for $\pn=\red$, $\F_\red^G(A)^*$ is isomorphic to Rieffel's bimodule constructed in \cite{Rieffel:Proper} for $A$
viewed as a Rieffel proper algebra with respect to the dense subalgebra $A_c$. The above isomorphism therefore gives a new point of view for Rieffel's bimodule implementing the partial equivalence between $A\rtimes_{\alpha,\red}G$ and $A^G_\red$: it is just the reduced fixed-point bimodule of $L^2(G,A)\cong A\otimes L^2(G)$ with respect to the action $\alpha\otimes \rho$. Of course, our results only allow us to prove this for weakly proper algebras $A$, but  it could be true also for general Rieffel proper actions.

We close this section with some natural questions:
\begin{question}
 (1) If $H$ is amenable, we were able to show in Proposition~\ref{prop-compatible}
that every crossed-product norm $\png$ for $G$ fits
into a compatible triple $(\png,\pn,\unh)$. Note that $\unh$ is the unique crossed-product norm for $H$.
Is it true in general that every given norm $\png$ for $G$ (besides the universal or reduced norm) fits into a
compatible triple $(\png,\pn,\pnh)$ for suitable norms $\pn$ and $\pnh$ for $G\times H$ and $H$, respectively?
\\
(2) In the light of Corollary~\ref{cor-subgroup-triple}, suppose
$H$ is a closed subgroup of $G$, $\png$ is a norm for $G$  associated to a $G$\nb-invariant weak-* closed
ideal $E\subseteq B(G)$ and $\pnh$ is the
norm for $H$ corresponding to the ideal $E_H$ generated by $E|_H$ in $B(H)$.
Does the pair $(\png,\pnh)$ always fit into a compatible triple $(\png,\pn,\pnh)$?
\end{question}

\section{Functoriality in correspondence categories}\label{sec:functoriality}

Throughout this section, we fix a locally compact group $G$ and a $G$\nb-invariant {nonzero ideal} $\En$ of $B(G)$ and we denote by $\Enorm$ the corresponding crossed-product norm on $\contc(G,A)$ for any $G$\nb-algebra $A$.

We are going to prove in this section that our constructions are functorial. More precisely, we show that the assignment $A\mapsto A^G_\Enorm$ is a functor between appropriate categories of \cstar{}algebras with correspondences as their morphisms. For categories of \cstar{}algebras based on \Star{}homomorphisms as their morphisms, the functoriality of fixed-point algebras has been proved already in \cite{Buss-Echterhoff:Exotic_GFPA}: in this case, one obtains functoriality in complete generality (the actions here might be non-saturated). On the other hand, since functors preserve invertible morphisms, which in correspondence categories are imprimitivity bimodules, Example~\ref{ex-not-equivalent} shows that we cannot expect functoriality of the construction $A\mapsto A^G_\Enorm$ between correspondence categories except if we restrict to saturated actions.
But we show below that the construction $A\mapsto A^G_\Enorm$ is functorial with respect to (isomorphism classes of) equivariant correspondences if we restrict to the subcategory whose objects are weakly proper saturated actions. This includes (but is not limited to) the
important case where $G$ acts freely on the underlying proper space $X$.

We first need to fix some notation. Given two \cstar{}algebras $A$ and $B$, by a \emph{correspondence} from $A$ to $B$ we mean a (right) Hilbert $B$-module $\E$ together with a nondegenerate \Star{}homomorphism $\phi\colon A\to \Lb(\E)$. We usually use $\phi$ to view $\E$ as a left $A$-module and write $a\cdot \xi\defeq \phi(a)\xi$ for all $a\in A$ and $\xi\in \E$.
The \emph{correspondence category} $\CCCorr$ has \cstar{}algebras as its objects and isomorphism classes, denoted $[\E]$, of $A-B$-correspondences $\E$ as the set of morphisms from $A$ to $B$. The composition of two morphisms $[\D]\colon A\rightsquigarrow B$ and $[\E]\colon B\rightsquigarrow C$ is given by the usual balanced tensor product of correspondences: $[\E]\circ[\D]\defeq [\D\otimes_B\E]$.

For us it will also be important to consider equivariant versions of correspondence categories. For a fixed locally compact group $G$ and a $G$\nb-space $X$, we write $\Corr{X}{G}$ for the category whose objects are weak $X\rtimes G$-\cstar{}algebras, \ie, $G$\nb-algebras $A$ endowed with $X\rtimes G$\nb-equivariant nondegenerate \Star{}homomorphism of $\contz(X)$ into $\M(A)$,  and whose morphisms $A\rightsquigarrow B$ are isomorphisms classes of \emph{$G$\nb-equivariant correspondences}. This means correspondences $\E$ from $A$ to $B$ endowed with $G$\nb-actions compatible with the $G$\nb-actions on $A$ and $B$ (in the usual sense; see \cite{Echterhoff-Kaliszewski-Quigg-Raeburn:Categorical}). Notice that in this case the structural $\contz(X)$-homomorphisms into $\M(A)$ and $\M(B)$ induce left and right module $\contz(X)$-actions on $\E$ by $f\cdot \xi\defeq \phi_A(f)\xi$ and $\xi\cdot \phi_B(f)$ for all $f\in \contz(X)$ and $\xi\in \E$, but we do not require any compatibility conditions between these two actions. The left action $f\cdot \xi= \phi_A(f)\xi$ gives $\E$ the structure of a weak $(B,X\rtimes G)$-module, which then also induces a structure of a weak $(X\rtimes G)$-\cstar{}algebra on $\K(\E)$. In what follows, we always equip $\E$ and $\K(\E)$ with these structures. If the $G$-action on $X$ is proper, $\E$ is therefore a weakly proper $(B,X\rtimes G)$-module and we can apply the theory developed in the previous sections.

Notice that with these settings, $\Corr{X}{G}$ is a full subcategory of the correspondence
category $\CorrG$ which has $G$-\cstar{}algebras as objects and isomorphism classes of $G$\nb-equivariant
correspondences as morphisms.

Assume now that the $G$-space $X$ is proper. Let $\E$ be a $G$\nb-equivariant correspondence from $A$ to $B$ (two given weak $X\rtimes G$-algebras). Forgetting the left $A$-action, we may  view $\E$ as a partial $\K(\E)-B$ equivalence bimodule and therefore consider the $\Enorm$-fixed-point bimodule $\E^G_\Enorm$,
which is a partial equivalence bimodule from $\K(\E)^G_\Enorm$ to $B^G_\Enorm$ (see Lemma~\ref{lem-XG-bimodule}). In particular, $\K(\E^G_\Enorm)$ is an ideal in $\K(\E)^G_\Enorm$ and hence we get a canonical
\Star{}homomorphism $\M(\K(\E)^G_\Enorm)\to \M(\K(\E^G_\Enorm))$ (we will shortly assume that all actions are saturated;
in this case $\K(\E^G_\Enorm)$ is actually isomorphic to $\K(\E)^G_\Enorm$, but we do not need this here).

On the other hand, the nondegenerate $G$-equivariant \Star{}homomorphism $A\to \Lb(\E)=\M(\K(\E))$ induces, by \cite[Proposition~6.2]{Buss-Echterhoff:Exotic_GFPA}, a nondegenerate \Star{}homomorphism $A^G_\Enorm\to \M(\K(\E)^G_\Enorm)$. Composing this with $\M(\K(\E)^G_\Enorm)\to \M(\K(\E^G_\Enorm))$, we therefore get a  \Star{}homomorphism $A^G_\Enorm\to \Lb(\E^G_\Enorm)$ which is easily seen (using the construction of $\E_\pn^G$ as in Remark~\ref{rem:AlternativeDescriptionFixedModule}) to be given by $\EE(a)\cdot x=\EE(a\cdot x)$ for all $a\in A_c$ and $x\in \E_c^G$. Therefore $\E^G_\Enorm$ has a canonical structure of a correspondence from $A^G_\Enorm$ to $B^G_\Enorm$. Since both fixed-point algebras carry canonical structures as weak $G\bs X$-algebras via $\EE(f)\cdot a=\EE(f\cdot a)$ for $f\in C_c(X)$ and $a\in A_c^G$ (and similarly for $B_\Enorm^G$), we obtain a morphism $[\E^G_\Enorm]\colon A^G_\Enorm\rightsquigarrow B^G_\Enorm$ in $\CCorr{G\bs X}$.  It is clear that the construction $\E\mapsto \E^G_\Enorm$ descends to isomorphism classes and hence gives a well-defined assignment between the categories $\Corr{X}{G}$ and $\CCorr{G\bs X}$ which on the level of objects maps $A$ to $A^G_\Enorm$ and on the level of morphisms maps $[\E]$ to $[\E^G_\Enorm]$.

Our next goal is to show that these assignments determine a functor $\Fix^G_\Enorm$ if we restrict to the full subcategory $\Corr{X}{G}_{\sat}^\Enorm$ consisting of $\Enorm$-saturated objects in $\Corr{X}{G}$, \ie, the objects in $\Corr{X}{G}_{\sat}^\Enorm$ are weak $X\rtimes G$-algebras whose action is $\Enorm$-saturated (and the morphisms between two such objects are the same as in $\Corr{X}{G}$). Notice that, by Lemma~\ref{lem-univ-sat}, if the $G$\nb-action on $X$ is free, then $\Corr{X}{G}_{\sat}^\Enorm=\Corr{X}{G}$, but in general $\Corr{X}{G}$ is strictly bigger. As already noticed at the beginning of this section, $\Fix^G_\Enorm$ cannot be a functor from $\Corr{X}{G}$ to $\CCorr{G\bs X}$ in general. The main point is the following result:

\begin{proposition}\label{prop:TensorOfFixedModules}
Let $A,B,C$ be weak $X\rtimes G$-algebras and
let $\D$ be an $X\rtimes G$-equivariant correspondence from $A$ to $B$ and let $\E$ be
a $G$-equivariant correspondence from $B$ to $C$.
Assume that the action on $B$ is $\Enorm$-saturated. Then the canonical embedding $\D^G_c\odot\E^G_c\into (\D\otimes_B\E)^G_c$ extends to an isomorphism
$$\D^G_\Enorm\otimes_{B^G_\Enorm}\E^G_\Enorm\congto (\D\otimes_B\E)^G_\Enorm$$
of correspondences from $A^G_\Enorm$ to $C^G_\Enorm$.
\end{proposition}
\begin{proof}
By Corollary~\ref{cor-full}, we have canonical isomorphisms
$$\D^G_\Enorm\cong \F^G_\Enorm(\D)\otimes_{B\rtimes_\Enorm G}\F^G_\Enorm(B)^* \quad\mbox{and}\quad \E^G_\Enorm\cong \F^G_\Enorm(\E)\otimes_{C\rtimes_\Enorm G}\F^G_\Enorm(C)^*$$
mapping $\EE(\xi\cdot b^*)\mapsto \xi\otimes b$ and $\EE(\eta\cdot c^*)\mapsto \eta\otimes c$, respectively,
for all $\xi\in \contc(X)\cdot \D$, $b\in \contc(X)\cdot B$, $\eta\in \contc(X)\cdot \E$ and $c\in \contc(X)\cdot C$.
Since $B\to \Lb(\E)$ is nondegenerate, we have a canonical isomorphism $B\otimes_B\E\cong \E$, $a\otimes \eta\mapsto a\cdot \eta$. Hence,
Proposition~\ref{prop:TensorDecompositionFixedPointModules} yields an isomorphism
$$\F^G_\Enorm(B)\otimes_{B\rtimes_\Enorm G}\E\rtimes_\Enorm G\cong \F^G_\Enorm(\E)$$
which identifies $a\otimes f\in \F_c(B)\odot\contc(G,\E)$ with $a*f=\int_G\Delta(t)^{-1/2}\gamma_t(a\cdot f(t^{-1}))dt$.
Since the $G$\nb-action on $B$ is $\Enorm$-saturated, we also have an isomorphism
$$\F^G_\Enorm(B)^*\otimes_{B^G_\Enorm}\F^G_\Enorm(B)\cong B^G_\Enorm,\quad b\otimes a\mapsto \bbraket{b}{a}_{B\rtimes_\Enorm G}.$$
Also, if $\theta$ and $\gamma$ denote the $G$\nb-actions on $\D$ and $\E$, respectively, then, by Proposition~\ref{prop:TensorDecompositionFixedPointModules}, we have an isomorphism
$$\F^G_\Enorm(\D)\otimes_{B\rtimes_\Enorm G}\E\rtimes_\Enorm G\congto \F^G_\Enorm(\D\otimes_B\E)$$
sending $\xi\otimes g\in \F_c(\D)\odot \contc(G,\E)$ to $\xi*g=\int_G \Delta(t)^{-1/2}\theta_t(\xi)\otimes \gamma_t(g(t^{-1})) dt$. Coupling these isomorphisms
and also the canonical isomorphism $B\rtimes_\Enorm G\otimes_{B\rtimes_\Enorm G}\E\rtimes_\Enorm G\cong \E\rtimes_\Enorm G$
which sends $\varphi\otimes f$ to $\varphi* f|_t=\int_G\varphi(s)\cdot \gamma_s(f(s^{-1}t)) ds$ (see~\eqref{eq-left-action}), we obtain the following chain of isomorphisms:
\begin{equation}\label{eq:isomorphismsFixedModules}
\begin{split}
\D^G_\Enorm &\otimes_{B^G_\Enorm}\E^G_\Enorm\\
   & \cong \F^G_\Enorm(\D)\otimes_{B\rtimes_\Enorm G}\F^G_\Enorm(B)^*\otimes_{B^G_\Enorm}\F^G_\Enorm(\E)\otimes_{C\rtimes_\Enorm G}\F^G_\Enorm(C)^*\\
    &\cong \F^G_\Enorm(\D)\otimes_{B\rtimes_\Enorm G}\F^G_\Enorm(B)^*\otimes_{B^G_\Enorm}\F^G_\Enorm(B)\otimes_{B\rtimes_\Enorm G}\E\rtimes_\Enorm G\otimes_{C\rtimes_\Enorm G}\F^G_\Enorm(C)^*\\
    &\cong \F^G_\Enorm(\D)\otimes_{B\rtimes_\Enorm G}B\rtimes_\Enorm G\otimes_{B\rtimes_\Enorm G}\E\rtimes_\Enorm G\otimes_{C\rtimes_\Enorm G}\F^G_\Enorm(C)^* \\
    &\cong \F^G_\Enorm(\D)\otimes_{B\rtimes_\Enorm G}\E\rtimes_\Enorm G\otimes_{C\rtimes_\Enorm G}\F^G_\Enorm(C)^*\\
    &\cong \F^G_\Enorm(\D\otimes_B\E)\otimes_{C\rtimes_\Enorm G}\F^G_\Enorm(C)^*\\
    &\cong (\D\otimes\E)^G_\Enorm,
\end{split}
\end{equation}
The last isomorphism follows from Corollary~\ref{cor-full}.

We now show that the composition of the above isomorphisms extends the canonical embedding
$\D^G_c\odot\E^G_c\into (\D\otimes_B\E)^G_c$ given by the identity on elementary tensors.
To see this let $x\in \D^G_c$ be of the form $x=\EE(\xi\cdot b^*)$ with $\xi\in \F_c(\D)=C_c(X)\cdot\D$ and $b\in B_c$ and let $y\in \E^G_c$ be of the form
$y=\EE(\eta\cdot c^*)$ with $\eta\in \F_c(\E)=C_c(X)\cdot\E$ and $c\in C_c$.  Notice that such elements form dense subspaces of $\D^G_c$ and $\E^G_c$, respectively.
Furthermore, we may assume that $\eta=a*f$ with $a\in B_c$ and $f\in \contc(G,\E)$ since these form a dense subspace of
$\F_c(\E)$, which follows from Proposition~\ref{prop:TensorDecompositionFixedPointModules} applied to the decomposition $\E\cong B\otimes_B\E$.
Now, starting with $x\otimes y=\EE(\xi\cdot b^*)\otimes \EE(\eta\cdot c^*)\in\D^G_\Enorm \otimes_{B^G_\Enorm}\E^G_\Enorm$ we follow the isomorphisms in~\eqref{eq:isomorphismsFixedModules}
writing $\cong$ to identify the corresponding elements:
\begin{align*}
x\otimes y&=\EE(\xi\cdot b^*)\otimes \EE(\eta\cdot c^*)\\
        &\cong \xi\otimes b\otimes a*f\otimes c\\
        &\cong \xi\otimes b\otimes a\otimes f\otimes c\\
        &\cong \xi\otimes \bbraket{b}{a}_{B\rtimes_\pn G}\otimes f\otimes c\\
        &\cong \xi\otimes \bbraket{b}{a}_{B\rtimes_\pn G}* f\otimes c\\
        &\cong \xi*(\bbraket{b}{a}* f)\otimes c\\
        &\cong \EE\big((\xi*(\bbraket{b}{a}_{B\rtimes_\pn G}* f))\cdot c^*\big).
\end{align*}
Now we compute
\begin{align*}
\xi*(\bbraket{b}{a}&_{B\rtimes_\pn G}* f)=\int_G\Delta(t)^{-1/2}\tilde\gamma_t(\xi\otimes_B\bbraket{b}{a}_{B\rtimes_\pn G}*f(t^{-1})) dt\\
       &=\int_G\Delta(t)^{-1/2}\theta_t(\xi)\otimes_B\gamma_t\left(\int_G\bbraket{b}{a}_{B\rtimes_\pn G}(s)\gamma_s(f(s^{-1}t^{-1}))ds\right) dt\\
        &=\int_G\int_G\Delta(t)^{-1/2}\Delta(s)^{-1/2}\theta_t(\xi)\otimes_B\beta_t(b^*)\beta_{ts}(a)\gamma_{ts}(f(s^{-1}t^{-1}) ds dt\\
        &\stackrel{s\mapsto t^{-1}s}{=}\int_G\int_G\Delta(s)^{-1/2}\theta_t(\xi)\otimes_B\beta_t(b^*)\beta_s(a)\gamma_s(f(s^{-1})) ds dt\\
        &=\int_G\theta_t(\xi\cdot b^*)\otimes_B\int_G\Delta(s)^{-1/2}\gamma_s(a\cdot f(s^{-1})) ds\\
        &=\EE(\xi\cdot b^*)\otimes_B a*f=\EE(\xi\cdot b^*)\otimes_B \eta.
\end{align*}
It follows that $\EE\big((\xi*(\bbraket{b}{a}_{B\rtimes_\pn G}\cdot f))\cdot c^*\big)=\EE(\xi\cdot b^*)\otimes \EE(\eta\cdot c^*)=x\otimes y$.
This proves that the composition of isomorphisms in~\eqref{eq:isomorphismsFixedModules} is the canonical embedding, as desired.
It is clear that this canonical isomorphism preserves the left
$A^G_\Enorm$-actions as well as the right $C^G_\Enorm$-actions and the $C^G_\Enorm$-valued inner products (all
the isomorphisms involved preserve this right structure). Therefore $\D^G_\Enorm\otimes_{B^G_\Enorm} \E^G_\Enorm\cong (\D\otimes_B\E)^G_\Enorm$
as $A^G_\Enorm-C^G_\Enorm$ correspondences.
\end{proof}

We are now ready to prove one of the main results of this section:

\begin{theorem}\label{thm-Fix-functor}
Let $G$ be a locally compact group and let $X$ be a proper $G$\nb-space. Then
the assignment mapping a weak $X\rtimes G$-algebra $A$ to $A^G_\Enorm$ and the isomorphism class of an
$X\rtimes G$-equivariant correspondence $[\E]$ to $[\E^G_\Enorm]$ defines a functor $\Fix^G_\Enorm$
from $\Corr{X}{G}_{\sat}^\Enorm$ to $\CCorr{G\bs X}$.
\end{theorem}
\begin{proof}
Since all objects in $\Corr{X}{G}_{\sat}^\Enorm$ are, by definition, $\Enorm$-saturated weak $X\rtimes G$-algebras,
Proposition~\ref{prop:TensorOfFixedModules} says that
$$[\E^G_\Enorm]\circ[\D^G_\Enorm]=[\D^G_\Enorm\otimes_{B^G_\Enorm}\E^G_\Enorm]=[(\D\otimes_B\E)^G_\Enorm]$$
for every pair of composable morphisms $[\D]$ and $[\E]$ in $\Corr{X}{G}$.
This means that $\Fix^G_\Enorm\colon \Corr{X}{G}_{\sat}^\Enorm\to\CCorr{G\bs X}$, $[\E]\mapsto [\E^G_\Enorm]$ preserves composition of morphisms.
Since the identity morphism $[A]\colon A\to A$ in $\Corr{X}{G}$ for an object $A$ in $\Corr{X}{G}$
is obviously sent to the identity morphism $[A^G_\Enorm]\colon A^G_\Enorm \to A^G_\Enorm$
in $\CCorr{G\bs X}$, it follows that $\Fix^G_\Enorm$ is a functor $\Corr{X}{G}_{\sat}^\Enorm\to\CCorr{G\bs X}$.
\end{proof}

For a free and proper $G$\nb-space $X$, every weak $X\rtimes G$-algebra is $\Enorm$-saturated by Lemma~\ref{lem-univ-sat}, so we immediately get the following  consequence:

\begin{corollary}
If $X$ is a free and proper $G$\nb-space, then the assignments $A\mapsto A^G_\Enorm$, $[\E]\mapsto [\E^G_\Enorm]$
define a functor $\Fix^G_\Enorm\colon \Corr{X}{G}\to \CCorr{G\bs X}$.
\end{corollary}

We should notice that for reduced generalized fixed-point algebras, the above corollary
has been proved by an Huef, Raeburn and Williams in \cite[Theorem~13]{Huef-Raeburn-Williams:FunctorialityGFPA}.

Next, we interpret the construction $A\mapsto \F_\Enorm(A)$ as a natural isomorphism in the  following sense (similar to what has been done for reduced generalized fixed-point algebras and \emph{free} actions on spaces in \cite{Kaliszewski-Quigg-Raeburn:ProperActionsDuality}).
The result  generalises the main result of \cite{Huef-Kaliszewski-Raeburn-Williams:Naturality-Rieffel} where the authors
obtain a reduced version under the additional assumption that $G$ acts freely on $X$:

\begin{proposition}\label{prop-functor}
If $G$ acts properly on a space $X$, then the assignment that maps a $\Enorm$-saturated weak $X\rtimes G$-algebra $A$ to the $A^G_\Enorm-A\rtimes_\Enorm G$ equivalence bimodule
$\F^G_\Enorm(A)$ is a natural isomorphism between the fixed functor $\Fix^G_\Enorm\colon A\mapsto A^G_\Enorm$
and the crossed product functor $\CP_\Enorm\colon A\mapsto A\rtimes_\Enorm G$ considered as functors from $\Corr{X}{G}_{\sat}^\Enorm$ to $\CCorr{G\bs X}$. In particular, if $G$ acts freely and properly on $X$, then $A\mapsto \F_\Enorm(A)$ is a natural isomorphism between $\Fix^G_\Enorm$ and $\CP_\Enorm$ as functors $\Corr{X}{G}\to\CCorr{G\bs X}$.
\end{proposition}
\begin{proof}
Recall that on the level of objects, the crossed product functor $\CP_\Enorm$ sends a weak $X\rtimes G$-algebra $A$ to the crossed product $A\rtimes_\Enorm G$,
viewed as a weak $G\bs X$-algebra with respect to the canonical homomorphism $\contz(G\bs X)\to \M(A\rtimes_\Enorm G)$ induced by $\contz(X)\to\M(A)$. On the level of morphisms it sends
a $G$-equivariant $A-B$ correspondence $[\E]\colon A\rightsquigarrow B$ to $[\E\rtimes_\Enorm G]$,
(see Section~\ref{sec:CrossedProductsHilbertModules} for the precise formulas).

All we have to prove is naturality of $A\mapsto \F^G_\Enorm(A)$, \ie, given two $\Enorm$-saturated weak $X\rtimes G$-algebras $A$ and $B$
and a morphism $[\E]\colon A\rightsquigarrow B$, we have to show that the
following diagram of morphisms in $\CCorr{G\bs X}$ commutes:
$$
\begin{CD}
A^G_\Enorm @>\F^G_\Enorm(A)>> A\rtimes_\Enorm G\\
@V\E^G_\Enorm VV  @VV\E\rtimes_\Enorm G V\\
B^G_\Enorm
@> \F^G_\Enorm(B)  >> B\rtimes_\Enorm G.
\end{CD}
$$
In other words, we need to show that $\F^G_\Enorm(A)\otimes_{A\rtimes_\Enorm G}\E\rtimes_\Enorm G\cong \E^G_\Enorm\otimes_{B^G_\Enorm}\F^G_\Enorm(B)$ as $G\bs X$-equivariant $A^G-B\rtimes_\Enorm G$ correspondences.
we have a canonical isomorphism $\F^G_\Enorm(A)\otimes_{A\rtimes_\Enorm G}\E\rtimes_\Enorm G\cong \F^G_\Enorm(\E)$ of Hilbert $B\rtimes_\Enorm G$-modules
Hence it is  an isomorphism of $G\bs X$-equivariant $A^G_\Enorm-B\rtimes_\Enorm G$ correspondences.
On the other hand, by Corollary~\ref{cor-full}, we have a canonical isomorphisms $\E^G_\Enorm\cong \F^G_\Enorm(\E)\otimes_{B\rtimes_\Enorm G}\F^G_\Enorm(B)^*$.  Together with the canonical isomorphism $\F^G_\Enorm(B)^*\otimes_{B^G_\Enorm}\F^G_\Enorm(B)\cong B\rtimes_\Enorm G$ (here we have used that $B$ is $\Enorm$-saturated) this yields an isomorphism $\E^G_\Enorm\otimes_{B^G_\Enorm}\F^G_\Enorm(B)\cong \F^G_\Enorm(\E)\otimes_{B\rtimes_\Enorm G}\F^G_\Enorm(B)^*\otimes_{B^G_\Enorm}\F^G_\Enorm(B)\cong \F^G_\Enorm(\E)$ of Hilbert $B\rtimes_\Enorm G$-modules.  It is easy to check that this isomorphism preserves
the left $A_\pn^G$-actions. This yields the desired result.
\end{proof}
\begin{remark}\label{rem-universal-cat} Recall that a proper $G$-space $Z$ is called universal, if
for every proper $G$-space $X$ there exists a continuous $G$-map $\varphi:X\to Z$
which is unique up to $G$-homotopy. It is shown in \cite{Kasparov-Skandalis:Bolic}
that such universal proper $G$-space $Z$ always exists, and is unique up to $G$-homotopy.
Let us fix a universal proper $G$-space which, following the usual convention, we call
$\underline{EG}$.

As observed in Remark \ref{rem-unique}, we can always replace a given proper $G$\nb-space $X$ by $\underline{EG}$ by composing
the structure map $\phi:C_0(X)\to \M(A)$ with  the \Star{}homomorphism $\varphi^*:C_0(\underline{EG})\to \M(C_0(X))$ given by $\varphi^*(f)=f\circ \varphi$ for a continuous $G$-map $\varphi:X\to\underline{EG}$. It follows from Proposition \ref{prop-independence} that the generalized fixed-point algebras $A_\pn^G$ remain unchanged. Thus, if $X$ and $Y$ are two proper
$G$-spaces, $(A,\alpha)$ is a saturated weakly proper $X\rtimes G$-algebra, and $(B,\beta)$ is a
saturated weakly proper $Y\rtimes G$-algebra, we may regard both systems $(A,\alpha)$ and $(B,\beta)$ as objects in $\Corr{\underline{EG}}{G}_{\sat}^\pn$. We therefore obtain
a universal theory in which the objects are the saturated weakly proper $X\rtimes G$-algebras {\bf for all proper $G$-spaces $X$}. Moreover, if $(A,\alpha)$ and $(B,\beta)$ are two such objects,
the morphisms from $(A,\alpha)$ to $(B,\beta)$ are just the
isomorphism classes of $G$\nb-equivariant correspondences $[\E]\colon A\rightsquigarrow B$.
\end{remark}

We finish by deducing a symmetric version of the above result.  Suppose that $G$ and $H$ are locally compact groups and
that $(\png,\pn,\pnh)$ is a compatible triple of norms as in Definition \ref{def-compatible-action}. Assume that $X$ is a $G\times H$-space
such that $G$ and $H$ act properly on $X$. In this situation we write $\Corr{X}{G\times H}_{\sat}^\pn$ for the category of weak
$X\rtimes (G\times H)$-algebras $(A,\alpha)$ such that the both underlying weakly proper $X\rtimes G$- and $X\rtimes H$-structures
$(A,\alpha^G)$ and $(A,\alpha^H)$ are saturated. If $(\mathcal E,\gamma)$ is a correspondence from $(A,\alpha)$ to $(B,\beta)$
in $\Corr{X}{G\times H}_{\sat}^\pn$, it follows from a combination of Lemma~\ref{lem-XG-bimodule} with Lemma~\ref{lem-actionH} (the latter applied
to $A$, $B$ and the linking algebra $L(\E)$) that
the restriction of $\gamma^H$ to $\E_c^G\subseteq \E$ extends to an $H$\nb-action
on the $A_\png^G-B_\png^G$-correspondence $\E_\png^G$.  By considering diagonal actions on  tensor products of correspondences,
it is straightforward to obtain an equivariant version of Proposition~\ref{prop:TensorOfFixedModules} and Theorem~\ref{thm-Fix-functor}
 to obtain the following result, in which  $\CorrG$ (resp. $\CorrH$) denotes the correspondence category of \cstar{}dynamical
 $G$\nb-systems (resp. $H$\nb-systems)
 as in \cite[Chapter 2.2]{Echterhoff-Kaliszewski-Quigg-Raeburn:Naturality}:

\begin{proposition}\label{prop-symmetric-Fix-functor}
Let $(\png,\pn,\pnh)$ be a compatible triple of crossed-product norms as in Definition \ref{def-compatible-action}.
Then there are functors
\begin{align*}
&\Fix_\png^G(H):\Corr{X}{G\times H}_{\sat}^\pn\to \CorrH\quad \text{and}\\
&\Fix_\png^H(G):\Corr{X}{G\times H}_{\sat}^\pn\to \CorrG
\end{align*}
in which $\Fix_{\pn^G}^G(H)$ sends an object $(A,\alpha)\in \Corr{X}{G\times H}_{\sat}^\pn$ to the object \linebreak
$(A_\png^G,\alpha^H)\in \CorrH$
and an isomorphism class $[(\E,\gamma)]$ of an $(A,\alpha) - (B,\beta)$ correspondence $(\E,\gamma)$ to the isomorphism class of the
$(A_\png^G, \alpha^H)-(B_\png^G,\beta^H)$ correspondence $(\E_\png^G,\gamma^H)$, and similarly for $\Fix_{\pn^H}^H(G)$.
\end{proposition}

Recall that $\CCCorr $ denotes the correspondence category with \cstar{}algebras as objects.
In what follows, we let $\CP_\png: \CorrG\to \CCCorr$, $A\mapsto A\rtimes_\png G$
(resp. $\CP_\pnh: \CorrH\to \CCCorr$, $A\mapsto A\rtimes_\pnh H$) denote the $\png$\nb-(resp. $\pnh$-)crossed product functor  for
$G$ (resp. $H$). Moreover, we let
$\CP_\png^H$ and $\CP_\pnh^G$  denote the functors from $\Corr{X}{G\times H}$ to
$\CorrH$ and $\CorrG$, respectively, which send  a system $(A,\alpha)$ to $(A\rtimes_{\alpha^G,\png} G, \alpha^H)$ and
$(A\rtimes_{\alpha^H,\pnh} H, \alpha^G)$, respectively.
We are now ready to combine the results of this section with the symmetric imprimitivity theorem to obtain the
following generalization of the main result  of \cite{Huef-Kaliszewski-Raeburn-Williams:Naturality-Symmetric}:

\begin{theorem}\label{thm-symmetric-naturality}
Assume that $X$ is a $G\times H$-space such that $G$ and $H$ act properly on $X$ and let $(\png,\pn,\pnh)$ be a
compatible triple of norms as in Definition \ref{def-compatible-action}. Then the
$A_\png^G\rtimes_{\pnh}H-A_\pnh^H\rtimes_{\png}G$ equivalence bimodule $\F_\pn(A)_G^H$ of Theorem  \ref{thm-symmetric-direct}
is a natural isomorphism between the functors $\CP_\pnh\circ \Fix_\png^G(H)$ and $\CP_\png\circ \Fix_\png^H(G)$.
Moreover, both functors are naturally isomorphic to the crossed-product functor
$$\CP_\pn: \Corr{X}{G\times H}_{\sat}^\pn\to \CCCorr\; ;\; (A,\alpha)\mapsto A\rtimes_\pn(G\times H)$$
via the equivalence bimodules $\F_\png^G(A)\rtimes_\pnh H$ and $\F_\pnh^H(A)\rtimes_\png G$, respectively.
\end{theorem}
\begin{proof} Note that the canonical isomorphisms
$$(A\rtimes_{\alpha^G,\png}G)\rtimes_{\alpha^H,\pnh} H\cong A\rtimes_{\alpha,\pn}(G\times H)\cong
(A\rtimes_{\alpha^H,\pnh}H)\rtimes_{\alpha^G,\png} G,$$
which exist by the compatibility assumption for $(\png,\pn,\pnh)$, give natural isomorphisms between the functors
$\CP_\pnh\circ \CP_\png^H$, $\CP_\pn$, and $\CP_\png\circ \CP_\pnh^G$, respectively.
Moreover, by the above discussion we obtain an $H$\nb-equivariant version
of Proposition~\ref{prop-functor} which shows that $(A,\alpha)\mapsto (\F_\png^G(A), \tilde{\alpha}^H)$ (in which $\tilde\alpha^H$ denotes the
action of $H$ on $\F_\png^G(A)$ induced from $\alpha$ as in Lemma~\ref{lem-actionH}) is
a natural isomorphism between the functor $\Fix_\png^G(H)$ and $\CP_\png^H$.
 Taking crossed products by $H$, we therefore get the
natural isomorphism $(A,\alpha)\mapsto \F_\png^G(A)\rtimes_{\pnh}H$ between $\CP_\pnh\circ \Fix_\png^G(H)$ and
$\CP_\pnh\circ \CP_\png^H\cong \CP_\pn$.
Similarly $(A,\alpha)\mapsto \F_\pnh^H(A)\rtimes_{\png}G$ gives a natural isomorphism
between the functors $\CP_\png\circ \Fix_\pnh^H(G)$ and
$\CP_\png\circ \CP_\pnh^G\cong \CP_\pn$. By  the proof of Theorem~\ref{thm-symmetric-direct} we have an isomorphism
$\F_\pn(A)_G^H\cong \left(\F_\pn^G(A)\rtimes_\pn H\right)\otimes_{A\rtimes_\pn(G\times H)}\left(\F_\pn^H(A)\rtimes_\pn G\right)^*$
and the result follows.
\end{proof}


\end{document}